\documentclass{amsart}

\usepackage{amsfonts}
\usepackage{amssymb}
\usepackage{amstext}

\usepackage{graphicx}
\usepackage{tikz}
\usepackage{caption}
\usepackage{enumitem}
\usepackage[]{amscd}
\usepackage{hyperref}
\newtheorem{theorem}{Theorem}[section]
\newtheorem{lemma}{Lemma}[section]
\newtheorem{corollary}{Corollary}[section]
\newtheorem{proposition}[theorem]{Proposition}

\theoremstyle{definition}
\newtheorem{definition}{Definition}[section]
\theoremstyle{remark}
\newtheorem{remark}{Remark}[section]
\title[Elliptic fibrations]{Some observations about isogenies between $K3$ surfaces }

\usepackage{amsmath}

\begin{document}

\author[M.-J. BERTIN]{Marie Jos\'e BERTIN }


\curraddr{Sorbonne Universit\'e \\Institut
de Math\'ematiques de Jussieu-Paris Rive Gauche \\ Case 247\\ 4 Place Jussieu, 75252 PARIS, Cedex 85, France}

\email{marie-jose.bertin@imj-prg.fr}

\author[O. LECACHEUX]{Odile LECACHEUX}


\curraddr{Sorbonne Universit\'e\\ Institut
de Math\'ematiques de Jussieu-Paris Rive Gauche \\Case 247 \\4 Place Jussieu, 75252 PARIS, Cedex 85, France}

\email{odile.lecacheux@imj-prg.fr}

\date{\today}

\begin{abstract} Even if there are too many elliptic fibrations to investigate and describe on the singular $K3$ surface $Y_{10}$ of discriminant $72$ and belonging to the Ap\'ery-Fermi pencil ($Y_k$), we find on it many interesting properties. For example some of its elliptic fibrations with $3$-torsion section induce by $3$-isogeny either an elliptic fibration of $Y_2$, the unique $K3$ surface of discriminant $8$, or an elliptic fibration of other $K3$ surfaces of discriminant $72$.

\end{abstract}

\maketitle

\section{Introduction}
Bertin's results \cite{B1}, \cite{B2} on Mahler measure of $K3$  surfaces suggest a link between the two $K3$ surfaces $Y_2$ and $Y_{10}$ of the Ap\'ery-Fermi pencil $Y_k$ defined by the equation

\[(Y_k) \qquad X+\frac{1}{X}+Y+\frac{1}{Y}+Z+\frac{1}{Z}=k.\]

In particuliar, it was proved that the Mahler measures of the polynomials defining $Y_2$ and $Y_{10}$ are expressed with the $L$-series of their transcendental lattices $T(Y_2)$ and $T(Y_{10})$ with
\[T(Y_2)=\begin{pmatrix}
	2 & 0\\
	0 & 4\\
\end{pmatrix} \qquad T(Y_{10})=\begin{pmatrix}
	6 & 0\\
	0 & 12\\
\end{pmatrix} \]
and that their transcendental $L$-series are the same.

These $K3$ surfaces $Y_2$ and $Y_{10}$ belong to the same class in Sch\"{u}tt's classification \cite{Sc}, since the discriminant of $Y_2$ is $8$ while the discriminant of $Y_{10}$ is $3^2\times 8$ and the newform in the two transcendental $L$-series are the same up to twist.

Even more, from
results of Kuwata \cite{Kuw}  and Shioda \cite{Sh} (theorem 2.1), the equality $T\left(  Y_{10}\right)  =T\left(  Y_{2}\right)
\left[  3\right]  $ , reveals a relation between $Y_{2}$ and $Y_{10}.$ Indeed,
starting with an elliptic fibration of $Y_2$ with two singular fibers $II^*$ and Weierstrass equation
\[(Y_2)_h\qquad y^2=x^3+\alpha x+h+\frac{1}{h}+\beta,\]
a base change $h=u^3$ gives a Weierstrass equation denoted $(Y_2)_h^{(3)}$ of an elliptic fibration of a $K3$ surface with transcendental lattice $T(Y_2)[3]$, which is precisely $Y_{10}.$ If, instead of the previous base change, we use the base change $h=u^2$, we obtain a Weierstrass equation  $(Y_2)_h^{(2)}$ of an elliptic fibration of a $K3$ surface with transcendental lattice $T(Y_2)[2]$ which is precisely the Kummer surface $K_2$. The idea, previously developped in \cite{B-L} when searching $2$-isogenies between some elliptic fibrations of $Y_2$ and its Kummer $K_2$, suggests possible $3$-isogenies between some elliptic fibrations of $Y_2$  and $Y_{10}$.

To answer this question it was tempting to study $Y_{10}$ as Bertin and Lecacheux did for $Y_2$ \cite{BL}, that is determine all the non isomorphic elliptic fibrations of $Y_{10}$ together with their Weierstrass equations. However such a study is more tricky and leads only to partial results.

The first difficulty is the application of the Kneser-Nishiyama method to determine all the elliptic fibrations. Concerning $Y_2$, these fibrations are given by primitive embeddings of a root lattice, namely $A_1 \oplus D_5$, in the various Niemeier lattices, which are only primitive embeddings in the root part of Niemeier lattices. 
This is no longer the case for $Y_{10}$ since we have to embed the lattice $M=A_1 \oplus A_2 \oplus N$ with $N= \big (\begin{smallmatrix}
	-2 & 0 & 1\\
	0 & -2 & 1\\
	1 & 1 & -4\\
\end{smallmatrix}\big)$. Since $N$ is not a root lattice, we have to consider primitive embeddings into Niemeier lattices and not only in their root part. This fact added to the facts that $M$ is composed of three irreducible lattices, $A_1$ and $A_2$ embedding primitively in all the other root lattices, give a huge amount of such embeddings thus of elliptic fibrations. This difficulty is explained in section 3.

Such a situation has been encountered by Braun, Kimura and Watari \cite{BKW}. But their examples\
 were simpler since in their case $M=A_5 \oplus (-4)$. 
   And even in that simpler case, probably Braun, Kimura and Watari encountered so many cases that they restricted to primitive embeddings containing $E_6$, $E_7$ or $E_8$.

Moreover, to answer our motivating question, we must know Weierstrass equations of $3$-torsion elliptic fibrations of $Y_{10}$. In \cite{B-L}, Bertin and Lecacheux obtained all elliptic fibrations, called generic, of the Ap\'ery-Fermi pencil together with a Weierstrass equation. However, their specializations, object of section 4, give only a few elliptic fibrations and Weierstrass equations of $Y_{10}$. In particuliar, all the $10$ extremal fibrations of $Y_{10}$ given in Shimada-Zhang \cite{Shim} are missing.

Hence, follows in section 5 the determination of primitive embeddings into Niemeier lattices giving these extremal fibrations. 
Their corresponding Weierstrass equations are also exhibited.

An important property of $Y_{10}$ contrary to $Y_2$ must be underlined. We can find on the $K3$ surface $Y_{10}$ high rank fibrations up to rank $7$. Examples are given in section 6.

In \cite{B-L}, the $2$-isogenies of $Y_2$ are divided in two classes, the Morisson-Nikulin ones, i.e. $2$-isogenies from $Y_2$ to its Kummer $K_2$, and the others, called van Geemen-Sarti involutions. It was also proved that these latter are in fact "self-isogenies", meaning either they preserve the same elliptic fibration ("PF self-isogenies") or they exchange two elliptic fibrations of $Y_2$ ("EF self-isogenies"). In section 7 we ask for a similar result concerning $Y_{10}$, even if we do not know all the $2$-torsion elliptic fibrations. Specialized Morisson-Nikulin involutions remain Morisson-Nikulin involutions and specialized van-Geemen-Sarti involutions are "self-isogenies". But in addition we found a Morisson-Nikulin involution from $Y_{10}$ to its Kummer $K_{10}$ not obtained by specialization.

Finally, in the last section 8, we prove the main result, our first motivation, that is, all the $3$-isogenies from elliptic fibrations of $Y_2$ are $3$-isogenies from $Y_2$ to $Y_{10}$. It remains a natural question: what about the other $3$-isogenies from elliptic fibrations of $Y_{10}$? Indeed we found $3$-isogenies from $Y_{10}$ to two other $K3$ surfaces with respective transcendental lattices $[4 \quad 0 \quad 18]=\begin{pmatrix} 4& 0\\0&18\end{pmatrix}$ and $[2 \quad 0 \quad 36]$.

In the same section we use the elliptic fibration $(Y_2)_h^{(3)}$ to construct elliptic fibrations of $Y_{10}$ of high rank (namely $7$ the highest we found)  and by the $2$-neighbour method a rank $4$ elliptic fibration with a $2$-torsion section
defining the Morisson-Nikulin involution exhibited in section $7$.

This famous Ap\'ery-Fermi pencil is identified by Festi and van Straten \cite{FS} to be a pencil of $K3$ surfaces $\mathcal{D}_s$ appearing in the $2$-loop diagrams in Bhabha scattering. These authors remark that $\mathcal{D}_1$ is a $K3$ surface with transcendental lattice isometric to $\langle 2 \rangle \oplus \langle 4 \rangle $, that is $\mathcal{D}_1$ is $Y_2$. From the relation $k=2-4s$ \cite{FS}, we may also remark that $\mathcal{D}_{-2}$ and $\mathcal{D}_3$ are the $K3$ surface $Y_{10}$ studied in the present paper.

\section{Definitions and results}
We recall briefly what is useful for the understanding of the paper.

 A rank $r$ \textbf{lattice} is a free $\mathbb Z$-module $S$ of rank $r$ together with a symmetric bilinear form $b$.

A lattice $S$ is called {\textbf{even}} if $x^2:=b(x,x)$ is even for all $x$ from $S$.
For any integer $n$ we denote by ${\langle \boldsymbol{n} \rangle}$ the lattice $\mathbb Z e$ where $e^2=n$.

If $e=(e_1,\ldots, e_r)$ is a $\mathbb Z$-basis of a lattice $S$, then the matrix 
$G(e)=(b(e_i,e_j))$ is called the \textbf{Gram matrix} of $S$ with respect to $e$.
An injective homomorphism of lattices is called an embedding.

An embedding $i:S \rightarrow S'$ is called {\textbf{primitive}} if $S'/i(S)$ is a free group.
A sublattice is a subgroup equipped with the induced bilinear form. A sublattice $S'$ of a lattice $S$ is called primitive if the identity map $S' \rightarrow S$ is a primitive embedding.
The \textbf{primitive closure} of $S$ inside $S'$ is defined by
$\overline{S} =\{x\in S' / mx\in S \hbox{  for some positive integer }m \}.$
A lattice $M$ is an \textbf{ overlattice} of $S$ if $S$ is a sublattice of $M$ such that the index $[M:S]$ is finite.

By $S_1\oplus S_2$ we denote the orthogonal sum of two lattices defined in the standard way. We write $S^n$ for the orthogonal sum of $n$ copies of a lattice $S$. The \textbf{ orthogonal complement of a sublattice} $S$ of a lattice $S'$ is denoted $(S)_{S'}^{\perp}$ and defined by 
$(S)_{S'}^{\perp}=\{x\in S'/ b(x,y)=0 \,\,\,\hbox{for all  }y\in S\}$.

\subsection{Discriminant forms}

Let $L$ be a non-degenerate lattice. 
The \textbf{dual lattice} $L^*$ of $L$ is defined by
$$L^*:=\text{Hom}(L,\mathbb Z) =\{x\in L \otimes \mathbb Q /\,\,\, b(x,y)\in \mathbb Z \hbox{  for all }y \in L \}.$$
and the \textbf{discriminant group $G_L$} by
$$G_L:=L^*/L.$$
This group is finite if and only if $L$ is non-degenerate. In the latter case, its order is equal to the absolute value of the lattice determinant $\mid \det (G(e)) \mid$ for any basis $e$ of $L$.
A lattice $L$ is \textbf{unimodular} if $G_L$ is trivial.

Let $G_L$ be the discriminant group of a non-degenerate lattice $L$. The bilinear form on $L$ extends naturally to a $\mathbb Q$-valued symmetric bilinear form on $L^*$ and induces a symmetric bilinear form 
$$b_L: G_L \times G_L \rightarrow\mathbb Q / \mathbb Z.$$
If $L$ is even, then $b_L$ is the symmetric bilinear form associated to the quadratic form defined by
$$
\begin{matrix}
q_L: G_L  &\rightarrow  & \mathbb Q/2\mathbb Z\\
q_L(x+L) & \mapsto & x^2+2\mathbb Z.
\end{matrix}
$$
The latter means that $q_L(na)=n^2q_L(a)$ for all $n\in \mathbb Z$, $a\in G_L$ and $b_L(a,a')=\frac {1}{2}(q_L(a+a')-q_L(a)-q_L(a'))$, for all $a,a' \in G_L$, where $\frac {1}{2}:\mathbb Q/2 \mathbb Z \rightarrow \mathbb Q / \mathbb Z$ is the natural isomorphism.
The pair $\boldsymbol{(G_L,b_L)}$ (resp. $\boldsymbol{(G_L,q_L)}$) is called the \textbf{discriminant bilinear} (resp. \textbf{quadratic}) \textbf{form} of $L$.

When the even lattice $L$ is given by its Gram matrix, we can compute its discriminant form using the following lemma as explained in Shimada \cite{Shim1}.
\begin{lemma}\label{lem:gram}
Let $A$ the Gram matrix of $L$ and $U$, $V \in Gl_n(\mathbb Z)$ such that
\[ UAV=D=\begin{pmatrix}
        d_1 &  & 0\\
         & \ddots & \\
        0 &  & d_n\\
\end{pmatrix}
\]
with $1=d_1=\ldots =d_k < d_{k+1} \leq \ldots \leq d_n$. Then
\[G_L\simeq \oplus_{i>k} \mathbb Z/(d_i).\]
Moreover the $i$th row vector of $V^{-1}$, regarded as an element of $L^*$ with respect to the dual basis $e_1^*$, ..., $e_n^*$ generate the cyclic group $\mathbb Z/(d_i)$.
\end{lemma}
\subsection{Root lattices}
In this section we recall only what is needed for the understanding of the paper. For proofs and details one can refer to Bourbaki \cite{Bo} or Martinet \cite{Ma}.
 
Let $L$ be a negative-definite even lattice. We call $e\in L$ a \textbf{root} if $q_L(e)=-2$. Put $\Delta(L):=\{e\in L/ q_L(e)=-2\}$. Then the sublattice of $L$ spanned by $\Delta(L)$ is called the \textbf{ root type} of $L$ and is denoted by $ \mathbf{L_{\hbox{root}}}$.

The \textbf{ lattices} $A_n=\langle a_1,a_2,\ldots ,a_n \rangle$ ($n\geq 1$), $D_l=\langle d_1,d_2,\ldots ,d_l \rangle$ ($l\geq 4$), $E_p=\langle e_1,e_2,\ldots ,e_p \rangle$ ($p=6,7,8$) defined by the following \textbf{Dynkin diagrams} are called the \textbf{root lattices}. All the vertices $a_j$, $d_k$, $e_l$ are roots and two vertices $a_j$ and $a_j'$ are joined by a line if and only if $b(a_j,a_j')=1$. We use Bourbaki's numbering \cite{Bo} and in brackets Conway-Sloane definitions \cite{CS}.

Denote $\epsilon_i$ the vectors of the canonical basis of $\mathbb R^n$ with the usual scalar product.

The lattice  $A_n$ can be  represented by the set of points in $\mathbb R^{n+1}$ with integer coordinates whose sum is zero,
and the lattice $D_l$ as the set of points of $\mathbb R^l$ with integer coordinates of even sum.
We can represent $E_8$  in the \textit{even coordinate system} \cite{CS} p.120 $E_8=D_8^+=D_8\cup (v+D_8)$ where $v=\frac{1}{2}\sum_{i=1}^8\epsilon_i.$
Then we represent $E_7$ as the orthogonal in $E_8$ of $v$, and $E_6$ as the orthogonal of $\langle v,w\rangle \simeq A_2$ in $E_8$ where $w=\epsilon_1+\epsilon_8$.  

\noindent
\begin{minipage}[t]{0.44\textwidth}
$\mathbf{A_n,\,\, G_{A_n}}$

Set 

$[1]_{A_n}=a_n^*=-\frac {1}{n+1} \sum_{j=1}^{n}(j)a_j$

then $A_n^*=\langle A_n,[1]_{A_n} \rangle$ and 

$ G_{A_n}=A_n^*/A_n  \simeq \mathbb Z /(n+1)\mathbb Z .$

\noindent
$q_{A_n}([1]_{A_n})=-\frac{n}{n+1}.$

Glue vectors $[i]_{A_n}=a_{n+1-i}^*$

Glue group $[i+j]=[i]+[j]$

\end{minipage}
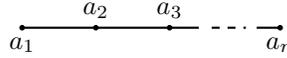
\begin{minipage}[t]{0.52\textwidth}
\bigskip
\begin{center}
\begin{tikzpicture}[scale=0.35]%
\draw[fill=black](0.1,1.7039)circle (0.0865cm)node [below] {$a_1$};
\draw[fill=black](2.9,1.7039)circle (0.0865cm)node [above] {$a_2$};
\draw[fill=black](5.7,1.7039)circle (0.0865cm)node [above] {$a_3$};
\draw[fill=black](9.9,1.7039)circle (0.0865cm)node [below]{$a_n$};
\draw[thick](0.1,1.7039)--(6.81,1.7039);
\draw[dashed,thick](7.3,1.7039)--(8.6,1.7039);
\draw[thick](9.01,1.7039)--(9.9,1.7039);
\end{tikzpicture}
\end{center}
\end{minipage}

\bigskip
\noindent
\begin{minipage}[t]{0.48\textwidth}
$\mathbf{D_l, G_{D_l}}.$ 

Set

\noindent
$[1]_{D_l} =-d_{l-1}^*=\hfill$
$\frac {1}{2}\left( \sum_{i=1}^{l-2}id_i+\frac {1}{2}(l-2)d_{l}+\frac {1}{2}ld_{l-1} \right)$

\noindent
$[2]_{D_l}=d_1^*=  \hfill$

$\sum_{i=1}^{l-2} d_i+\frac {1}{2}(d_{l-1}+d_l)$

\noindent
$[3]_{D_l}=-d_{l-1}^*+d_1^*=\hfill$
$ \frac {1}{2}\left( \sum_{i=1}^{l-2}id_i+\frac {1}{2}ld_{l}+\frac {1}{2}(l-2)d_{l-1}  \right)$

\noindent
then $D_l^*= \langle D_l, [1]_{D_l}, [3]_{D_l} \rangle,$
\end{minipage}
\begin{minipage}[t]{0.48 \textwidth}
\smallskip
\begin{center}
\begin{tikzpicture}[scale=0.4]%
\draw[fill=black](0.1,1.7039)circle (0.0865cm)node [above] {$d_l$};
\draw[fill=black](2.9,1.7039)circle (0.0865cm)node [below] {$d_{l-2}$};
\draw[fill=black](5.7,1.7039)circle (0.0865cm)node [below] {$d_{l-3}$};
\draw[fill=black](9.9,1.7039)circle (0.0865cm)node [below]{$d_1$};
\draw[fill=black](2.9,3.10)circle (0.0865cm)node [above]{$d_{l-1}$};
\draw[thick](0.1,1.7039)--(6.81,1.7039);
\draw[dashed,thick](7.3,1.7039)--(8.6,1.7039);
\draw[thick](9.01,1.7039)--(9.9,1.7039);
\draw[thick](2.9,1.7039)--(2.9,3.10);
\end{tikzpicture}
\end{center}
\end{minipage}

\noindent
\begin{minipage}[t]{0.95\textwidth}
\noindent
$G_{D_l}=D_l^*/D_l=<[1]_{D_l}>\simeq \mathbb Z/4\mathbb Z \,\,\,\,\text{if } l \text{ is odd,}$

\noindent
$G_{D_l}=D_l^*/D_l=<[1]_{D_l}, [3]_{D_l}>\simeq \mathbb Z/2\mathbb Z \times \mathbb Z/2\mathbb Z \,\,\,\text{if }l \text{ is even}.$

\noindent
$q_{D_l}([1]_{D_l})=-\frac{l}{4},\, \,\,q_{D_l}([2]_{D_l})=-1,\,\,b_{D_l}([1],[2])=-\frac{1}{2}. $
\end{minipage}

\bigskip
\noindent
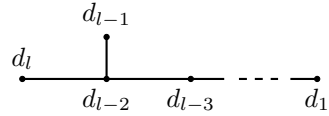
\begin{minipage}[t]{0.50\textwidth}
$\mathbf{E_6,\,G_{E_6}}\,\,$ 

Set

$[1]_{E_6}=\eta_6 =e_6^* =$

\noindent 
$-\frac {1}{3}(2e_1+3e_2+4e_3+6e_4+5e_5+4e_6)$,

then

$E_6^*=\langle E_6, \eta_6 \rangle$ and 

$G_{E_6}=E_6^*/E_6\simeq \mathbb Z /3 \mathbb{ Z}$

$[2]_{E_6}=-[1]_{E_6}.$

\noindent
$q_{E_6}(\eta_6)=-\frac{4}{3}.$ 


\end{minipage}
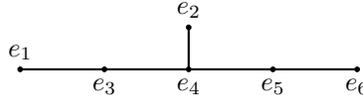
\begin{minipage}[t]{0.48\textwidth}
\smallskip
\begin{center}
 \begin{tikzpicture}[scale=0.4]%
\draw[fill=black](0.1,1.7039)circle (0.0765cm)node [above] {$e_1$};
\draw[fill=black](2.9,1.7039)circle (0.0765cm)node [below] {$e_{3}$};
\draw[fill=black](5.7,1.7039)circle (0.0765cm)node [below] {$e_{4}$};
\draw[fill=black](11.3,1.7039)circle (0.0765cm)node [below]{$e_6$};
\draw[fill=black](8.5,1.7039)circle (0.0765cm)node [below]{$e_5$};
\draw[fill=black](5.7,3.10)circle (0.0765cm)node [above]{$e_{2}$};
\draw[thick](0.1,1.7039)--(11.3,1.7039);
\draw[thick](5.7,1.7039)--(5.7,3.10);
\end{tikzpicture}
\end{center}
\end{minipage}

\bigskip
\noindent
\begin{minipage}[t]{0.48\textwidth}

$\mathbf{E_7,\,G_{E_7}}\,\,$

Set

\noindent
$[1]_{E_7}=\eta_7 =e_7^*=\hfill$

\noindent
$-\frac {1}{2}(2e_1+3e_2+4e_3+6e_4+5e_5+4e_6+3e_7)$,

then

$E_7^*=\langle E_7, \eta_7 \rangle$ and

\noindent
$G_{E_7}=E_7^*/E_7 \simeq \mathbb Z /2 \mathbb Z,$  

\noindent
$q_{E_7}(\eta_7)=-\frac{3}{2}.$ 
\end{minipage}
\begin{minipage}[t]{0.50\textwidth}
\smallskip
\begin{center}
\begin{tikzpicture}[scale=0.4]%
\draw[fill=black](0.1,1.7039)circle (0.0765cm)node [below] {$e_1$};
\draw[fill=black](2.9,1.7039)circle (0.0765cm)node [below] {$e_{3}$};
\draw[fill=black](5.7,1.7039)circle (0.0765cm)node [below] {$e_{4}$};
\draw[fill=black](9.9,1.7039)circle (0.0765cm)node [below]{$e_7$};
\draw[fill=black](5.7,3.10)circle (0.0765cm)node [above]{$e_{2}$};
\draw[thick](0.1,1.7039)--(6.81,1.7039);
\draw[dashed,thick](7.3,1.7039)--(8.6,1.7039);
\draw[thick](9.01,1.7039)--(9.9,1.7039);
\draw[thick](5.7,1.7039)--(5.7,3.10);
\end{tikzpicture}
\end{center}
\end{minipage}

\bigskip
\noindent
\begin{minipage}[t]{0.44\textwidth}
\smallskip
$\mathbf{E_8,\,G_{E_8}}\,\,$

$E_8^*= E_8. $  
 
\end{minipage}
\begin{minipage}[t]{0.52\textwidth}
\smallskip
\begin{center}
\begin{tikzpicture}[scale=0.4]%
\draw[fill=black](0.1,1.7039)circle (0.0765cm)node [below] {$e_1$};
\draw[fill=black](2.9,1.7039)circle (0.0765cm)node [below] {$e_{3}$};
\draw[fill=black](5.7,1.7039)circle (0.0765cm)node [below] {$e_{4}$};
\draw[fill=black](9.9,1.7039)circle (0.0765cm)node [below]{$e_8$};
\draw[fill=black](5.7,3.10)circle (0.0765cm)node [above]{$e_{2}$};
\draw[thick](0.1,1.7039)--(6.81,1.7039);
\draw[dashed,thick](7.3,1.7039)--(8.6,1.7039);
\draw[thick](9.01,1.7039)--(9.9,1.7039);
\draw[thick](5.7,1.7039)--(5.7,3.10);
\end{tikzpicture}
\end{center}
\end{minipage}

\subsection{Niemeier lattices}

An even unimodular lattice $Ni(L_{\text{root}})$  in dimension $24$ is called a Niemeier lattice and is obtained by gluing certain component lattices of $L_{\text{root}}$ by means of glue vectors.

If $L_{\text{root}}$ has components $L_1$, ...,$L_k$, the glue vectors have the form $y=(y_1, ...,y_k)$ where each $y_i$ can be regarded as a coset representative (or glue vector) for $L_i^*$ modulo $L_i$. These coset representatives, labeled $[0]$, $[1]$, ... $[d-1]$ for a component of determinant $d$ are listed in the previous subsection. 

The set of glue vectors for $Ni(L_{\text{root}})$ forms an additive group called the glue code. The Table \ref{Ta:Nie} below gives generators for the glue code. If a glue vector contains parentheses, this indicates that all vectors obtained by cyclically shifting the part of the vector inside the parentheses are also glue vectors. For example, the glue vectors for the Niemeier lattice $Ni(D_8^3)$ is described by $[(1 2 2)]$, that is the glue words are spanned by 
\[[1 2 2]=([1],[2],[2])=((1/2)^8,(0^7 1),(0^7 1))\]
\[[2 1 2]=([2],[1],[2])=((0^7 1),(1/2)^8,(0^7 1))\]
\[ [2 2 1]=([2],[2],[1])=((0^7 1),(0^7 1),(1/2)^8)\]

The full glue code for this example contains the eight vectors $[0 0 0]$, $[1 2 2]$, $[2 1 2]$, $[2 2 1]$, $[0 3 3 ]$, $[3 0 3 ]$, $[3 3 0]$, $[1 1 1]$.
\subsection{Elliptic fibrations}

We recall that all the elliptic fibrations of a $K3$ surface come from primitive embeddings of a specific rank $6$ lattice into a Niemeier lattice denoted $Ni(L_{\text{root}})$.

\subsection{Nikulin and Niemeier's results}

\begin{lemma}[Nikulin \cite{Nik}, Proposition 1.4.1]\label{L:nik}
Let $L$ be an even lattice. Then, for an even overlattice $M$ of $L$, we have a subgroup $M/L$ of $G_L=L^*/L$ such that $q_L$ is trivial on $M/L$. This determines a bijective correspondence between even overlattices of $L$ and subgroups $G$ of $G_L$ such that $q_L\mid_G=0$.
\end{lemma}

\begin{lemma}[Nikulin \cite{Nik}, Proposition 1.6.1]\label{L:nik0}
Let $L$ be an even unimodular lattice and $T$ a primitive sublattice. Then we have
$$G_T\simeq G_{T^{\perp}} \simeq L/(T\oplus T^{\perp}),\,\,\,\,\,\,\,q_{T^{\perp}}=-q_T.$$
In particular, $ \det T =\det  T^{\perp} =[L:T\oplus T^{\perp}]$.
\end{lemma}

\begin{theorem} [Nikulin \cite{Nik} Corollary 1.6.2]\label{T:nik1}
Let $L$ and $M$ be non-degenerate even integral lattices such that
$$G_L \simeq G_M,\,\,\,\,\,\,\,q_L=-q_M.$$
Then there exists an unimodular overlattice $N$ of $L\oplus M$ such that

1) the embeddings of $L$ and $M$ in $N$ are primitive

2) $L_N^{\perp}=M$ and $M_N^{\perp}=L$.
\end{theorem}

\begin{theorem}[Niemeier \cite{Nie}]\label{T:Nie}
A negative-definite even unimodular lattice $L$ of rank $24$ is determined by its root lattice $L_{\hbox{root}}$ up to isometries. There are $24$ possibilities for $L$ and $L/L_{\hbox{root}}$ listed in Table \ref{Ta:Nie}.
\end{theorem}
\begin{table}
\begin{center}
\begin{tabular}{|l|r|l|r|}
\hline

$L_{\hbox{root}}$ & $L/L_{\hbox{root}}$ & \text{Generators for the glue code} \\ \hline
  $E_8^3$    &  (0) & (000)\\ \hline
$ D_{16}E_8 $ & $\mathbb Z /2 \mathbb Z $ & [10] \\ \hline
$ D_{10}E_7^2$ & $(\mathbb Z/2 \mathbb Z)^2$ & [110],[301]\\\hline 
$ A_{17}E_7$ & $\mathbb Z /6 \mathbb Z$ & [31] \\ \hline
$D_{24}$ & $\mathbb Z /2 \mathbb Z $ & [1]\\ \hline
$D_{12}^{ 2}$ & $(\mathbb Z /2 \mathbb Z)^2$ & [12],[21]\\ \hline
$D_8^{ 3 }$ & $(\mathbb Z /2 \mathbb Z)^3$ & [(122)] \\ \hline
$ A_{15}D_9 $ & $\mathbb Z /8 \mathbb Z $ & [21]\\ \hline
$E_6^{ 4}$ &  $(\mathbb Z /3 \mathbb Z)^2$ & [1(012)] \\ \hline

$ A_{11}D_7E_6$ & $\mathbb Z /12 \mathbb Z $ & [111] \\ \hline
$D_6^{ 4}$ & $(\mathbb Z /2)^4$ & [\text{even perms of} (0123)]\\ \hline
$A_9^{ 2}D_6$ & $\mathbb Z /2 \mathbb Z \oplus \mathbb Z /10 \mathbb Z$ & [240],[501],[053]\\
\hline
$A_7^2D_5^2$& $\mathbb Z/4 \times \mathbb Z/8$ & [1112],[1721] \\ \hline
$A_8^3$ & $\mathbb Z/3 \times \mathbb Z/9$ & [(114)] \\ \hline
$A_{24}$ & $\mathbb Z/5$ & [5] \\ \hline
$A_{12}^2$ & $\mathbb Z/13 \mathbb Z$ & [2(11211122212)]\\ \hline
$D_4^6$& $(\mathbb Z/2)^6$ & [111111],[0,(02332)] \\ \hline
$A_5^4D_4$ & $\mathbb Z/2 \times (\mathbb Z/6)^2$ & [2(024)0],[33001],[30302],[30033]\\ \hline
$A_6^4$ & $(\mathbb Z/7)^2$ & [1(216)] \\ \hline
$A_4^6$& $(\mathbb Z/5)^3$ & [1(01441)]\\ \hline
$A_3^8$ & $(\mathbb Z/4)^4$ & [3(2001011)]\\ \hline
$A_2^{12}$ & $(\mathbb Z/3)^6$& [2(11211122212)] \\ \hline
$A_1^{24}$& $(\mathbb Z/2)^{12}$ &[1(00000101001100110101111)]\\ \hline
\end{tabular}
\end{center}
\caption{Niemeier lattices}
\label{Ta:Nie}
\end{table}
The lattices $L$ defined in Table \ref{Ta:Nie} are called \textbf{Niemeier lattices}.

\subsection{The Kneser-Nishiyama's method}

Recall that a $K3$ surface may admit more than one elliptic fibration, but up to isomorphism, there is only a finite number of elliptic fibrations \cite{St}.
To establish a complete classification of the elliptic fibrations on the $K3$ surface $Y_2$, we use Nishiyama's method based on lattice theoretic ideas \cite{Nis}. The technique builds on a converse of Nikulin's results.

Given an elliptic $K3$ surface $X$, Nishiyama aims at embedding the frames of all elliptic fibrations into a negative-definite lattice, more precisely into a Niemeier lattice of rank $24$. For this purpose, he first determines an even negative-definite lattice $M$ such that
$$q_M=-q_{NS(X)}, \,\,\,\,\,\,\hbox{rank}(M)+\rho(X)=26.$$
By theorem \ref{T:nik1}, $M\oplus W(X)$ has a Niemeier lattice as an overlattice for each frame $W(X)$ of an elliptic fibration on $X$. Thus one is bound to determine the (inequivalent) primitive embeddings of $M$ into Niemeier lattices $L$. To achieve this, it is essential to consider the root lattices involved. In each case, the orthogonal complement of $M$ into $L$ gives the corresponding frame $W(X)$.

\subsubsection{The transcendental lattice and argument from Nishiyama paper}
Denote by $ T(X)$ the transcendental lattice of $X$, i.e. the orthogonal complement of $NS(X)$ in $H^2(X,\mathbb Z)$ with respect to the cup-product,
$$T(X)=\hbox{NS}(X)^{\perp} \subset H^2(X,\mathbb Z).$$
In general, $ T(X)$ is an even lattice of rank $r=22-\rho(X)$ and signature $(2,20-\rho(X))$. Let $t=r-2$. By Nikulin's theorem \cite{Nik}, $T(X)[-1]$ admits a primitive embedding into the following indefinite unimodular lattice:
$$T(X)[-1] \hookrightarrow U^t \oplus E_8.$$
Then define $M$ as the orthogonal complement of $ T(X)[-1]$ in $ U^t \oplus E_8$. By construction, $M$ is a negative-definite lattice of rank $2t+8-r=r+4=26-\rho(X)$. 

\noindent By lemma \ref{L:nik0} the discriminant form satisfies
$$q_M=-q_{ T(X)[-1]}=q_{ T(X)}=-q_{\hbox{NS}(X)}.$$
Hence $M$ takes exactly the shape required for Nishiyama's technique.

\section{Elliptic fibrations of $Y_{10}$}
\noindent To prove the primitivity of the embeddings we shall use the following.

\begin{lemma}\label{prim}
A lattice embedding is primitive if and only if the greatest common divisor of the maximal minors of the embedding matrix with respect to any choice of basis is $1$.
\end{lemma}

\begin{remark}
In that case Smith form of the embedding matrix has only $1$ on the diagonal.
\end{remark}

\subsection{Kneser-Nishiyama technique applied to $Y_{10}$}


By the Kneser-Nishiyama method we determine elliptic fibrations of $Y_{10}$. For further details we refer to \cite{Nis}, \cite{SS}, \cite{BL}, \cite {BGL}. In  \cite{Nis}, \cite{BL}, \cite {BGL} only singular $K3$ (i.e. of Picard number $20$) are considered. In this paper we follow \cite{SS} we briefly recall.



Let us describe how to determine  $M$ in the case of the  $K3$ surface $Y_{10}$.

Let $T(Y_{10})$ be the transcendental lattice of $Y_{10}$, that is the orthogonal complement of $NS(Y_{10})$ in $H^2(Y_{10},\mathbb Z)$. The lattice $T(Y_{10})$ is an even lattice of rank $r=22-\rho(Y_{10})=2$ and signature $(1,\
1)$. Let  $t:=r-2=0$. By Nikulin's theorem \cite{Nik} (Theorem $1.12.4$), $T(Y_{10})[-1]$ admits a primitive embedding into the following indefinite unimodular lattice:
\[T(Y_{10})[-1] \hookrightarrow E_8[-1], \]
where $E_8$ denotes the unique  positive definite  even  unimodular  lattice of rank $8$.  Since 
$$ T(Y_{10})=\left ( \begin{matrix}
                 6 &  0 \\
                 0  & 12
                    \end{matrix}
                   \right )
                   $$
by Bertin's results \cite{Ber}, we must determine a primitive embedding of $T_{Y_{10}}[-1]$ into $E_8[-1]$. If ${u,v}$ denotes a basis of  $T_{Y_{10}}[-1]$, define
$$
\begin{matrix}
\Phi (u)= & 2e_1+e_3\\
\Phi(v) = & 3e_2+2e_1+4e_3+6e_4+5e_5+4e_6.
\end{matrix}
$$

By lemma \ref{prim}, this is a primitive embedding since the greatest common divisor of the maximal minors of the embedding matrix is $1$.

Then
$$M=\langle \Phi(u),\Phi(v) \rangle _{E_8}^{\perp}=\langle e_2,e_1+3e_4,e_3,e_5,4e_6+3e_7,e_8 \rangle.$$
With a reduced basis using LLL algorithm, we get

$$M=\langle e_3 \rangle \oplus \langle e_8,x \rangle \oplus \langle e_2,e_5,y \rangle $$
with
$$x=2e_1+3e_2+4e_3+6e_4+5e_5+4e_6+3e_7+e_8$$
and
$$y=e_2+e_1+2e_3+3e_4+e_5.$$
Moreover
$$\langle e_8,x \rangle =A_2$$
and the Gram matrix of $N=\langle e_2,e_5,y \rangle$ is
$$
\left (
\begin{matrix}
-2 & 0 & 1 \\
0  &  -2 & 1 \\
1  &  1  & -4
\end{matrix}
\right )
$$
with determinant $-12$.
So we get $M=A_1\oplus A_2 \oplus N$ and all the elliptic fibrations are given by the primitive embeddings of $M$ into the Niemeier lattices  $Ni(L_{\text{root}})$.

\begin{remark}
Since $M$  is not a root lattice, the primitive embeddings of $M$ into a Niemeier lattice $Ni(L_{\text{root}})$ are not necessarily given by the primitive embeddings of $M$ into its root lattice  $L_{\text{root}}$. This is the crucial difference with the situation encountered in the previous papers \cite{BL}, \cite{BGL}.
\end{remark}

\subsection{Types of primitive embeddings of $N$ into the Niemeier lattices}\label{l}

There are essentially three types of primitive embeddings of $N$ into a Niemeier lattice $Ni(L_{\text{root}})$ .

\begin{enumerate}[leftmargin=*]
\item
The embeddings of $N$ into the root lattices $A_n$, $D_n$ and $E_l$ of $L_{\text{root}}$.

\item \label{(2)}

The embeddings of $N$ into a direct sum of two root lattices of $L_{\text{root}}$, following from the fact that a vector or norm $4$ is the sum of two roots belonging to two different root lattices.

\item
The embeddings of $N$ into $Ni(L_{\text{root}})$.

\end{enumerate}

The type $(2)$ subdivises in two cases denoted as follows:
\begin{enumerate}[label=\alph*), leftmargin=*]
\item Type $(A_2,A_2).$

\noindent
That means $(-4)=(r_1,r_2)$ and a root $v_1$ (resp.$v_2$) is in the same root lattice as $r_1$ (resp. $r_2$) and satisfies $r_1.v_1=1$ (resp. $r_2.v_2=1$). Thus in the first root lattice, the roots $r_1$ and $v
_1$ (resp. in the second root lattice $r_2$ and $v_2$) realizes a root lattice $A_2$. (The 3-isogenous of fibration 21-c is on $Y_{10}$ and has rank $1$ and fibration with singular fibers of type $6A_2 A_5$: it is an embedding of type $(A_2,A_2)$ into $E_6^4$.)

\noindent Another example of embedding of type $(A_2,A_2)$ is the rank $7$ fibration $2D_4 3A_1$ obtained with the following primitive embedding into $D_4^6$: $A_2$ in $D_4$ three times and $A_1$ in $D_4$.

\item Type $(A_1,A_3)$ \label{l}

\noindent That means $(-4)=(r_1,r_2)$ and the roots $v_1$ and $v_2$ are in the same root lattice as $r_2$  and satisfies $r_2.v_1=1$, $r_2.v_2=1$ and $v_1.v_2=0$. Thus in the first root lattice, the root $r_1$ is viewed \
as $A_1$ and in the second root lattice, the roots $v_1$, $r_2$, $v_2$ are viewed as $A_3$. All the fibrations presented in theorem 7.1, except fibrations $(5)$ and $(6)$ can be obtained with this type of embedding.

\noindent Thus, case a) and b) are reduced to find primitive embeddings of direct sums of root lattices, that is of $A_1\oplus A_2 \oplus A_2 \oplus A_2$, two of the $A_2$ being embedded in different root lattices, in case a) and  $A_1\oplus A_2 \oplus A_1 \oplus A_3$, the two $A_1$ and $A_3$ being embedded in different root lattices, in case b).
\end{enumerate}

As for type $(2)$ the type $(1)$ can be divided in two cases.

1.a)

 From a primitive embedding of $A_2 \oplus A_2$ into a root lattice we get a primitive embedding of $N$ in the same root lattice as follows.

Denoting $A_2^{(1)}=\langle a_1,a_2 \rangle$ and $A_2^{(2)}=\langle b_1,b_2 \rangle$ we can take, for example, one of the four possible embeddings $(-4)=a_1+b_1$, $v_1=a_2$, $v_2=b_2$.

1.b)

 From a primitive embedding of $A_3 \oplus A_1$ into a root lattice we get a primitive embedding of $N$ in the same root lattice as follows.

Denoting $A_3=\langle a_1,a_2,a_3 \rangle$ and $A_1=\langle b_1 \rangle$ with $a_i.a_i=-2$, $i=1,2,3$, $a_1.a_2=a_2.a_3=1$ and $a_1.a_3=0$, we get the primitive embedding of $N$ as $(-4)=a_2+b_1$, $v_1=a_1$, $v_2=a_3$.

\begin{remark}
In an irreducible root lattice, contrary to case $2,$ the decomposition of a
$\left(  -4\right)  $ vector in two orthogonal roots is not unique. Thus the
same embedding can be viewed either of type \ $\left(  A_{2},A_{2}\right)  $ or
of type $\left(  A_{3},A_{1}\right)  .$
\end{remark}

The type $(3)$ needs to find representatives of glue vectors with norm $-4$.

As we can see there is a lot of glue vectors of norm $-4$, hence added to the other types of embeddings, provide a huge amount of elliptic fibrations for $Y_{10}$. We are not intending to give all of them but only those useful for the main results of our paper.

For computations in case $(3)$ we use Conway and Sloane notations \cite{CS}.

For the two first types we use notations and results of Nishiyama \cite{Nis} expressed in Bourbaki's notations \cite{Bo}.
 
 We proceed in the following way. Once is found a primitive embedding of $N$ into a Niemeier lattice (eventually its root lattice), we compute its orthogonal. If the orthogonal contains $A_1$ or $A_2$ or both we test if the previous embedding of $N$ plus $A_1$ (resp. $A_2$) (resp.$A_1 \oplus A_2$) is primitive to insure the existence of an elliptic fibration of $Y_{10}$.

\subsection{Some remarks}
\begin{lemma}\label{lem:3.2}

\begin{enumerate}[leftmargin=*]
\item The lattice $N$ embeds primitively in $D_5$ by $N=\langle d_5,d_4,d_3+d_1 \rangle$ and
\[(N)_{D_5}^{\perp}=\langle d_1+d_2,d_5+d_4+2d_3+d_2 \rangle \simeq A_2.\]
\item This defines an embedding of $A_2 \oplus N$ in $D_5$, by
\[N\oplus A_2=<d_5,d_4,d_3+d_1> \oplus <d_1+d_2,d_5+d_4+2d_3+d_2>\]
 which is not primitive, since $N\oplus A_2$ is a sublattice of index $3$ into $D_5$.
 \end{enumerate}
\end{lemma}
\begin{proof}
The first assertion is a simple computation and the second follows from the fact that the two lattices $N\oplus A_2$ and $D_5$ have the same rank and satisfy the relations $\det(A_2 \oplus N)=4\times 3^2$ and $\det(D_5)=4$.
\end{proof}

\begin{lemma}
\begin{enumerate}[leftmargin=*]
\item Apart from itself, $A_2 \oplus N$ has a unique overlattice $S=D_5$ such that $S/(A_2\oplus N) \simeq \mathbb Z /3\mathbb Z$.

\item The unique overlattice of $A_1 \oplus N$ is $A_1 \oplus N$ itself.

\item Apart from itself, $A_1 \oplus A_2 \oplus N$ has a unique overlattice $S=A_1 \oplus D_5$ such that $S/(A_1 \oplus A_2\oplus N) \simeq \mathbb Z /3\mathbb Z$.
\end{enumerate}

\end{lemma}
\begin{proof}
The Gram matrix of $N$ with respect to a basis $(e_i)$ being 
\[A=\begin{pmatrix}
	-2 & 0 & 1\\
	0 & -2 & 1\\
	1 & 1 & -4\\
\end{pmatrix},\]
we find easily its Smith normal form $S$  
\[S=\begin{pmatrix}
	1& 0 & 0\\
	0 & 1 & 0\\
	0 & 0 & 12\\
\end{pmatrix}
\]
with $S=UAV$ and $V^{-1}=\begin{pmatrix}
	1 & 1 & -4\\
	0& -2 & 1\\
	-1 & 0 & 3\\
\end{pmatrix}.$

Applying lemma \ref{lem:gram}, we deduce a generator $a$ of the cyclic group $G_N \simeq \mathbb Z/12 \mathbb Z$
\[a=-e_1^*+3e_3^*\]
Since the Gram matrix of the $e_i^*$ is the matrix inverse of the Gram matrix of the $e_i$, that is $N^{-1}$
\[N^{-1}=\begin{pmatrix}
	-7/12 & -1/12 & -1/6\\
	-1/12 & -7/12 & -1/6\\
	-1/6 & -1/6 & -1/3\\
\end{pmatrix}\]

we find $q_N=q(a) \operatorname{mod} 2=-7/12$.

\begin{enumerate}[leftmargin=*]

\item
Now
$$(A_2 \oplus N)^*/A_2 \oplus N \simeq A_2^*/A_2 \oplus N^*/N.$$
\noindent and $q_{A_2}([1]_{A_2})=\left( - \frac {2}{3}\right)$.

\noindent 
By Nikulin's lemma \ref{L:nik}, for an overlattice $M$ of $A_2\oplus N$ we have a subgroup $M/A_2\oplus N$ of $G_{A_2\oplus N}= A_2^*/A_2 \oplus N^*/N$ such that $q_{A_2\oplus N}$ is trivial on $M/A_2\oplus N$. This determines a bijective correspondance between even overlattices of $A_2 \oplus N$ and isotropic subgroups $G$ of $G_{A_2\oplus N}$.

\noindent In our situation, there is a unique isotropic index $3$ subgroup $G=\langle [1]_{A_2}+4a \rangle \simeq \mathbb Z/3 \mathbb Z$ of $G_{A_2\oplus N}$ such that

$$q_{A_2\oplus N}([1]_{A_2}+4a)=-\frac{2}{3}-\frac{28}{3}=-10=0\,\,\, \operatorname{mod} 2$$

\noindent Thus $A_2\oplus N$ has a unique overlattice $S$ such that $S/(N\oplus A_2)\simeq \mathbb Z/3 \mathbb Z$. By lemma \ref{lem:3.2}, $A_2\oplus N$ embeds in $D_5$ with index $3$; thus this unique overlattice is $D_5$.

\item Since $ A_1^*/A_1=<[1]_{A_1}> $ with $q_{A_1}([1]_{A_1})=-\frac{1}{2}$, we deduce that the unique isotropic subgroup $G$ of $G_{A_1 \oplus N}$ is $G=<2 [1]_{A_2} +12 a>$.

\item The assertion follows from $q_{A_1 \oplus A_2 \oplus N}(2 [1]_{A_1}+[1]_{A_2}+4a)=-2-10=0 \quad \operatorname{mod} 2.$
\end{enumerate}
\end{proof}
So we deduce the following corollary.

\begin{definition}
An embedding of $M$ into $L_{\text{root}}$ is said to be extremal if its orthogonal in $L_{\text{root}}$ is a root lattice.
\end{definition}

\begin{corollary}
There is no primitive extremal embedding of $M$ into $Ni(D_5^2A_7^2)$.
\end{corollary}
\begin{proof}
Such an embedding would be realized by embedding $N$ into $D_5$ and $A_1 \oplus A_2$ into $D_5$ and in that case it cannot be extremal.
\end{proof}

\subsection{Primitive embeddings into root lattices}

 We give only the examples useful for explaining the Weierstrass equations of the elliptic fibrations given in the paper.

\subsubsection{Embedding of $N\oplus A_1$ into $D_5$}
The primitive embedding of $N\oplus A_1=\langle d_5,d_4,d_3+d_1,d_1+d_2 \rangle$ into $D_5$ satisfies 
$$(N\oplus A_1)_{D_5}^{\perp}=\langle d_1+2d_2+4d_3+2d_5+2d_4 \rangle=(-6).$$

\subsubsection{Embeddings into $D_6$}
Exactly as in the case \ref{(2)} we get essentially two types of embeddings which are not isomorphic.

\begin{lemma}\label{D_6}
a) The following primitive embedding of $N$ into $D_6$ given by
\[\langle d_6,d_5,d_4+d_2\rangle \hookrightarrow D_6\]
has its orthogonal in $D_6$ isomorphic to $A_2 \oplus (-4)$.

The embedding of $N\oplus A_1$ into $D_6$ given by
\[\langle d_6,d_5,d_4+d_2,d_3+d_2\rangle \hookrightarrow D_6\]
is primitive.

b) But the primitive embedding
\[\langle d_6,d_4+d_2,d_1\rangle \hookrightarrow D_6\]
has its orthogonal in $D_6$ isomorphic to $A_1\oplus (-4) \oplus (-6)$.
\end{lemma}
\begin{proof}
a) The orthogonal is $\langle d_6+d_5+2d_4+2d_1,d_3-2d_1,d_2+2d_1\rangle$. With LLL, we find
\[N_{D_6}^{\perp}=\langle d_6+d_5+2d_4+d_3,d_3+d_2,d_6+d_5+2d_4+2d_3+2d_2+2d_1 \rangle =A_2 \oplus (-4).\]
The matrix of the embedding of $N\oplus A_1$ being
\begin{center}
$\begin{pmatrix}                                                                                                                                                                                                   
        1 & 0 & 0 & 0 & 0 & 0\\                                                                                                                                                                                    
        0 & 1 & 0 & 0 & 0 & 0\\                                                                                                                                                                                    
        0& 0 & 1 & 0 & 1 & 0\\                                                                                                                                                                                     
        0 & 0 & 0 & 1 & 1 & 0\\                                                                                                                                                                                    
                                                                                                                                                                                                                   
\end{pmatrix}$

\end{center}
it follows that this embedding is primitive since we can extract a matrix of dimension $4$ and determinant $1$.
Its orthogonal is $\langle d_6+d_5+2d_4+2d_3+2d_2+2d_1,d_3+3d_2+d_1\rangle $, with Gram matrix of determinant $24$ and no roots.

b) In the second case the orthogonal in $D_6$ of the embedding
\[\langle d_6,d_4+d_2,d_1\rangle \hookrightarrow D_6\]
is
\[\langle d_6+3d_5+2d_4,-2d_5+d_3,3d_5+2d_2+d_1\rangle \]

With LLL , we get
\[\langle A_1=d_6+d_5+2d_4+d_3\oplus (-4) \oplus (-6) \rangle .\]

As previously we can prove that the embedding of $N\oplus A_1$

\[\langle d_6,d_4+d_2,d_1, d_6+d_5+2d_4+d_3\rangle \hookrightarrow D_6\]
is primitive.

\end{proof}
\begin{lemma} \label{lem:D_n}
For all $n\geq 7$, there exists a primitive embedding of $N$ into $D_n$, such that
\[((N)_{D_n}^{\perp})_{\text{root}} \simeq A_2 \oplus D_{n-5}.\] 
 
For $n=6$ 
\[((N)_{D_6}^{\perp})_{\text{root}}\simeq A_2.\]

\end{lemma}
\begin{proof}
Suppose first $n>8$ and consider the following embedding
\[N=\langle d_n,d_{n-1},d_{n-2}+d_{n-4}\rangle \hookrightarrow D_n.\]
The following roots of $D_n$ are orthogonal to $N$:
\[d_1,d_2,...,d_{n-6},x=d_n+d_{n-1}+2(d_{n-2}+d_{n-3}+d_{n-4}+d_{n-5})+d_{n-6}.\]
These roots satisfy the relations:
\[d_{n-6}.d_{n-7}=1,\quad  d_{n-7}.d_{n-8}=1,\quad ...,\quad d_2.d_1=1\]
\[x.d_{n-6}=x.d_{n-8}=...=x.d_1=0\]
\[x.d_{n-7}=1.\]
We deduce $\langle d_{n-6},x,d_{n-7},d_{n-8}, ...,d_1\rangle \simeq D_{n-5}.$
Consider also the following roots $y$ and $z$ of $D_n$:
\[y=d_n+d_{n-1}+2(d_{n-2}+d_{n-3})+d_{n-4},\qquad z=-(d_{n-3}+d_{n-4}).\]
They satisfy $y.z=1$, hence $\langle y,z \rangle \simeq A_2$. They are orthogonal to $N$, $d_{n-6}$, $x$, $d_{n-8}$, ..., $d_1$. Finally $(N_{D_n}^{\perp})_{\text{root}} \simeq A_2 \oplus D_{n-5}$.

Notice, if $n=8$, that $D_{n-5}\simeq A_3$.

If $n=7$, 
\[(N_{D_7}^{\perp})_{\text{root}} \simeq  \langle y,z\rangle \oplus \langle x,d_1 \rangle \simeq A_2 \oplus 2A_1 \simeq A_2 \oplus D_2.\]
If $n=6$, a direct computation, see lemma \ref{D_6} a), gives $N_{D_6}^{\perp}\simeq A_2 \oplus (-4) \simeq A_2 \oplus D_1$.

\end{proof}

\begin{lemma}\label{lem:D_7}
\begin{enumerate}
\item There is a primitive embedding of $N\oplus A_1$ into $D_7$ with orthogonal $A_1\oplus A_2$.

\item There is no  primitive embedding of $N\oplus A_2$ into $D_7$ with orthogonal $A_1\oplus A_1$.
\end{enumerate}
\end{lemma}
\begin{proof}
\begin{enumerate}[leftmargin=*]
\item Consider the embedding:
\[N\oplus A_1=\langle d_7,d_6,d_5+d_3,d_1\rangle \hookrightarrow D_7\]
Take $d=a_1d_1+..+a_7d_7$ satisfying $nd=\lambda_7 d_7+\lambda_6d_6+\lambda_5 (d_5+d_3)+\lambda_1d_1$. From the relations $na_1=\lambda_1$, $a_2=0$, $na_3=na_5=\lambda_5$, $a_4=0$, etc..we find $d=a_1d_1+a_3(d_
3+d_5)+a_6d_6+a_7d_7 $, that is  $d\in N\oplus A_1$, proving the primitivity of the embedding.

\noindent We can also prove the primitivity using the lemma \ref{prim}.

\noindent The orthogonal of the embedding is
\[\langle d_6+d_7+2d_5+2d_4+2d_3+2\                                                                                                                                                                    d_2+d_1 \rangle \oplus \langle d_4+d_3, -d_6-d_7-2d_5-2d_4-d_3 \rangle\simeq A_1\oplus A_2\]

\item The embedding
\[N\oplus A_2=\langle d_7,d_6,d_5+d_3,d_4+d_3,-d_7-d_6-2d_5-2d_4-d_3\rangle \hookrightarrow D_7\]
is not primitive
since its matrix
\begin{center}
$\begin{pmatrix}                                                                                                                                                                                                   
        1& 0 & 0 & 0 & 0& 0 & 0 \\                                                                                                                                                                                 
        0 & 1 & 0 & 0 & 0 & 0 & 0\\                                                                                                                                                                                
        0 & 0 & 1 & 0 & 1 & 0 & 0\\                                                                                                                                                                                
        0 & 0 & 0 & 1 & 1 & 0 & 0\\                                                                                                                                                                                
        -1 & -1 & -2 & -2 & -1 & 0 & 0\\                                                                                                                                                                           
                                                                                                                                                                                                                   
\end{pmatrix}$
\end{center}
has its extremal minor of determinant $3$.

\end{enumerate}

\end{proof}


\subsubsection{Embedding of $N$ into $E_8$}


\begin{lemma}\label{lem:unimod}

\begin{enumerate}[leftmargin=*]
\item There is a primitive embedding of $N$ into $E_8$ whose orthogonal in $E_8$ is $A_2 \oplus A_3$.

\item There is a primitive embedding of $N\oplus A_1$ into $E_8$ whose  orthogonal in $E_8$ is $ A_3\oplus (-6)$.

\item There is a primitive embedding of $N\oplus A_1 \oplus A_2$ into $E_8$ whose orthogonal in $E_8$ is $(-6) \oplus (-12)$.

\end{enumerate}
\end{lemma}
\begin{proof}
\begin{enumerate}[leftmargin=*]
\item Embed primitively $N$ in $E_8$ as $N=\langle e_2,e_7,e_4+e_6 \rangle$, we get
$$(N)_{E_8}^{\perp}=\langle e_1,e_2+3e_3+2e_4,e_5-2e_3,2e_3+e_6-e_8,-e_3+e_7+2e_8 \rangle$$
With LLL we obtain

 \begin{tabular}{ll}                                                                                                                                                                     
$(N)_{E_8}^{\perp}=$ &$\langle e_1,e_2+e_3+2e_4+e_5,2e_1+2e_2+3e_3+4e_4+3e_5+2e_6+e_7 \rangle  $\\
 & $\oplus \langle -e_1-e_2-2e_3-2e_4-e_5-e_6-e_7-e_8, $\\
  & $2e_1+3e_2+4e_3+6e_4+5e_5+4e_6+3e_7+2e_8 \rangle$
\end{tabular}

with Gram matrix

$$\left (
\begin{matrix}
-2 & 1 & -1 \\
1  & -2 & 0 \\
-1 & 0 & -2
\end{matrix}
\right) \oplus \left (
\begin{matrix}
-2 & 1 \\
1 & -2
\end{matrix}
\right ).
$$
Thus
$$(N)_{E_8}^{\perp}=A_3 \oplus A_2.$$
\item The following embedding
\[N \oplus A_1=\langle e_2,e_7,e_4+e_6,-e_1-e_2-2e_3-2e_4-e_5-e_6-e_7-e_8 \rangle  \hookrightarrow E_8\]
has for orthogonal in $E_8$, the lattice $A_3\oplus (-6)$.

\item If moreover we embed $A_2$ into the previous $A_3$ we find $(-6)\oplus (-12)$ as orthogonal of $N\oplus A_1 \oplus A_2$ into $E_8$. 
\end{enumerate}
\end{proof}

\subsubsection{Embedding of $N\oplus A_1$ into $E_7$}\label{3.4.10}

The following primitive embedding of $N\oplus A_1$ into $E_7$ given by
\[\langle e_1,e_3+e_5,e_6,e_2 \rangle \hookrightarrow E_7\]
satisfies
\[((N\oplus A_1)_{E_7}^{\perp})_{\text{root}}=\langle e_7\rangle =A_1.\]

\subsubsection{Embeddings of $N$ into $E_6$}
\begin{lemma}\label{lem:E_6}
There are at least two types of non isomorphic primitive embeddings of $N$ into $E_6$:
\begin{enumerate}[leftmargin=*]
\item
\[\phi_0(N)=\langle e_1,e_1+e_2+2e_3+2e_4+e_5,e_2+e_3+2e_4+2e_5+2e_6 \rangle.\]
with 
\[(\phi_0(N))_{E_6}^{\perp}=\langle e_2,e_4,e_5 \rangle \simeq A_3.\]
\item
\[\phi_1(N)=\langle e_1,e_3+e_5,e_6 \rangle,\]
with 
\[((\phi_1(N))_{E_6}^{\perp})_{\text{root}}\simeq A_2 .\]

\end{enumerate}

\end{lemma}

\begin{proof}
\begin{enumerate}[leftmargin=*]
\item The embedding $\phi_0$ is given in Nishiyama \cite{Nis}.
\item  We get
\[\begin{matrix}
(\phi_1(N))_{A_6}^{\perp}= &
\langle e_1+2e_3-2e_5-e_6,3e_4+4e_5+2e_6,e_2\rangle\\
  = &\langle e_1+2e_2+2e_3+2e_5+e_6,e_2,3e_4+4e_5+2e_6
\rangle,
\end{matrix}\]
hence the result.

\end{enumerate}
\end{proof}
\subsubsection{Fibrations involving embeddings $\phi_0$ and $\phi_1$}

a) The embedding $\phi_0$ of $N$ into $E_6$  leads to the rank $0$ and $3$-torsion fibration $A_3 2A_2 A_5 E_6$ with Weierstrass equation $H_j$.

b) With the same embeddings but replacing the embedding $\phi_0$ by $\phi_1$ we get the rank $1$ fibration $3A_2$ $A_5$ $E_6$.


\begin{remark}
To illustrate the complexity of the determination of elliptic fibrations of $Y_{10}$, notice the following fibrations all with two fibers of type $A_5$ coming from primitive embeddings into various Niemeier lattices:

$2A_1 2A_5 $ $(r=6)$ resulting from an embedding in $A_5^4D_4$ (type $(A_3,A_1)$ into $A_5^2$),

$2A_12A_22A_5$  $(r=2)$ resulting from an embedding into $E_6^4$,

$A_12A_2A_32A_5$ $(r=0)$ resulting from an embedding into $Ni(A_{11}D_7E_6)                                                                                                                                           
$.
\end{remark}

Besides, we recall a result obtained by Nishiyama.
\begin{proposition}\label{PNis}
Up to the action of the Weyl group, the unique primitive embeddings of $A_1$ and $A_2$ in the following root lattices together with their orthogonals are given in the following list

\begin{itemize}[leftmargin=*]

\item $A_1= \langle d_l \rangle  \subset D_l, \,\, l\geq 4,\,\,\,\,\,(A_1)_{D_l}^{\perp}=A_1\oplus D_{l-2}$
\item $A_1= \langle d_4 \rangle  \subset D_4, \,\,\,\,\,\,\,(A_1)_{D_4}^{\perp}=A_1^{\oplus 3}$

\item $A_1=\langle e_1 \rangle \subset E_p, \, \, p=6,7,8$
$$\begin{matrix}
(A_1)_{E_6}^{\perp}    &   =   & A_5 \\
(A_1)_{E_7}^{\perp}    &   =   & D_6 \\
(A_1)_{E_8}^{\perp}    &   =   & E_7 
\end{matrix}
$$
\item $A_2=\langle e_1, e_3 \rangle \subset E_p, \, \, p=6,7,8$
$$
\begin{matrix}
(A_2)_{E_6}^{\perp}    &   =   & A_2^{\oplus 2} \\
(A_2)_{E_7}^{\perp}    &   =   & A_5 \\
(A_2)_{E_8}^{\perp}    &   =   & E_6 
\end{matrix}
$$
\end{itemize}

\end{proposition}

\section{Specialized Weierstrass equations of $Y_{10}$ }
\subsection{Embeddings}

\begin{proposition}
The specialized elliptic fibrations of $Y_{10}$ have the same singular fibers and torsion as the generic ones. Their rank is equal to the corresponding generic one plus one, hence is bounded by $3$.

All the embeddings giving such fibrations can be derived from embedding $N\oplus A_1 \oplus A_2$ into the root lattices of the Niemeier lattices.

\end{proposition}

\begin{proof}
We get first the specialized Weierstrass equations of $Y_{10}$ from the generic ones given in Bertin-Lecacheux \cite{B-L}, Tables 3 and 4.

Then we deduce the primitive embeddings of $N\oplus A_1 \oplus A_2$ into the Niemeier lattices giving the corresponding elliptic fibration in Table 2. Recalling the elliptic fibrations in the generic case \cite{B-L} Table 2, we derive some observations.

\begin{table}\footnotesize

\begin{center}
\begin{tabular}{|c|c|c|c|c|c|c|}
\hline
$L_{\text{root}}$ & $L/L_{\text{root}}$ &  & & \text{type of Fibers} & \text{Rk} & \text{Tors.} \\ \hline
  $E_8^3$    &  $(0)$ & & & & &\\ \hline
& \#1  & $A_2\subset E_8$ & $(A_1 \oplus N) \subset E_8$  & $E_6 A_3 E_8$ & $1$ & $(0)$\\ \hline

& \#2  & $A_2 \oplus (A_1\oplus N) \subset E_8$ &   & $ E_8  E_8$ & $2$ & $(0)$\\ \hline
  $ D_{16}E_8$    &  $\mathbb Z /{2 \mathbb Z}$ & & & & &\\ \hline
&\#3   & $A_2\subset E_8$ & $(A_1 \oplus N) \subset D_{16}$  & $E_6 D_{11}$ & $1$ & $(0)$\\ \hline
& \#4  & $A_2 \oplus (A_1\oplus N) \subset E_8$ &   & $ D_{16}$ & $2$ & $\mathbb Z /{2 \mathbb Z}$\\ \hline
& \#5  & $(A_1 \oplus N)\subset E_8$ & $A_2 \subset D_{16}$  & $A_3  D_{13}$ & $2$ & $(0)$\\ \hline
& \#6  & $A_2 \oplus (A_1\oplus N) \subset D_{16}$ &   & $E_8  D_{8}$ & $2$ & $(0)$\\ \hline
  $ D_{10}E_7^2$    &  $(\mathbb Z /{2 \mathbb Z})^2$ & & & & &\\ \hline
& \#7  & $A_2\subset E_7$ & $(A_1 \oplus N) \subset D_{10}$  & $E_7 A_5 D_5$ & $1$ & $\mathbb Z /{2 \mathbb Z}$\
\\ \hline
&  \#8 & $A_2\subset E_7$ & $(A_1 \oplus N)\subset E_7$  & $ A_5 A_1 D_{10}$ &$2$ & $\mathbb Z /{2 \mathbb Z}$\\
 \hline
                                                                                                                
& \#9  & $A_2\oplus (A_1 \oplus N)\subset D_{10}$ &   & $E_7 E_7 A_1 A_1$ & $2$ & $\mathbb Z /{2 \mathbb Z}$\\ 
\hline
& \#10  & $(A_1 \oplus N)\subset E_7$ & $A_2 \subset D_{10}$  & $A_1 D_7 E_7$ & $3$ & $(0)$ \\ \hline
 $ A_{17}E_7$    &  $\mathbb Z /{6 \mathbb Z}$ & & & & &\\ \hline
&\#11   & $(A_1 \oplus N)\subset E_7$ & $A_2 \subset A_{17}$  & $A_1 A_{14}$ & $3$ & $(0)$\\ \hline
 $D_{24}$    &  $\mathbb Z /{2 \mathbb Z}$ & & & & &\\ \hline
& \#12  & $A_2\oplus (A_1 \oplus N) \subset D_{24}$ &   & $D_{16}$ & $2$ & $(0)$\\ \hline

 $D_{12}^2$    &  $(\mathbb Z /{2 \mathbb Z})^2$ & & & & &\\ \hline
&\#13   & $A_2 \subset D_{12}$ & $A_1 \oplus N\subset D_{12}$  & $D_{9} D_{7}$ & $2$ & $(0)$ \\ \hline
&\#14   & $A_2\oplus (A_1 \oplus N) \subset D_{12}$ &   & $D_4 D_{12}$ & $2$ & $\mathbb Z /{2 \mathbb Z}$\\ \hline
 $D_8^3$    &  $(\mathbb Z /{2 \mathbb Z})^3$ & & & & &\\ \hline
&\#15   & $A_2 \subset D_{8}$ & $(A_1 \oplus N) \subset D_{8}$  & $ D_{5}A_3 D_{8}$ & $2$ &$\mathbb Z/{2\mathbb\
 Z} $\\ \hline
& \#16  & $A_2\oplus (A_1 \oplus N)\subset D_{8}$ &   & $ D_8 D_{8}$ & $2$ & $\mathbb Z /{2 \mathbb Z}$\\ \hline
 $ A_{15}D_9$    &  $\mathbb Z /{8 \mathbb Z}$ & & & & &\\ \hline
& \#17  & $A_2\oplus (A_1 \oplus N) \subset D_{9}$ &   & $ A_{15}$ & $3$ &$\mathbb Z /{2 \mathbb Z}$ \\ \hline
& \#18  & $(A_1 \oplus N) \subset D_{9}$ & $A_2 \subset A_{15}$  & $D_4 A_{12}$ & $2$ & $(0)$\\ \hline
 $E_6^4$    &  $(\mathbb Z /{3 \mathbb Z)^2}$ & & & & &\\ \hline
& \#19  & $A_2 \subset E_6$ & $(A_1 \oplus N)\subset E_6$  & $A_2 A_2 E_6 E_6$ & $2$ & $\mathbb Z /{3 \mathbb Z\
}$\\ \hline

 $ A_{11} D_7E_6$    &  $\mathbb Z /{12 \mathbb Z}$ & & & & &\\ \hline
&\#20   & $A_2 \subset E_6$ & $(A_1 \oplus N)\subset D_7$  & $A_2 A_2 A_1 A_1 A_{11}$ & $1$ & $\mathbb Z /{6 \mathbb Z}$\\ \hline
& \#21  & $A_2 \subset A_{11}$ & $(A_1 \oplus N) \subset D_7$  & $A_8 A_1 A_1 E_6$ & $2$ & $(0)$\\ \hline
                                                                                                                
& \#22  & $A_2 \subset A_{11}$ & $(A_1 \oplus N)\subset E_6$  & $ A_8 D_7$ & $3$ & $(0)$\\ \hline
&\#23   & $(A_1 \oplus N)\subset E_6$ & $A_2 \subset D_7$  & $ A_{11}  D_4$ & $3$ &$\mathbb Z /{2 \mathbb Z}$ \\
 \hline
 $D_6^4$    &  $(\mathbb Z /{2 \mathbb Z})^4$ & & & & &\\ \hline
& \#24  & $A_2 \subset D_6$ & $(A_1 \oplus N) \subset D_6$  & $ A_3 D_6 D_6$ & $3$ &$\mathbb Z /{2 \mathbb Z}$ \
\\ \hline
 $ A_9^2D_6$    &  $\mathbb Z /{2}\times \mathbb Z /{10} $ & & & & &\\ \hline
& \#25  & $(A_1 \oplus N) \subset D_6$ & $A_2 \subset A_9$  & $ A_{6} A_9$ & $3$ & $(0)$\\ \hline

 $A_7^2D_5^2$    &  $\mathbb Z /{4}\times \mathbb Z /{8} $ & & & & &\\ \hline
&\#26   & $(A_1 \oplus N)\subset D_5$ & $A_2 \subset D_5$  & $A_1  A_{1} A_7 A_7$ & $2$  &$\mathbb Z /{4 \mathbb Z}$ \\ \hline
&\#27   & $(A_1 \oplus N)\subset D_5$ & $A_2 \subset A_7$  & $D_5 A_4 A_7 $ & $2$ & $(0)$\\ \hline

\end{tabular}
\end{center}
\caption{The specialized elliptic fibrations of $Y_{10}$}\label{Ta:Fib2}
\end{table}

 The specialized fibrations are all obtained by replacing a primitive embedding of $D_5$ in the generic case by a primitive embedding of $A_1 \oplus N$ in the same corresponding root lattice in the $Y_{10}$ case. Moreover the trivial lattices of the elliptic fibrations are the same. Since the Picard number is $19$ in the generic case and $20$ for $Y_{10}$, this explains why the rank of the specialized $Y_{10}$ always increases by $1$. This is a consequence of the following lemma.

\begin{lemma}
\begin{enumerate}[leftmargin=*]
\item Denote $R$ any root lattice $D_n$, $n\geq 5$, $E_6$, $E_7$ or $E_8$. There is a primitive embedding of $N\oplus A_1$ into $R$ such that $((N\oplus A_1)_{R}^{\perp})_{\text{root\
}}=((D_5)_{R}^{\perp})_{\text{root}}$.
\item There is a primitive embedding of $N\oplus A_1\oplus A_2$ into $E_8$ such that $((N\oplus A_1 \oplus A_2)_{E_8}^{\perp})_{\text{root\
}}=((A_2 \oplus D_5)_{E_8}^{\perp})_{\text{root}}=0$.

\end{enumerate}

\end{lemma}
\begin{proof}
\begin{enumerate}[leftmargin=*]
\item If $R=D_n$, we embed $N$ into $D_n$ as in lemma \ref{lem:D_n} and $A_1$ into the $A_2$ part of its orthogonal in $D_n$. That is precisely if $n\geq 7$, the primitive embedding
\[N\oplus A_1 =\langle d_n,d_{n-1}, d_{n-2}+d_{n-4},d_{n-3}+d_{n-4} \rangle \hookrightarrow D_n.\]
The following roots of $D_n$,
\[d_{n-6}, d_{n-7}, ..., d_1, d_n+d_{n-1}+2(d_{n-2}+d_{n-3}+...+d_{n-5})+d_{n-6}\]
are orthogonal to $N\oplus A_1$,
and generate $D_{n-5}$.
If $n=6$, take the embedding of lemma \ref{D_6}a).

\noindent For example, the rank $0$ elliptic fibration $\#20$ in \cite{B-L} Table $2$ ($2A_12A_2A_{11}$) comes from the following primitive embedding
\[D_5=\langle d_7,d_6,d_5,d_4,d_3\rangle \hookrightarrow D_7,\]
whose orthogonal into $D_7$ is $D_2\simeq 2A_1$ by \cite{Nis} p. 309, 310, 311.

\noindent And the rank $1$ specialized fibration $\#20$ in Table $2$ ($2A_12A_2A_{11}$) comes from the following primitive embedding
\[N\oplus A_1=\langle d_7,d_6 d_5+d_3,d_4+d_3 \rangle \hookrightarrow D_7,\]
whose orthogonal into $D_7$ is $\simeq 2A_1$ by lemma \ref{lem:D_7}(1).

\noindent If $R=E_6$, we embed $N$ as in lemma \ref{lem:E_6} (2). If $R=E_7$ we embed $N\oplus A_1$ as in \ref{3.4.10}. If $R=E_8$, we embed $N\oplus A_1$ as in lemma \ref{lem:unimod} (2).

\item It follows from the embedding given in lemma \ref{lem:unimod} (3). 
\end{enumerate}

 \end{proof}                                                                                                   \

\end{proof}

\subsection{Generators for specialization of \#16 fibration }

The rank of the  specialization for $k=10$  increases by one, so we have to
determine one more generator for the Mordell-Weil group. We give an example where
the computation is easy. 
We consider the Kummer surface 
$K_{10}=\text{Kum}(E1,E2)$ associated to $Y_{10}$, where 
$E1,E2$ have complex multiplication. 
Then using the method developped in \cite{Sh2} and \cite{K-U}, we determine a section on a $K_{10}$ fibration.
From \cite{B-L} Corollary 4.1, using a two-isogeny, we recover a section for  fibration \#16 of $Y_{10}.$  
From \cite{B-L} Corollary 4.1, the two elliptic curves $E_{1}$
and $E_{2}$ have respective  invariants $j_{1}=8000$ and $j_{2}=188837384000\pm
77092288000\sqrt{6}.$ Take
\[
E_{1}:Y^{2}=X\left(  X^{2}+4X+2\right)
\]
as a model of the first curve. The $2$-torsion sections have $X$-coordinates
$0,$ $-2\pm\sqrt{2},$ the $3-$ torsion sections have $X$-coordinates
$\frac{1}{3}\left(  1\pm i\sqrt{2}\right)  $ and $-1\pm\sqrt{6}$ that are roots of
$(3X^{2}-2X+1)(X^{2}+2X-5).$ The elliptic curve $E_{1}$ has complex
multiplication by $m_{2}=\sqrt{-2}$ defined by
\[
\left(  X,Y\right)  \overset{m_{2}}{\mapsto}\left(  -\frac{1}{2}\frac
{X^{2}+4X+2}{X},\frac{i\sqrt{2}}{4}\frac{Y\left(  X^{2}-2\right)  }{X^{2}%
}\right)
.\]
Let $C_{3}$ and $\widetilde{C_{3}}$ the two groups of order $3$ generated by the points of respective $X$-coordinates $\frac{1}{3}\left(  -2+i\sqrt
{2}\right)  $ and $\frac{1}{3}\left(  -2-i\sqrt{2}\right)  .$ These groups are fixed
by $m_{2}$ while the two order $3$ groups $\Gamma_{3}$ and $\widetilde{\Gamma_{3}}$  generated by the points of respective $X$-coordinates $-2+\sqrt
{6}$ and $-2-\sqrt{6}$ are exchanged by $m_{2}.$

If $M=\left(  X,Y\right)  $ is a general point on $E_{1}$, the $3$-isogenous
curve to $E_{1}$ of kernel $\Gamma_{3}$ is thus obtained with $X_{2}%
=X_{M}+\sum_{S\in B}X_{M+S}+k$ and $Y_{2}=Y_{M}+\sum_{S\in B}Y_{M+S}$ where
$k$ can be chosen so that the image of $\left(  0,0\right)  $ is $X_{2}=0.$
It follows the $3$-isogeny%

\[
w_{3}:X_{2}=\frac{X\left(  X-2-\sqrt{6}\right)  ^{2}}{\left(  X+2-\sqrt
{6}\right)  ^{2}},\quad Y_{2}=-\frac{Y\left(  X^{2}+\left(  8-2\sqrt
{6}\right)  X+2\right)  \left(  X-2-\sqrt{6}\right)  }{\left(  X+2-\sqrt
{6}\right)  ^{3}}%
\]
and and its $3$-isogenous curve $E_2$%

\begin{align*}
E_{2}  &  :Y_{2}^{2}=X_{2}^{3}+28X_{2}^{2}+\left(  98+40\sqrt{6}\right)
X_{2}\\
j\left(  E_{2}\right)   &  =188837384000-77092288000\sqrt{6}.%
\end{align*}

An equation for the Kummer surface $K_{10}$ is therefore%
\[
K_{10}:X\left(  X^{2}+4X+2\right)  =y^{2}X_{2}\left(  X_{2}^{2}+28X_{2}%
+98+40\sqrt{6}\right)
.\]

\subsubsection{Elliptic fibrations of $K_{10}$ and $Y_{10}$}

We consider the elliptic fibration of $K_{10}$%
\begin{align*}
K_{10}  &  \rightarrow\mathbb{P}^{1}\\
\left(  X,X_{2},y\right)   &  \mapsto t=\frac{X_{2}}{X}.%
\end{align*}
We use the following units of $\mathbb{Q}\left(  \sqrt{2},\sqrt{3}\right)  $%
\begin{align*}
r_{1}  &  =1+\sqrt{2}+\sqrt{6},\quad r_{1}^{\prime}=1-\sqrt{2}+\sqrt{6},\quad
r_{1}r_{1}^{\prime}=s=\left(  \sqrt{2}+\sqrt{3}\right)  ^{2}\\
r_{2}  &  =1+2\sqrt{3}+\sqrt{6},\quad r_{2}^{\prime}=1+2\sqrt{3}-\sqrt{6}.%
\end{align*}

Notice that $X_{2}=tX$ and a Weierstrass equation for this fibration is
obtained with the transformation%

\[
X=-\sqrt{2}\left(  1+\sqrt{2}\right)  \frac{X_{1}-2t\left(  t-r_{1}%
^{2}\right)  \left(  t-r_{2}^{2}\right)  }{X_{1}\left(  3+2\sqrt{2}\right)
-2t\left(  t-r_{1}^{2}\right)  \left(  t-r_{2}^{2}\right)  },\qquad
y=2\sqrt{2}\frac{X_{1}}{Y_{1}}%
\]%

\begin{equation}
K_{t}:Y_{1}^{2}=X_{1}\left(  X_{1}-2t\left(  t-r_{1}^{2}\right)  \left(
t-r_{1}^{\prime2}\right)  \right)  \left(  X_{1}-2t\left(  t-r_{2}^{2}\right)
\left(  t-r_{2}^{\prime2}\right)  \right)  \label{A}%
\end{equation}

where the singular fibers are in $t=0$ and $\infty$ of type $I_{2}^{\ast}$ and
at $t=r_{1}^{2},r_{2}^{2},r_{1}^{\prime2},r_{2}^{\prime2}$ of type $I_{2}.$
The rank is $2.$

\subsubsection{Sections on the Kummer fibration}

In many papers (\cite{Sh} \cite{Sh1} \cite{Sh2} Th 1.2. and \cite{Ku-Ku},
\cite{K-U}) results on the Mordell-Weil lattice of the Inose fibration are
given. We follow the same idea here, with the previous fibration of parameter
$t$.

We find a section on the previous fibration using $w_{3}\in \hom\left(
E_{1},E_{2}\right)  .$ The graph of $w_{3}$ on $E_{1}\times E_{2}$ and the 
image on $K_{10}=E_{1}\times E_{2}/\pm1$ correspond to $X_{2}=\frac{X\left(
X-2-\sqrt{6}\right)  ^{2}}{\left(  X+2-\sqrt{6}\right)  ^{2}}$ or
$t=\frac{\left(  X-2-\sqrt{6}\right)  ^{2}}{\left(  X+2-\sqrt{6}\right)  ^{2}%
}.$ If we consider the base-change of the fibration $u^{2}=t$ we obtain a
section defined by $u=\frac{\left(  X-2-\sqrt{6}\right)  }{\left(
X+2-\sqrt{6}\right)  }$ or $X=\frac{\left(  -2+\sqrt{6}\right)  \left(
u-s\right)  }{u-1},$ that is $P_{u}=\left(  X_{1}\left(  u\right)
,Y_{1}\left(  u\right)  \right)  $ on the Weierstrass equation $K_{u^{2}}$
\begin{align*}
X_{1}\left(  u\right)   &  =\frac{-1}{s_{2}^{2}s_{3}^{2}}u^{2}\left(
u^{2}-r_{2}^{2}\right)  \left(  u+r_{1}\right)  \left(  u-r_{1}^{\prime
}\right) \\
Y_{1}\left(  u\right)   &  =2uX_{1}\left(  u\right)  \left(  \left(  \sqrt
{3}-\sqrt{2}\right)  u^{2}+\left(  -2\sqrt{3}+\sqrt{2}\right)  u+\sqrt
{3}+\sqrt{2}\right),
\end{align*}
where $s_{2}=\frac{\sqrt{2}}{2}\left(  \sqrt{3}-1\right)  ,s_{3}=\sqrt{2}-1$ .
If $\widetilde{P_{u}}=\left(  X_{1}\left(  -u\right)  ,Y_{1}\left(  -u\right)
\right)  ,$ then $\widetilde{P_{u}}\in K_{u^{2}}$ and $P=\widetilde{P_{u}%
}+P_{u}\in K_{t}$, thus%
\begin{align*}
P  &  =\left(  x_{P},y_{P}\right) \\
x_{P}  &  =\frac{1}{s}\left(  t+s\right)  ^{2}\left(  t-r_{1}^{2}\right)
\left(  t-r_{1}^{\prime2}\right)  , y_{P}=x_{P}\frac{2-\sqrt{6}}%
{2}\left(  \frac{t-s}{t+s}\right)  \left(  t^{2}-14t-4\sqrt{6}t+s^{2}\right)
\end{align*}
so we recover a section $P$ on the fibration of the Kummer surface. \ \ 

\subsubsection{Sections on the fibration \#16 of $Y_{10}$}

The $2$-isogenous elliptic curve to (\ref{A}) in the isogeny of kernel
$\left(  0,0\right)  $ has a Weierstrass equation
\begin{align}
Y_{3}^{2}  &  =X_{3}\left(  X_{3}^{2}+8\,t\left(  {t}^{2}-28\,t+{s}%
^{2}\right)  X_{3}+64\,{\frac{{t}^{4}}{{s}^{2}}}\right) \label{B}\\
X_{3}  &  =\left(  \frac{Y_{1}}{X_{1}}\right)  ^{2},\qquad Y_{3}=\frac
{Y_{1}(B-X_{1}^{2})}{{X_{1}}}\nonumber
\end{align}
where $B$ is the coefficient of $X_{1}$ in (\ref{A}). Singular fibers are in
$t=0$ and $\infty$ of type $I_{4}^{\ast}$ and of type $I_{1}$ at $t=r_{1}%
^{2},r_{2}^{2},r_{1}^{\prime2},r_{2}^{\prime2}$ $.$

We can show that it is a fibration of $Y_{10}$ . More precisely, the
specialization of the fibration $\#16$ (\cite{B-L} Table 4) for $k=10$ has a
Weierstrass equation
\[
y^{2}=x^{3}+t_{0}\left(  4\left(  t_{0}^{2}+s^{2}\right)  +t_{0}\left(
s^{4}+14s^{2}+1\right)  x^{2}\right)  +16s^{6}t_{0}^{4}x
\]
with parameter $t_{0}$ . If $t=-t_{0}s^{2}$ and after scaling ($x=-s^{6}%
X_{3}/2,y=is^{9}Y_{3}/2^{3/2}$) this gives the equation (\ref{B}). The image of
$P$ by the isogeny, in the Weierstrass equation (\ref{B})\ is $Q=\left(
\xi,\eta\right)  $ with
\begin{align*}
\xi & =\frac{1}{2s}\frac{\left(  t^{2}-14t-4\sqrt{6}t+s^{2}\right)
^{2}\left(  t-s\right)  ^{2}}{(t+s)^{2}}\\
\eta & =-\frac{\left(  -2+\sqrt{6}\right)  }{4s}\frac{\left(  t^{2}%
-14t-4\sqrt{6}t+s^{2}\right)  \left(  t-s\right)  L_{t}}{(t+s)^{3}}\\
\text{where} \quad L_{t}  & =  t^{6}+2\left(  1+2\sqrt{6}\right)  t^{5}-\left(
993+404\sqrt{6}\right)  t^{4}+\left(  17820+7272\sqrt{6}\right)  t^{3}\\
 &-\left(
97137+39656\sqrt{6}\right)  t^{2}+\left(  56642+23124\sqrt{6}\right)
t+s^{6}.
\end{align*}

Recall that by specialization we have also a point $P^{\prime}=\left(
\xi^{\prime},\eta^{\prime}\right)  $ of $X_{3\text{ }}$-coordinate%
\begin{align*}
\xi^{\prime}  & =-8\frac{t^{3}\left(  t-1\right)  ^{2}}{\left(  t-s^{2}%
\right)  ^{2}}\\
\eta^{\prime}  & =\frac{i32}{19}\frac{\left(  5\sqrt{2}+2\sqrt{3}\right)
t^{4}\left(  19t^{2}-(326+140\sqrt{6})t+931+380\sqrt{6}\right)  }{\left(
t-s^{2}\right)  ^{3}}.%
\end{align*}
Finally the point $P^{\prime\prime}=P^{\prime}+\left(  0,0\right)  =\left(
\xi^{\prime\prime},\eta^{\prime\prime}\right)  $ satisfies%
\[
\xi^{\prime\prime}=\frac{-8t\left(  t-s^{2}\right)  ^{2}}{s^{2}\left(
t-1\right)  ^{2}}.%
\]
If $Q'=Q+(0,0),$ we verify that $ht(P^{\prime\prime}+Q^{\prime})=ht\left(
P^{^{\prime\prime}}-Q^{\prime}\right)  $ so $<P^{\prime},Q^{\prime}>=0$ and
$ht\left(  P^{\prime}\right)  .ht\left(  Q^{\prime}\right)  =18$. So by Shioda-Tate 
formula (\cite{Shio} \cite{Shio1}) $P'$ and $Q'$ generate the Mordell-Weil lattice.

\section{The extremal elliptic fibrations of $Y_{10}$}
\subsection{Embeddings}
The list of the extremal elliptic fibrations of $Y_{10}$ can be found in Shimada-Zangh paper \cite{Shim}. We shall keep Shimada-Zangh numbering and give the corresponding primitive embeddings of $M=N\oplus A_2\oplus A_1$ into Niemeier lattices in the following theorem.
\begin{theorem}

The extremal elliptic fibrations of $Y_{10}$ come from possible primitive embeddings of $M=N\oplus A_2\oplus A_1$ into the following Niemeier lattices with $L_{\text{root}}$
$$E_8^3, \,\, D_{16}E_8,\,\, D_{10}E_7^2,\,\,\, E_6^4,\,\, A_{11}D_7E_6.$$
Eight of them, namely fibrations number 80, 153, 200, 224, 252, 262, 292, 302 are obtained from primitive embeddings of $N$ into a root lattice while the three remaining come from an embedding of $N$ into $Ni(D_{16}E_8)$, namely number 87 and 241 or into $Ni(A_{11}D_7E_6)$, namely number 8.
They are listed below with their Mordell-Weil groups:
$$
\begin{matrix}
\text{Number }& \text{Singular fibers} & \text{Torsion} & \text{from} \\
292 & A_2+A_3+E_6+E_7 &  (0)&  E_8^3\\
302 & A_2 +A_3+A_5+E_8 & (0)&  E_8^3\\
87 & A_1+A_1+A_5+A_{11} &  \mathbb Z/(2)& Ni(D_{16}E_8)\\
241 & A_1+A_{11}+E_6 &   (0) & Ni(D_{16}E_8)\\
200 & A_2+A_5+D_{11}  &  (0) &  D_{16}E_8\\
252 & A_1+A_2+D_9+E_6 &  (0)&  D_{16}E_8 \\
153 & A_2+A_5+D_5+D_6 &  \mathbb Z/(2)&  D_{10}E_7^2\\
262 & A_1+A_2+A_3+A_5+E_7 &  \mathbb Z/(2) &D_{10}E_7^2\\
224 & A_2+A_2+A_3+A_5+E_6 &  \mathbb Z/(3)& E_6^4 \\
80 & A_1+A_2+A_2+A_2+A_{11} &  \mathbb Z/(3)& A_{11}D_7E_6\\

8 & A_1+A_2+A_2+A_3+A_5+A_5 & \mathbb Z/(6) & Ni(A_{11}D_7E_6)
\end{matrix}
$$
\end{theorem}

\begin{proof}

\begin{itemize}[font=\bf, leftmargin=*] 
\item [Fibration 292]  ($A_2A_3E_6E_7$)
Embed primitively
$N$ into $E_8^{(1)}$ as in lemma \ref{lem:unimod} $(1)$, $A_1$ into $E_8^{(2)}$ and $A_2$ into $E_8^{(3)}$ as in Nishiyama \cite{Nis} Prop. 3.2. Since for these embeddings, $\det (N\oplus A_1 \oplus A_2)_{E_8^3}^{\perp}=\det( A_2 \oplus A_3 \oplus E_6 \oplus E_7)=72$, the Mordell-Weil group is equal to $(0)$. 
\item [Fibration 302] ($A_2A_3A_5E_8$)
Embed primitively
$N$ into $E_8^{(1)}$ as in lemma \ref{lem:unimod} $(1)$,  and $A_1 \oplus A_2$ into $E_8^{(2)}$ as in Nishiyama \cite{Nis} p. 332. Since there is a fiber of type $E_8$, the Mordell-Weil group is equal to $(0)$.
\begin{remark}
There is no other possible extremal primitive embedding into $E_8^3$ by lemma \ref{L:nik0}.
\end{remark}
\item [Fibration 87 and 241]
They follow from a primitive embedding of $N$ into $Ni(D_{16}E_8)$ given in the lemma below.
\begin{lemma}\label{L:D16}
There is a primitive embedding of $N$ into $Ni(D_{16}E_8)$ whose root part of its orthogonal in $D_{16}$ contains $A_{11}\oplus 2A_1$.
\end{lemma}
\begin{proof}
The glue code of $Ni(D_{16}E_8)$ is generated by $([1],0)$ cf. Table 1. The norm 4 vector $v=([3],0)$ and the norm 2 vectors $v_1=(d_{15},0)$, $v_2=(d_1,0)$ define a primitive embedding of $N$ into $Ni(D_{16}E_8)$. Indeed, $[3]=-d_{15}^*+d_1^*$. The primitivity follows from the fact that in the basis 
$[3],d_1, d_i,3 \leq i \leq 16$ of $D_{16}^*$ the matrix of the embedding has for maximal minor the matrix identity.
Moreover $(\langle -d_{15}^*+d_1^*,d_{15},d_1 \rangle_{D_{16}}^{\perp})_{\text{root}}$ contains $\langle d_3,d_4,d_5,d_6,d_7,d_8,d_9,d_{10},d_{11},d_{12},d_{13}\rangle \oplus \langle d_{16} \rangle\oplus \langle d_{16}+d_{15}+2d_{14}+...+2d_2+d_1 \rangle= A_{11}\oplus A_1 \oplus A_1$.
\end{proof}
\item [Fibration 87 ] ($2A_1A_5A_{11}$)
It is obtained from the embedding of $N$ as in the previous lemma and $A_1\oplus A_2$ into $E_8$ as in Nishiyama \cite{Nis} p. 332.
 Since for these embeddings, $\det (N\oplus A_1 \oplus A_2)_{E_8\oplus D_{16}}^{\perp}=\det (A_1 \oplus A_1\oplus A_{11}\oplus A_5)=72\times 4$, the Mordell-Weil group is equal to $\mathbb Z/(2)$.
\item [Fibration 241] ($A_1A_{11}E_6$)
The lattice $N$ being embedded as in lemma \ref{L:D16}, we embed $A_1=d_{16}$ in $D_{16}$ and $A_2$ in $E_8$ as in Proposition \ref{PNis}.

 Since for these embeddings, $\det (N\oplus A_1 \oplus A_2)_{E_8\oplus D_{16}}^{\perp}=\det( A_1 \oplus A_{11}\oplus E_6)=72$, the Mordell-Weil group is equal to $(0)$.
\item [Fibration 200] ($A_2A_5D_{11}$)
We embed $N$ into $D_{16}$ as in lemma \ref{lem:D_n} and  $A_1\oplus A_2$ into $E_8$ as in Nishiyama \cite{Nis} p. 332.

Since for these embeddings, $\det (N\oplus A_1 \oplus A_2)_{E_8\oplus D_{16}}^{\perp}=\det A_2 \oplus A_5 \oplus D_{11}=72$, the Mordell-Weil group is equal to $(0)$.
\item [Fibration 252] ($A_1A_2D_9E_6$)
We embed $N\oplus A_1$ into $D_{16}$ as $\langle d_{16},d_{15},d_{14}+d_{12},d_{10}\rangle$. A direct computation gives the orthogonal $A_1 \oplus D_9\oplus A_2$. We complete by embedding $A_2$ into $E_8$ with orthogonal $E_6$ as in Proposition \ref{PNis}.
Since for these embeddings, $\det (N\oplus A_1 \oplus A_2)_{E_8\oplus D_{16}}^{\perp}=\det (A_1 \oplus A_2 \oplus D_9 \oplus E_6)=72$, the Mordell-Weil group is equal to $(0)$.

\item [Fibration 153] ($A_2 A_5 D_5 D_6$ )
We embed $N$ in $D_{10}$ as in lemma \ref{lem:D_n}, $A_1$ in $E_7^{(1)}$ and $A_2$ in $E_7^{(2)}$ as in Proposition \ref{PNis}.
 Since for these embeddings, $\det (N\oplus A_1 \oplus A_2)_{E_7^2\oplus D_{10}}^{\perp}=\det (A_2 \oplus A_5 \oplus D_5 \oplus D_6)=4 \times 72$, the Mordell-Weil group is equal to $\mathbb Z/(2)$.

\item [Fibration 262] ($A_1A_2A_3 A_5 E_7$ )
We embed  $N\oplus A_1$ in $D_{10}$ as $\langle d_{10},d_9,d_8+d_6,d_4 \rangle$ and  $A_2$ in $E_7^{(1)}$. Since for these embeddings, $\det (N\oplus A_1 \oplus A_2)_{E_8\oplus D_{16}}^{\perp}=\det (A_1 \oplus A_2 \oplus A_3 \oplus A_5 \oplus E_7)=4 \times 72$, the Mordell-Weil group is equal to $\mathbb Z/(2)$.
\item [Fibration 224] ($2A_2A_3 A_5 E_6$ )
We embed $N$ into $E_6^{(1)}$ as in lemma \ref{lem:E_6} (1), $A_1$ into $E_6^{(2)}$, $A_2$ into $E_6^{(3)}$ as in Proposition \ref{PNis}.

Since for these embeddings, $\det (N\oplus A_1 \oplus A_2)_{E_6^4}^{\perp}=\det (A_2 \oplus A_2 \oplus A_3 \oplus A_5 \oplus E_6)=9 \times 72$, the Mordell-Weil group is equal to $\mathbb Z/(3)$.

\item [Fibration 80] ($A_13A_2A_{11}$ )
We embed $N\oplus A_1$ into $D_7$ as in lemma \ref{lem:D_7} (1) and $A_2$ into $E_6$ as in Proposition \ref{PNis}. 
Since for these embeddings, $\det (N\oplus A_1 \oplus A_2)_{A_{11} \oplus E_6 \oplus D_7}^{\perp}=\det (A_1 \oplus A_2^3 \oplus A_{11})=9 \times 72$, the Mordell-Weil group is equal to $\mathbb Z/(3)$.
\item [Fibration 8] ($A_12A_2A_{3}2A_5$ )
\begin{lemma}
There is a primitive embedding of $N\oplus A_2 \oplus A_1$ into $Ni(A_{11}\oplus D_7 \oplus E_6)$ giving the rank $0$
elliptic fibration $A_1 2A_2 A_3 2A_5 $.
\end{lemma}
\begin{proof}
For $Ni(A_{11}D_7E_6)$, $L/L_{\text{root}}$ is generated by the class of the glue vector $g=[[1],[1],[1]]$. In the class  of $6g=[[6],[2],0]$, take the vector
\[ v=(((1/2)^{6},(-1/2)^6),(0^6,1),0)=(a_6^*,d_1^*,0) \]
of norm $4$ and the vectors
\[v_1=(0,d_1,0)\]
\[v_2=(0,(0,0,0,0,0,-1,-1),0)\]
of norm $2$. These vectors realize a primitive embedding of $N$ into $Ni(A_{11}D_7E_7)$ whose root part of its orthogonal in $Ni(A_{11}D_7E_7)$ is $2A_5\oplus D_5 \oplus E_6$ with $D_5=\langle d_7,d_6,d_5,d_4,d_3 \rangle$. Moreover, we embed $A_1=d_7$ into $D_5$ and $A_2$ into $E_6$ as in Proposition \ref{PNis}. From Shimada-Zangh \cite{Shim}, an extremal fibration with singular fibers of type $A_12A_2A_{3}2A_5$ is the fibration $8$ with $6$-torsion.
\begin{remark}

Embedding $A_1$ into $E_6$ and $A_2$ into $D_5$ gives the rank $1$ fibration $3A_5 2A_1$. (obtained also as the 2-isogenous of the specialisation of $\# 20$, hence a fibration of $Y_{10}$ of rank $1$).

Embedding $A_1 \oplus A_2$ into $E_6$, we find the rank $1$ fibration $A_2 2A_5 D_5$ .

Embedding $A_1 \oplus A_2$ into $D_5$, we find the rank $1$ fibration $A_1 2A_5E_6$.

All these embeddings are primitive.
\end{remark}
\end{proof}
\end{itemize}
\end{proof}

\subsection{Weierstrass equations of the extremal fibrations of $Y_{10}$}
Recall the following computations.
Let $Y_{k}$ the surface defined by the Laurent polynomial
\[                                                                                                              
X+\frac{1}{X}+Y+\frac{1}{Y}+Z+\frac{1}{Z}=k                                                                     
\]
We consider the elliptic fibration defined with the parameter $t=X+Y+Z-\frac                                    
{k}{2}.$ We obtain a Weierstrass equation with the transformation
\[                                                                                                              
X=\frac{y-x^{2}}{y\left(  \frac{k}{2}-t\right)  },\quad Y=\frac{y+x}{x\left(                                    
\frac{k}{2}-t\right)  },\quad Z=-X-Y+t+\frac{k}{2}%
\]
of inverse on $Y_{k}$%
\[                                                                                                              
x=-\frac{\left(  Z+Y\right)  \left(  Z+X\right)  }{Z^{2}},\quad y=-\frac                                        
{Y\left(  Z+Y\right)  \left(  Z+X\right)  ^{2}}{XZ^{3}},\quad t=X+Y+Z-\frac                                     
{k}{2}%
\]%
\begin{align*}                                                                                                  
E_{\#20}  &  :y^{2}+\left(  t^{2}+3-\frac{k^{2}}{4}\right)  yx+\left(                                           
t^{2}+1-\frac{k^{2}}{4}\right)  y=x^{3}\\                                                                       
&  I_{12}\left(  \infty\right)  ,\quad2I_{3}\left(  t^{2}+1-\frac{k^{2}}%
{4}\right)  ,\quad2I_{2}\left(  \pm\frac{k}{2}\right)  ,\quad2I_{1}\left(                                       
t^{2}+9-\frac{k^{2}}{4}\right).                                                                                  
\end{align*}
The generic rank is $0$, the point $\left(  -1,1\right)  $ is of order $2,$
and $\left(  0,0\right)  $ of order $3.$

For $k=10$ the equation becomes
\[                                                                                                              
y^{2}+\left(  t^{2}-22\right)  yx+\left(  t^{2}-24\right)  y=x^{3},%
\]
with rank  $1$ and a generator of the Mordell-Weil group 
\begin{align*}                                                                                                  
P  &  =\left(  -1-\frac{1}{432}t^{2}\left(  t^{2}-9\right)  ^{2},\frac                                          
{-1}{15552}i\sqrt{3}\left(  t^{2}+i\sqrt{3}t-12\right)  ^{3}\left(                                              
t+i\sqrt{3}\right)  ^{3}\right) \\                                                                              
&  =\left(  -\frac{1}{432}\left(  t^{2}+3\right)  \left(  t^{4}-21t^{2}%
+144\right)  ,\frac{-1}{15552}i\sqrt{3}\left(  t^{2}+i\sqrt{3}t-12\right)                                       
^{3}\left(  t+i\sqrt{3}\right)  ^{3}\right).                                                                     
\end{align*}
\begin{remark}
Using method of section 8 we can recover from this fibration the result: $T(Y_{10})=T(Y_2)[3]$ \cite{B1}, \cite{B2}.
\end{remark}
 
\subsection{Fibrations using the section $P$}
From the previous Weierstrass equation and the section $P$ we construct two fibrations.

\subsubsection{Fibration $E_{m}$ with singular fibers $I_{10},IV,2I_{3},2I_{2}$                                    
and rank 1}

We use the parameter
\[                                                                                                              
m=-\frac{i\sqrt{3}}{9}\frac{Y-Y_{P}}{X-X_{P}}+\frac{11i\sqrt{3}}{9}-\frac                                       
{1}{108}t\left(  t-3\left(  2-i\sqrt{3}\right)  \right)  \left(  t+3\left(                                      
2+i\sqrt{3}\right)  \right)                                                                                     
\]
and obtain the equation
\begin{align*}                                                                                                  
E_{m}  &  :Y^{2}=X^{3}+\left(  3t^{4}-20t^{2}+16\right)  X^{2}+t^{4}\left(                                      
t^{2}-1\right)  \left(  3t^{2}+8\right)  X+t^{8}\left(  t^{2}-1\right)  ^{2}
\end{align*}                                  
with singular fibers $  IV\left(  \infty\right)  ,$ $ I_{10}\left(  0\right)  ,$ $2I_{3}\left(                                   
27t^{2}-32\right)$,  $2I_{2}\left(  \pm1\right) $.

\begin{itemize}[font=\bf, leftmargin=*]
\item[Fibration $80$]
We use the parameter $n$
\[                                                                                                              
i\sqrt{3}n=\frac{1}{\left(  t+5\right)  }\frac{Y-Y_{P}}{X-X_{P}}-\frac{\left(                                   
11-3i\sqrt{3}\right)  }{\left(  t+5\right)  }-\frac{i\sqrt{3}}{36}%
\frac{\left(  t-6+3i\sqrt{3}\right)  \left(  t+6+3i\sqrt{3}\right)  }{t+5}%
\]
and obtain the Weierstrass equation of the new elliptic fibration $80$ of rank $0$
with $3$-torsion.%

\begin{align*}                                                                                                  
E_{n}  &  :Y^{2}+\left(  9t^{2}+6t-9\right)  YX+9t^{4}\left(  3t^{2}%
+6t-5\right)  Y=X^{3}
\end{align*}                                                                                         
with singular fibers $  I_{3}\left(  \infty\right)  ,$ $I_{12}\left(  0\right)$,                                             
$2I_{3}\left(  3t^{2}+6t-5\right) $,  $I_{2}\left(  \frac{3}{5}\right)                                        
$ $I_{1}\left(  -\frac{3}{4}\right)$.


\item [Fibration $262$]
From the previous Weierstrass equation $E_{n}$ of fibration $80$ and with the
parameter $u=\frac{X}{t^{4}}$ we obtain
\begin{align*}                                                                                                  
E_{u}  &  :Y^{2}-2\left(  t+9\right)  YX=X^{3}+9\left(  t+3\right)  \left(                                      
t+5\right)  X^{2}-t^{3}\left(  t+5\right)  ^{2}X\\                                                              
&  III^{\ast}\left(  \infty\right)  ,\quad I_{6}\left(  0\right)  ,\quad                                        
I_{4}\left(  -5\right)  ,\quad I_{3}\left(  -9\right)  ,\quad I_{2}\left(                                       
-4\right)                                                                                                       
\end{align*}
with rank $0$ and a $2$-torsion section $\left(  0,0\right)  $.

\begin{remark}
This fibration can also be obtained from $E_{n}$ with the parameter $u_{1}%
=\frac{Y}{X\left(  3t^{2}+6t-5\right)  }$; the singular fiber $I_{6}$ is for
$u_{1}=\infty$ and $III^{\ast}$ for $u_{1}=0.$
\end{remark}

\item [Fibration $292$]

From the Weierstrass equation $E_{u}$ of fibration $262$ and with the
parameter $w=8\left(  \frac{X}{t^{3}\left(  t+4\right)  }-\frac{1}{8\left(                                      
t+4\right)  }\right)  $ we obtain
\begin{align*}                                                                                                  
E_{w}  &  :Y^{2}=X^{3}+2t^{2}\left(  3t+5\right)  X^{2}+t^{3}\left(                                             
12t+1\right)  \left(  t-1\right)  ^{2}X+8t^{5}\left(  t-1\right)  ^{4}
\end{align*}                                        
with singular fibers $ IV^{\ast}\left(  \infty\right)$,  $ III^{\ast}\left(  0\right)$, $                                    
I_{4}\left(  1\right)$, $ I_{3}\left(  -\frac{1}{27}\right)$.                                                 
\item [Fibration $252$ ]
From fibration $262$ with the Weierstrass equation $E_{u}$ and the parameter
$\frac{X}{t^{3}}$ we obtain
\begin{align*}                                                                  
E_{t}  &  :Y^{2}=X^{3}+t\left(  10t-1\right)  X^{2}+10t^{4}\left(               
9t-1\right)  X+t^{7}\left(  216t-25\right) 
\end{align*}                                  
with singular fibers $ IV^{\ast}\left(  \infty\right)$, $I_{5}^{\ast}\left(  0\right)$, $  
I_{3}\left(  \frac{4}{27}\right)$, $I_{2}\left(  \frac{1}{8}\right).$        

\item [Fibration $302$]
We start from the \ equation $E_{u}$ of fibration $262$ and parameter
$r=\frac{X-2t}{2\left(  t+4\right)  }$ and obtain the Weierstrass equation
\begin{align*}                                                                                                  
Y^{2}  &  =X^{3}+\left(  -14t^{2}+\frac{63}{2}t+\frac{27}{2}\right)                                             
X^{2}+14t^{3}\left(  3t-19\right)  X-2t^{6}\left(  t-7\right) 
\end{align*}                                                
with singular fibers $ II^{\ast}\left(  \infty\right)$, $I_{6}\left(  0\right),$                                          
$I_{4}\left(  -5\right)$, $I_{3}\left(  9\right)$, $I_{1}\left(                                        
\frac{-7}{27}\right).$                                                                                            
\item [Fibration $200$ ]
We start with the previous fibration $302$, do the translation
$X=250-\frac{175}{2}\left(  t+5\right)  +X_{1}$ and obtain the new Weierstrass
equation
\[                                                                                                              
E_{r}:Y^{2}=X_{1}^{3}+A\left(  t\right)  X_{1}^{2}+B\left(  t\right)  \left(                                    
t+5\right)  ^{2}X_{1}+C\left(  t\right)  \left(  t+5\right)  ^{4}%
\]
where $A,B,C$ are polynomials of degree respectively $2,2,3.$ Then the new
parameter $s=\frac{X_{1}}{\left(  t+5\right)  ^{2}}$ gives the fibration
$200.$ A \ Weierstrass equation is obtained with $s=\frac{15}{2}+t$%
\begin{align*}                                                                                                  
E_{s}&:Y^{2}=X^{3}+(  t^{3}+\frac{17}{2}t^{2}+\frac{3}{4}t+\frac                                       
{27}{8})  X^{2}-t\left(  10t+9\right)  \left(  2t-3\right)                                                
X+2t^{2}\left(  27+50t\right) 
\end{align*}
with singular fibers                                                                                 
$I_{7}^{\ast}\left(  \infty\right),$  $I_{6}\left(  0\right)$,                                       
 $I_{3}\left(  \frac{-9}{4}\right)$, $2I_{1}\left(  4t^{2}+44t-7\right).$                                       
\item [Fibration $153$ of rank $0$]
From fibration $252$ we take the parameter $j=-\frac{X+4t^{3}}{t^{2}\left(                                      
8t-1\right)  }$ and obtain the fibration $153$ with Weierstrass equation
\begin{align*}                                                                                                  
Y^{2} &  =X^{3}+t\left(  t^{2}+10t-2\right)  X^{2}+t^{2}\left(  2t+1\right)                                     
^{3}X
\end{align*}                                                                                                         
with singular fibers  $  I_{2}^{\ast}\left(  \infty\right)$, $ I_{1}^{\ast}\left(  0\right)$                                       
, $I_{6}\left(  \frac{1}{2}\right)  ,$ $I_{3}\left(  4\right).$                                            
\item [Fibrations $8$ and $224$]
We consider the two elliptic fibrations of rank $0$ of $Y_{2}$ with Weierstrass
equations
\begin{align*}                                                                                                  
E_{w}  & :Y^{2}-\left(  t^{2}+2\right)  YX-t^{2}Y=X^{3}\\                                                       
E_{j}  & :Y^{2}-t\left(  t+4\right)  YX+t^{2}Y=X^{3}.%
\end{align*}
In these two cases the section $\left(  0,0\right)  $ is a $3$-torsion section.

Using  results of section $8$, the two curves $H_{w}$ and $H_{j}$ respectively $3$-isogenous
to $E_{w}$ and $E_{j}$ have have the following Weierstrass equations and
singular fibers %

\begin{align*}                                                                                                  
\text{Fibration }8\quad H_{w} &  :Y^{2}+3(t^{2}+2)YX+\left(  t^{2}+8\right)                                     
\left(  t-1\right)  ^{2}\left(  t+1\right)  ^{2}Y=X^{3}\\                                                       
&  2I_{6}\left(  1,-1\right)  ,\quad I_{4}\left(  \infty\right)  ,\quad                                         
2I_{3}\left(  t^{2}+8\right)  ,\quad I_{2}\left(  0\right)  .\\                                                 
\text{Fibration }224\quad H_{j} &  :Y^{2}+3t\left(  t+4\right)  YX+\allowbreak                                  
t^{2}\left(  t^{2}+10t+27\right)  \left(  t+1\right)  ^{2}Y=X^{3}\\                                             
&  VI^{\ast}\left(  0\right)  ,\quad I_{6}\left(  -1\right)  ,\quad                                             
I_{3}\left(  t^{2}+10t+27\right)  ,\quad I_{4}\left(  \infty\right)  .                                          
\end{align*}

Since they have also rank $0$, the discriminants of their transcendental
lattice, computed with the singular fibers and torsion sections are equal to
$72.$ From Shimada-Zhang table \cite{Shim}, only one fibration corresponds to the two
previous types of singular fibers and torsion and comes from $Y_{10}.$

\begin{remark}
These two results can be also recovered with the following method.

For the fibration $8$, we put $Y=-(8+t^{2})X+Z$ in the equation $H_{w}$ and
have the following Weierstrass equation %
\begin{align*}                                                                                                              
Z^{2}+\left(  t^{2}-10\right)  YX+\left(  t^{2}-1\right)  ^{2}\left(                                            
t^{2}+8\right)  ^{2}Y\\
=X^{3}+2\left(  t^{2}-1\right)  \left(  t^{2}+8\right)                                     
X^{2}+\left(  t^{2}-1\right)  ^{2}\left(  t^{2}+8\right)  ^{2}X .                                                
\end{align*}
Using the parameter $m=\frac{Z}{\left(  t-1\right)  ^{2}\left(  t^{2}%
+8\right)  ^{2}}$ and by the $2$-neighbor method we recover the fibration
$252.$ So the fibration $H_{w}$ is a fibration of $Y_{10.}$

For the fibration $224$, we put $Y=Y^{\prime}+\frac{1}{2}\left(  t^{2}%
+10t+27\right)  X$ in the equation $H_{j}$ and \ obtain the equation
\begin{align*}                                                                                                  
&  Y^{\prime2}+\left(  4t^{2}+22t+27\right)  Y^{\prime}X+t^{2}\left(                                            
t+1\right)  ^{2}+\left(  t^{2}+10t+27\right)  Y\\                                                               
&  =X^{3}-\frac{1}{4}\left(  7t+27\right)  \left(  t+1\right)  \left(                                           
t^{2}+10t+27\right)  X^{2}-\frac{1}{2}t^{2}\left(  t+1\right)  ^{2}\left(                                       
t^{2}+10t+27\right)  ^{2}X.                                                                                      
\end{align*}
Then we take $m=\frac{Y^{\prime}}{27\left(  t+1\right)  ^{2}\left(                                              
t^{2}+10t+27\right)  ^{2}}$ as new parameter, and obtain the fibration $292.$
\end{remark}
\item [Fibration $87$]
We consider the Weierstrass equation
\[                                                                                                              
Y^{2}=x^{3}+\left(  -wt^{4}-bt^{3}-ct^{2}-et+h\right)  x^{2}+\left(                                             
-rt^{2}-tm-a\right)  x                                                                                          
\]
with singular fibers $I_{12}\left(  \infty\right)  ,2I_{2}\left(  \left(                                        
rt^{2}+tm+a\right)  \right)  ,8I_{1}\left(  P\left(  t\right)  \right)  .$ The
polynomial $P\left(  t\right)  $ is of degree $8$ depending of $w,b,c,e,h...$ We
choose these coefficients in such a way to get a singular fiber of type
$I_{6}$ at $t=0,$ that is the $i-$coefficient of $P$ equal to $0$ for $i$ from
$0$ to $5.$ There is only two solutions, one of them is the fibration $7-w$ of
$Y_{2},$ the other corresponds to the Weierstrass equation%
\[                                                                                                              
E:Y^{2}=X^{3}-(9t^{4}+9t^{3}+6t^{2}-6t+4)X^{2}+(+21t^{2}-12t+4)X .                                               
\]
The discriminant of the transcendental lattice of the surface with this
elliptic fibration is $72$ or $8.$ The second case occurs if there is a $3$-torsion section and therefore this is an elliptic fibration of $Y_{2}.$ But,
using the equation of $E$ and the parameter   $m=\frac{Y}{3Xt^{2}}$, with
the change $X=2-3t+t^{2}W$, we obtain an equation in $W$ and $t$ of bidegree
$2$, then another elliptic fibration with a Weierstrass equation
\begin{align*}                                                                                                  
Y^{2}+3YX-9\left(  t^{2}+1\right)  ^{2}Y &  =X\left(  X^{2}-6\left(                                             
t^{2}+1\right)  X-9\left(  2t^{2}-1\right)  \left(  t^{2}+1\right)                                              
^{2}\right) 
\end{align*}
and singular fibers                                                                                                  
$ I_{0}^{\ast}\left(  \infty\right)$, $2I_{6}\left(  \pm I\right)$, $                                  
I_{2}\left(  0\right)$, $4I_{1}.$%

Since the rank is $3$, this elliptic fibration cannot come from $Y_{2},$ so the
discriminant of the transcendental lattice is $72.$ From the singular fibers
and results of Shimada-Zhang \cite{Shim}, this fibration comes from $Y_{10}$ and is fibration $87.$
\begin{remark}
We can also start from $E$ and with the parameter $m=3\frac{X-\left(                                            
2-3t+3t^{2}\right)  }{t^{3}}$ construct another elliptic fibration
with Weierstrass equation
\begin{align*}                                                                                                  
F &  :Y^{2}=X^{3}-\left(  3t^{4}-18t^{3}+15t^{2}+27t+9\right)  X^{2}\\%
+3t^{4}\left(  t^{2}-3t-2\right)  \left(  t^{2}-9t+21\right)  X
 & -t^{8}\left(                                     
t^{2}-9t+21\right)  ^{2}
\end{align*}                                                                                      
with singular fibers $  IV^{\ast}\left(  \infty\right)$, $ I_{10}\left(  0\right)$,                                         
$2I_{2}\left(  t^{2}-9t+21\right)$, $2I_{1}.$ %
                                                                 
From $F$ and the parameter $m=\frac{X}{t^{4}\text{ }}$ we have another
elliptic fibration with Weierstrass equation                                                                  
\begin{tabular}{l}                                                                                                  
$Y^{2}+18(2-t)XY+486t^{2}(t+2)Y  =X^{3}+6\left(  7t^{2}+76t-56\right)                                          
X^{2}$\\
$+9t^{2}\left(  t+8\right)  (64t^{2}+241t-224)X$
\end{tabular}                                                         
and singular fibers  $IV^{\ast}\left(  1\right),$ $III^{\ast}\left(  \infty\right),$                                     
$I_{4}\left(  0\right)  ,$  $I_{2}\left(  -7\right)$, $I_{1}\left(                                        
\frac{13}{256}\right) $.                                                                                          
This elliptic fibration can also be obtained from fibration $262$ with the
previous Weierstrass equation and parameter $\frac{X}{t^{3}\left(  t+9\right)                                    
}-\frac{4}{27}\frac{1}{t+9}.$
\end{remark}

\item [Fibration $241$ ]
We start with a general equation
\[                                                                                                              
Y^{2}=X^{3}+p\left(  t\right)  t^{2}X^{2}+q\left(  t\right)  t^{3}X+r\left(                                     
t\right)  t^{4}%
\]
where $p,q,r$ are general polynomials of degree $2.$ So there is singular fiber
at $t=0$ of type $IV^{\ast}.$ We choose the coefficients of $p,q,r$ in such a
way to get a singular fiber of type $I_{12}$ at $t=\infty$ and $I_{2}$ at
$t=-1.$ We find                                                                                                               
\begin{align*}                                                                                                  
Y^{2}  &  =X^{3}+2\left(  128t^{2}+8\left(  11+i\sqrt{3}\right)  t+\left(                                       
17-i\sqrt{3}\right)  \right)  t^{2}X^{2}\\                                                                      
&  +\frac{1}{3}\left(  -3+i\sqrt{3}\right)  \left(  384t^{2}+8\left(                                            
30+4i\sqrt{3}\right)  t+33-i\sqrt{3}\right)  t^{3}X\\                                                           
&  -2\left(  -1+i\sqrt{3}\right)  \left(  48t^{2}+\left(  27+5i\sqrt                                            
{3}\right)  t+2\right)  t^{4}.%
\end{align*}
An extremal fibration of discriminant $72$ comes from a fibration of $Y_{10}$ without 
torsion or a fibration of $Y_2$ with a $3$-torsion section. To conclude, we use a $2$-neighbor fibration. 
From the previous Weierstrass equation we have the following fibration with parameter $\frac{X}{t}$ and 
singular fibers $2I_{6}\left(  0,t_{1}\right)  ,IV^{\ast}\left(  \infty\right)                                     
,I_{2},2I_{1}$ .                                                         
After scaling, a Weierstrass equation of the  fibration  take the
following shape%
\begin{align*}                                                                                                  
Y^{2} &  =X^{3}-\frac{1}{73}\left(  -17+i\sqrt{3}\right)  \left(                                                
146t^{2}+\left(  -145+43i\sqrt{3}\right)  t+20-16i\sqrt{3}\right)  X^{2}\\                                      
&  +\frac{1}{31}\left(  11+i\sqrt{3}\right)  \left(  4t-3+i\sqrt{3}\right)                                      
\left(  124t-85+19i\sqrt{3}\right)  t^{3}X\\                                                                    
&  +16\left(  4t-3+i\sqrt{3}\right)  ^{2}t^{6}.%
\end{align*}
We have the section  $P=\left(  0,4t^{3}\left(  4t-3+i\sqrt{3}\right)                                    
\right)  .$ Computing the height of $P$ \ gives the discriminant of the
transcendental lattice equal to $72$ if $P\neq3S$ generates the Mordell-Weil
group. If $P=3S$, then the discriminant of the transcendental lattice should be equal to
$8$ and this fibration should be a fibration of $Y_{2},$ which is impossible since not in the list of all the elliptic fibrations of $Y_2$ \cite{BL}. So
this is a fibration of $Y_{10}$ since it is the only one in Shimada-Zhang's table\cite{Shim}.
\end{itemize}
\section{Fibrations of high rank and their corresponding embeddings}

A great difference between $Y_2$ and $Y_{10}$ is revealed in the following theorem.
\begin{theorem}
Contrary to $Y_2$ whose fibrations have rank $\leq 2$, the $K3$-surface $Y_{10}$ has elliptic fibrations of high rank, meaning a rank greater than the rank $3$ of some specialized fibrations. We exhibit, in the proof below, examples of fibrations of rank $4$, $5$, $6$ and $7$ together with their singular fibers and possible primitive embeddings. 

The rank $7$ elliptic fibration has been discovered in section 8.3. 
\end{theorem}
\begin{proof}
The elliptic fibration $252$  of rank $0$ has the following Weierstrass equation

\begin{tabular}{l}                                                                                                  
$Y^{2}$ $  =X^{3}+\left(  10t^{2}+90t+216\right)  X^{2}-t^{3}\left(  t+5\right)                                   
^{2}$\\                                                                                                          
$ III^{\ast}\left(  \infty\right)  ,\quad I_{6}\left(  0\right)  ,\quad                                        
I_{4}\left(  -5\right)  ,\quad I_{3}\left(  -9\right)  ,\quad I_{2}\left(                                       
-4\right)$.                                                                                                       
\end{tabular}

Taking as new parameter $m=\frac{Y}{X\left(  t+4\right)  }$ we obtain the
fibration
\begin{align*}                                                                                                  
E_{3} &  :Y^{2}=X^{3}+\left(  \frac{1}{4}t^{4}-5t^{2}+27\right)  X^{2}+\left(                                   
t^{2}-9\right)  \left(  2t^{2}-27\right)  X+\left(  2t^{2}-27\right)  ^{2}\\                                    
&  I_{10}\left(  \infty\right)  ,\quad IV\left(  0\right)  ,\quad I_{3}\left(                                   
2t^{2}-27\right)  ,\quad2I_{2}\left(  \pm4\right)  \qquad\text{rank                                             
1-generator }P_{0}.%
\end{align*}
Consider the $6$ points $P_{i}=\left(  X_{i},Y_{i}\right)  $

\begin{tabular}{ll}                                                                               
$P_{0}$ & $ =\left(  27-2t^{2},-\frac{1}{2}\left(  t^{2}-16\right)  \left(                                         
2t^{2}-27\right)  \right)$  \\                                                                                 
$P$ &  $=2P_{0}=\left(  0,2t^{2}-27\right)  \qquad P_{1}=3P_{0}=\left(                                             
-5,\frac{16-t^{2}}{2}\right)  \qquad$\\                                                                          
$P_{2}$ & $ =4P_{0}=\left(  \frac{2t^{2}-27}{4},-\frac{1}{8}\left(                                                 
t^{2}-1\right)  \left(  2t^{2}-27\right)  \right) $ \\                                                           
$P_{3}$ & $ =5P_{0}=\left(  \frac{-1}{25}\left(  2t^{2}-27\right)  \left(                                   
t^{2}-7\right)  ,\frac{1}{250}\left(  t^{2}-16\right)  \left(  6t^{2}%
-1\right)  \left(  2t^{2}-27\right)  \right) $ \\                                                                
$P_{4}$ & $ =6P_{0}=\left(  8\left(  t^{2}-1\right)  ,4t^{4}+22t^{2}-1\right) $                                     
\end{tabular}

and the following parameters
$                                                                                                              
\frac{Y-Y_{P_{i}}}{t\left(  X-X_{P_{i}}\right)  }+\frac{a_{i}t^{2}+b_{i}}%
{t}=m_{i}%
$.

For each example we give in the first column of the table the elliptic
parameter where $P_{i}\left(  a_{i},b_{i}\text{ }\right)  $ means $m_i$ given by
the previous formula, the Niemeier root lattice if the embedding is in it, or
the Niemeier lattice if we need glue vectors. We give in the second column a
Weierstrass equation, embeddings in the third line, singular fibers and in the
same second line the roots of orthogonal of the embeddings with the following
pattern: the first component corresponds to the root lattice of the orthogonal
of the first component of the embedding and so on. In the third column is the
rank and torsion and indications for the previous computations: here $\alpha$
refers to \cite{Nis} p. $308,310,322,$ $\delta$ refers to \cite{Nis} p.$309,311,323,\eta$
refers to \cite{Nis} p. $326, 327$ and $\ast$ to lemmas \ref{lem:D_n} or \ref{lem:D_7}, a sign $-$ meaning
that it is obvious for the corresponding factors. Notice that the Weierstrass
equation in the second column shows sometimes the $x$ coordinate of one or
three sections on the line $y=0.$%
 \newpage%
\[%
\begin{array}
[c]{|c|c|c|}\hline%
\begin{array}
[c]{c}%
\text{Param}\\
\text{Niemeier}\\
\text{lattice}%
\end{array}
&
\begin{array}
[c]{c}%
\text{Weierstrass Equations }\\
\text{Singular fibers }\\
\text{Embeddings}%
\end{array}
& \text{Rk%
$\vert$%
Tor}\\\hline%
\begin{array}
[c]{c}%
\left(  1\right) \\
P_{0}\left(  \frac{1}{2},6\right)
\end{array}
&
\begin{array}
[c]{c}%
\rule{0cm}{0.6cm}y^{2}-2t\left(  7+6t^{2}\right)  yx-72t\left(  21+t^{2}%
\right)  y=\\
\left(  x^{2}+432\left(  t^{2}+8\right)  \right)  \left(  x-\left(
36t^{6}+87t^{4}+223t^{2}+87\right)  \right)
\end{array}
& 5|0\\
& I_{12}\left(  \infty\right)  ,2I_{2}\left(  6t^{2}-49\right)  ,8I_{1}\qquad
A_{11}\left(  2A_{1}\right)  & \left(  \alpha\delta\right) \\
A_{15}D_{9} &
\begin{array}
[c]{c}%
N=\left(
\begin{array}
[c]{c}%
v=\left(  a_{1},d_{7}\right) \\
v_{1}=\left(  0,d_{9}\right) \\
v_{2}=\left(  0,d_{6}\right)
\end{array}
\right)  \qquad A_{2}=\left(
\begin{array}
[c]{c}%
\left(  0,d_{1}\right) \\
\left(  0,d_{2}\right)
\end{array}
\right) \\
A_{1}=\left(  \left(  a_{3},0\right)  \right)
\end{array}
& \\\hline
\left(  2\right)  & \rule{0cm}{0.6cm}y^{2}-\left(  t^{2}+8\right)
yx-t^{2}\left(  3t^{2}+11\right)  y=\left(  x^{2}+12t^{2}\right)  \left(
x-4+3t^{2}\right)  & 4|0\\
P\left(  \frac{1}{2},3\right)  & I_{10}\left(  \infty\right)  ,5I_{2}\left(
0,\pm3,3t^{2}-32\right)  ,4I_{1}\qquad A_{9}\left(  3A_{1}\right)  \left(
2A_{1}\right)  & \left(  \alpha\delta\eta\right) \\
A_{11}D_{7}E_{6} &
\begin{array}
[c]{c}%
N=\left(
\begin{array}
[c]{c}%
v=\left(  0,d_{7},e_{4}\right) \\
v_{1}=\left(  0,0,e_{2}\right) \\
v_{2}=\left(  0,0,e_{5}\right)
\end{array}
\right)  \qquad A_{2}=\left(
\begin{array}
[c]{c}%
\left(  0,d_{1},0\right) \\
\left(  0,d_{2},0\right)
\end{array}
\right) \\
A_{1}=\left(  a_{1},0,0\right)
\end{array}
& \\\hline
\left(  3\right)  & \rule{0cm}{0.6cm}y^{2}-\left(  3+t^{2}\right)
yx+3t^{2}y=\left(  x+1\right)  \left(  x^{2}+4t^{2}x+4t^{4}+27t^{2}\right)  &
4|0\\
P_{1}\left(  \frac{1}{2},-2\right)  & I_{10}\left(  \infty\right)
,2I_{3}\left(  2t^{2}-3\right)  ,I_{2}\left(  0\right)  ,6I_{1}\qquad
A_{9}A_{1}\left(  2A_{2}\right)  & \left(  \alpha\delta\delta\right) \\
A_{11}D_{7}E_{6} &
\begin{array}
[c]{c}%
N=\left(
\begin{array}
[c]{c}%
v=\left(  a_{1},d_{5},0\right) \\
v_{1}=\left(  0,d_{7},0\right) \\
v_{2}=\left(  0,d_{4},0\right)
\end{array}
\right)  \qquad A_{2}=\left(
\begin{array}
[c]{c}%
\left(  0,0,e_{1}\right) \\
\left(  0,0,e_{3}\right)
\end{array}
\right) \\
A_{1}=\left(  0,d_{1},0\right)
\end{array}
& \\\hline
\left(  4\right)  & \rule{0cm}{0.6cm}y^{2}-\left(  3+2t^{2}\right)
yx-t^{2}\left(  3t^{2}-1\right)  y=\left(  x+t^{2}\right)  \left(
x^{2}+12t^{2}\right)  & 4|0\\
P_{2}\left(  \frac{1}{2},\frac{3}{2}\right)  & 2I_{6}\left(  \infty,0\right)
,4I_{2}\left(  \pm1,3t^{2}-2\right)  ,4I_{1}\qquad A_{5}A_{5}0A_{1}\left(
3A_{1}\right)  & \left(  -\alpha\alpha\delta\right) \\
A_{5}^{4}D_{4} &
\begin{array}
[c]{c}%
N=\left(
\begin{array}
[c]{c}%
v=\left(  0,0,0,a_{2},d_{4}\right) \\
v_{1}=\left(  0,0,0,a_{1},0\right) \\
v_{2}=\left(  0,0,0,a_{3},0\right)
\end{array}
\right)  A_{2}=\left(
\begin{array}
[c]{c}%
\left(  0,0,a_{1},0,0\right) \\
\left(  0,0,a_{2},0,0\right)
\end{array}
\right) \\
A_{1}=\left(  0,0,a_{5},0,0\right)
\end{array}
& \\\hline
\left(  5\right)  & \rule{0cm}{0.6cm}y^{2}-\left(  3t^{2}+1\right)
yx-t^{2}\left(  6t^{2}-5\right)  y=\left(  x+4t^{4}-2t^{2}\right)  \left(
x^{2}+3t^{2}\right)  & 6|0\\
P_{3}\left(  \frac{3}{10},\frac{6}{5}\right)  & I_{4}\left(  \infty\right)
,I_{8}\left(  0\right)  ,I_{2}\left(  6t^{2}-1\right)  ,8I_{1}\qquad\left(
A_{7}A_{1}\right)  A_{3}A_{1} & \left(  --\delta\right) \\
L\left(  A_{9}^{2}D_{6}\right)  &
\begin{array}
[c]{c}%
N=\left(
\begin{array}
[c]{c}%
v=\left(  a_{8}^{\ast},a_{6}^{\ast},0\right) \\
v_{1}=\left(  0,a_{5}+a_{6},0\right) \\
v_{2}=\left(  0,a_{6}+a_{7}+a_{8},0\right)
\end{array}
\right)  \qquad A_{2}=\left(
\begin{array}
[c]{c}%
\left(  0,0,d_{1}\right) \\
\left(  0,0,d_{2}\right)
\end{array}
\right) \\
A_{1}=\left(  0,0,d_{6}\right)
\end{array}
& \\\hline
\left(  6\right)  & \rule{0cm}{0.6cm}y^{2}=x^{3}+\left(  t^{4}-44t^{2}%
+472\right)  x^{2}-16\left(  t^{2}-25\right)  x & 1|6\\
P_{4}\left(  \frac{1}{2},1\right)  & I_{12}\left(  \infty\right)
,2I_{3}\left(  t^{2}-24\right)  ,2I_{2}\left(  \pm5\right)  ,2I_{1}\qquad
A_{11}\left(  2A_{1}A_{2}\right)  A_{2}\qquad & \left(  -\ast\eta\right) \\
A_{11}D_{7}E_{6} &
\begin{array}
[c]{c}%
N=\left(
\begin{array}
[c]{c}%
v=\left(  0,d_{5}+d_{3},0\right) \\
v_{1}=\left(  0,d_{6},0\right) \\
v_{2}=\left(  0,d_{7},0\right)
\end{array}
\right)  \qquad A_{2}=\left(
\begin{array}
[c]{c}%
\left(  0,0,e_{1}\right) \\
\left(  0,0,e_{3}\right)
\end{array}
\right) \\
A_{1}=\left(  0,0,e_{2}\right)
\end{array}
& \\\hline
\end{array}
\]

\[%
\begin{array}
[c]{|c|c|c|}\hline%
\begin{array}
[c]{c}%
\text{Param}\\
\text{Niemeier}\\
\text{lattice}%
\end{array}
&
\begin{array}
[c]{c}%
\text{Weierstrass Equation }\\
\text{Singular fibers}\\
\text{Embeddings}%
\end{array}
&
\begin{array}
[c]{c}%
\text{Rk}\\
\text{Tor}%
\end{array}
\\\hline
\left(  7\right)  & \rule{0cm}{0.6cm}y^{2}+\left(  1+t^{2}\right)
yx-t^{2}\left(  6t^{2}+7\right)  y=\left(  x-4t^{2}\right)  \left(
x^{2}+3t^{2}\right)  & 6|0\\
P_{0}\left(  -\frac{1}{2},6\right)  & 2I_{6}\left(  \infty,0\right)
,2I_{2}\left(  3t^{2}-2\right)  ,8I_{1}\qquad A_{5}A_{5}A_{1}A_{1}0 & \left(
-\alpha\alpha\delta\right) \\
A_{5}^{4}D_{4} &
\begin{array}
[c]{c}%
N=\left(
\begin{array}
[c]{c}%
v=\left(  0,0,a_{1},a_{2},0\right) \\
v_{1}=\left(  0,0,0,a_{1},0\right) \\
v_{2}=\left(  0,0,0,a_{3},0\right)
\end{array}
\right)  A_{2}=\left(
\begin{array}
[c]{c}%
\left(  0,0,0,0,d_{4}\right) \\
\left(  0,0,0,0,d_{2}\right)
\end{array}
\right) \\
A_{1}=\left(  0,0,a_{3},0,0\right)
\end{array}
& \\\hline
\left(  8\right)  & \rule{0cm}{0.6cm}y^{2}+\left(  t^{2}+2\right)
yx-t^{2}\left(  3t^{2}+11\right)  y=\left(  x-3t^{2}\right)  \left(
x^{2}+12t^{2}\right)  & 4|2\\
P\left(  -\frac{1}{2},3\right)  & I_{8}\left(  \infty\right)  ,I_{4}\left(
0\right)  ,4I_{2}\left(  \pm1,6t^{2}-1\right)  ,4I_{1}\qquad A_{7}A_{3}\left(
2A_{1}\right)  \left(  2A_{1}\right)  & \left(  -\alpha\delta\delta\right) \\
A_{7}^{2}D_{5}^{2} &
\begin{array}
[c]{c}%
N=\left(
\begin{array}
[c]{c}%
v=\left(  0,a_{1},d_{3},0\right) \\
v_{1}=\left(  0,0,d_{5},0\right) \\
v_{2}=\left(  0,0,d_{4},0\right)
\end{array}
\right)  \qquad A_{2}=\left(
\begin{array}
[c]{c}%
\left(  0,0,0,d_{5}\right) \\
\left(  0,0,0,d_{3}\right)
\end{array}
\right) \\
A_{1}=\left(  0,a_{3},0,0\right)
\end{array}
& \\\hline
\left(  9\right)  & \rule{0cm}{0.6cm}y^{2}+\left(  t^{2}+7\right)
yx+t^{2}\left(  9+2t^{2}\right)  y=\left(  x+2t^{2}\right)  \left(
x^{2}+27t^{2}\right)  & 4|0\\
P_{1}\left(  -\frac{1}{2},-2\right)  & I_{8}\left(  \infty\right)
,I_{4}\left(  0\right)  ,2I_{3}\left(  t^{2}-6\right)  ,6I_{1}\qquad
A_{7}A_{3}\left(  2A_{2}\right)  & \left(  \alpha\delta\eta\right) \\
A_{11}D_{7}E_{6} &
\begin{array}
[c]{c}%
N=\left(
\begin{array}
[c]{c}%
v=\left(  a_{1},d_{5},0\right) \\
v_{1}=\left(  0,d_{7},0\right) \\
v_{2}=\left(  0,d_{4},0\right)
\end{array}
\right)  \qquad A_{2}=\left(
\begin{array}
[c]{c}%
\left(  0,0,e_{1}\right) \\
\left(  0,0,e_{3}\right)
\end{array}
\right) \\
A_{1}=\left(  a_{3},0,0\right)
\end{array}
& \\\hline
\left(  10\right)  & \rule{0cm}{0.6cm}y^{2}=x^{3}-\left(  3t^{4}%
+48t^{2}-264\right)  x^{2}-\left(  864t^{2}-3600\right)  x & 3|2\\
P_{2}\left(  -\frac{1}{2},\frac{3}{2}\right)  & I_{12}\left(  \infty\right)
,4I_{2}\left(  \pm2,6t^{2}-25\right)  ,4I_{1}\qquad A_{11}\left(  3A_{1}\right )%
A_{1}  & \left(  -\delta\eta\right) \\
A_{11}D_{7}E_{6} &
\begin{array}
[c]{c}%
N=\left(
\begin{array}
[c]{c}%
v=\left(  0,d_{1},e_{3}\right) \\
v_{1}=\left(  0,0,e_{1}\right) \\
v_{2}=\left(  0,0,e_{4}\right)
\end{array}
\right)  \qquad A_{2}=\left(
\begin{array}
[c]{c}%
\left(  0,d_{7},0\right) \\
\left(  0,d_{5},0\right)
\end{array}
\right) \\
A_{1}=\left(  0,0,e_{6}\right)
\end{array}
& \\\hline
\left(  11\right)  \, & \rule{0cm}{0.6cm}y^{2}+\left(  t^{2}-4\right)
yx-t^{2}\left(  t^{2}-63\right)  y=\left(  x-9t^{2}\right)  \left(
x^{2}+108t^{2}\right)  & 2|3\\
P_{4}\,\left(  -\frac{1}{2},1\right)  & 2I_{6}\left(  \infty,0\right)
,2I_{3}\left(  2t^{2}-3\right)  ,2I_{2}\left(  \pm1\right)  ,2I_{1}\qquad
A_{5}\left(  2A_{1}\right)  A_{5}\left(  2A_{2}\right)  & \left(  \eta
\delta\eta\eta\right) \\
E_{6}^{4} &
\begin{array}
[c]{c}%
N=\left(
\begin{array}
[c]{c}%
v=\left(  e_{1},e_{3},0,0\right) \\
v_{1}=\left(  0,e_{1},0,0\right) \\
v_{2}=\left(  0,e_{4},0,0\right)
\end{array}
\right)  \qquad A_{2}=\left(
\begin{array}
[c]{c}%
\left(  0,0,0,e_{1}\right) \\
\left(  0,0,0,e_{3}\right)
\end{array}
\right) \\
A_{1}=\left(  0,0,e_{1},0\right)
\end{array}
& \\\hline
\end{array}
\]

The rank $7$ elliptic fibration has a Weierstrass equation denoted $E_u$ in section 8.3 \ref{M} and \ref{P}
\[E_u \qquad Y^2=X^3-5u^2X^2+u^3(u^3+1)^2,\]
and singular fibers $2I_0^*(0,\infty)$, $3I_2(u^3+1)$, $6I_1$.

We can take as possible primitive embedding an embedding of type $(A_2,A_2)$ into $D_4^6$ (see section 3.2).

\end{proof}
\begin{remark}

 All the previous embeddings are of type $b)(A_1,A_3)$, except embedding $(5)$ which uses a glue vector of the corresponding Niemeier lattice and embedding $(6)$ which uses embedding \ref{lem:D_n}.

\end{remark}
\begin{remark}
To illustrate the complexity of the determination of elliptic fibrations of $Y_{10}$, notice the following fibrations:

$2A_1 2A_5 $ $(r=6)$ resulting from an embedding in $A_5^4D_4$ (type $(A_3,A_1)$ into $A_5^2$)

$2A_12A_22A_5$  $(r=2)$ resulting from an embedding in $E_6^4$

$A_12A_2A_32A_5$ $(r=0)$ resulting from an embedding in $Ni(A_{11}D_7E_6).                                                                                                                                           
$
\end{remark}

\section{$2$-isogenies of $Y_{10}$}

In \cite{B-L}, Bertin and Lecacheux classified all the $2$-isogenies of $Y_2$ in two sets, the first defining Morisson-Nikulin involutions, that is from $Y_2$ to its Kummer surface $K_2$ and the second giving van Geemen-Sarti involutions that is exchanging two elliptic fibrations (different or the same) of $Y_2$ named "self-isogenies".

Since we have no exhaustive list of elliptic fibrations of $Y_{10}$ with $2$-torsion sections, we cannot give such a classification. However we found on $Y_{10}$  Morisson-Nikulin involutions from $Y_{10}$
 to its Kummer surface $K_{10}$.
\subsection{Morisson-Nikulin involutions of $Y_{10}$} 

\subsubsection{Specialized Morisson-Nikulin involutions}
With the same argument as for specialization to $Y_2$, Morisson-Nikulin involutions specialized
 to $Y_{10}$ remain Morisson-Nikulin involutions of $Y_{10}$. 

We give in the following Table \ref{Ta:Iso8} the corresponding Weierstrass equations of such elliptic fibrations of the Kummer surface $K_{10}$ with transcendental lattice $\begin{pmatrix}
	12& 0\\
	0 & 24\\
\end{pmatrix}.$

\begin{table}[tp]
\[%
\begin{array}
[c]{|c|c|}

\hline
\text{No} & \text{Weierstrass Equation} \\ \hline
\#4 &
\begin{array}
[c]{c}%
                                                                                                                                                                                                      Y^{2}=X^{3}-2\left(  \frac{t^{3}}{8}-25\left(  49+12\sqrt                                                                                                                                              
{6}\right)  t-4\times7^{2}\left(  \frac{343}{27}+\frac{20\sqrt{6}}{3}\right)                                                                                                                                       
\right)  X^{2}\\                                                                                                                                                                                                   
 +\frac{1}{2^{6}}\left(  t-r_{1}\right)  \left(  t-r_{2}\right)  \left(                                                                                                                                          
t-r_{3}\right)  \left(  t-s_{1}\right)  \left(  t-s_{2}\right)  \left(                                                                                                                                             
t-s_{3}\right)  X\\                                                                                                                                                                                                
s_{1}   =-\frac{58}{3}-4\sqrt{6}+8\sqrt{3}-4\sqrt{2},s_{2}=-\frac{58}%
{3}-4\sqrt{6}-8\sqrt{3}+4\sqrt{2},s_{3}=\frac{116}{3}+8\sqrt{6}\\                                                                                                                                                  
r_{1}   =\frac{2}{3}+4\sqrt{6}+24\sqrt{2}+8\sqrt{3},r_{2}=\frac{2}{3}%
+4\sqrt{6}-24\sqrt{2}-8\sqrt{3},r_{3}=-\frac{4}{3}-8\sqrt{6}%
\end{array}
\\\hline\hline

\#8 &
\begin{array}
[c]{c}%

y^{2}=x^{3}+(2t^{3}-50t^{2}+100t-48)x^{2}+(t^{2}-24t+36)(t^{2}%
-24t+16)(t-1)^{2}x   \\
I_3(0), I_4(1), 4I_2(t^2-24t+16,t^2-24t+36), I_3^*(\infty) , (r=2)

\end{array}

\\\hline\hline
\#16 &
\begin{array}
[c]{c}%

y^{2}=x(656\sqrt{6}t^{2}+4t^{3}+80t\sqrt{6}+1608t^{2}+196t-x)\\
(3824\sqrt                                                                                                                                            
{6}t^{2}+4t^{3}+80\sqrt{6}t+9368t^{2}+196t-x)  \\


\text{or if } t=(5+2\sqrt{6})T \\

            Y=X(X+T(T^2+(238+96\sqrt{6})T+1)(X+T(T^2+(42+16\sqrt{6})T+1))\\
2I_2^*(0,\infty), 4I_2(T^2+16\sqrt{6} T+42T+1, T^2+96\sqrt{6}T+238T+1), (r=2)                                                                                                                                                                                                                
                                                                                                                                                                              
\end{array}
 \\\hline\hline
\#17 &
\begin{array}
[c]{c}%
y^2=x^3+(-2t^4+112t^2-784)x^2\\
+(t^2-48)(t^2-8)(t^2+4t+20)(t^2-4t-20)x

\\
=x(x-(t^2-48)(t^2-8))(x-(t^2+4t-20)(t^2-4t-20)))\\
I_8(\infty), 8I_2(t^2-4t-20,t^2+4t-20,t^2-8,t^2-48), (r=3) 

\end{array}\\\hline\hline
\#23 &
\begin{array}
[c]
{c}

y^2=x^3+(20t^3-50t^2+10t-\frac{1}{2})x^2\\
+\frac{1}{16}(8t-1)(12t-1)(4t^2-12t+1)(4t^2-8t+1)x\\
=x(x+\frac{1}{4}(4t^2-12t+1)(8t-1))(x+(12t-1)(4t^2-8t+1))
\\
I_0^*(\infty),I_6(0), 6I_2(\frac{1}{8}, \frac{1}{12}, 4t^2-8t+1, 4t^2-12t+1), (r=3)\\

\end{array}

\\\hline\hline

\#24 &
\begin{array}
[c]{c}%
y^{2}=x^3+2(49+20\sqrt{6})t(t^2-22t+1)x^2\\
+(49+20\sqrt{6})^2t^2(t^2-14t+1)(t^2-34t+1)x\\

2I_1^*(0,\infty), 5I_2(-1,t^2-14t+1,t^2-34t+1), (r=3)                                                                                                                                                     
\end{array}                                                                                                                                                                                                        
 \\\hline \hline 
 \#26 &
\begin{array}
[c]{c}%
 y^2=x^3+\frac{1}{4}(49+20\sqrt{6})(t^4-20t^3+118t^2-20t+1)x^2\\
 +(49+20\sqrt{6})^2t^2(t^2-10t+1)^2x\\
 =x(x-\frac{1}{4}(t^2-14t+1)(t^2-6t+1))(x-\frac{1}{4}(t^2-10t+1)^2)\\
 4I_4(0,\infty,t^2-10t+1), 4I_2(t^2-14t+1,t^2-6t+1), (r=2)
  
\end{array}                                                                                                                                                                                                    

\\\hline
\end{array}
\]

                                                                                                                                                                                                    \caption{Specialized fibrations of the Kummer $K_{10}$} \label{Ta:Iso8}
\end{table}

\subsubsection{A non specialized Morisson-Nikulin involution of $Y_{10}$}
 \begin{theorem}
 The rank $4$ elliptic fibration  of $Y_{10}$ (\ref{rk4}), with Weierstrass equation
 \[E_t \qquad y^2=x^3-(t^3+5t^2-2)x^2+(t^3+1)^2x\]
 and singular fibers $I_0^*(\infty)$, $3I_4(t^3+1)$, $I_2(0)$, $4I_1(1,-5/3,t^2-4t-4)$, has a $2$-torsion section defining a Morisson-Nikulin involution from $Y_{10}$ to $K_{10}$, that is $F_t=E_t/\langle (0,0) \rangle$ is a rank $4$ elliptic fibration of $K_{10}$ with Weierstrass equation
 \[ F_t \qquad Y^2=X^3+2(t^3+5t^2-2)X^2-(t^2-4t-4)(t-1)(3t+5)t^2X\]
 and singular fibers $I_0^*(\infty)$, $I_4(0)$, $7I_2(\pm 1,-5/3,t^2-t+1,t^2-4t-4)$.
 
\end{theorem}

\begin{proof}
Starting from $F_t$ and taking the new parameter $p=\frac{X}{(t^2-4t-4)(t-1)(3t+5)}$, we get a rank $1$ elliptic fibration with Weierstrass equation
\begin{align*}
F_{p} &  :Y^{2}=X^{3}+\frac{3}{4}p\left(  5p-1\right)  ^{2}X^{2}+\frac{1}%
{6}p^{2}\left(  2p-1\right)  \left(  5p-1\right)  \left(  49p^{2}%
-13p+1\right)  X\\
&  +\frac{1}{108}p^{3}\left(  2p-1\right)  ^{2}\left(  49p^{2}-13p+1\right)
^{2}
\end{align*}
and singular fibers $I_0^*(\infty)$, $I_3^*(0)$, $3I_3(\frac{1}{2}, 49p^2-13p+1)$.
The infinite section $P=(-\frac{1}{12} p(49p^2-13p+1), \frac{1}{8} p^2(49p^2-13p+1))$ is of height $h(P)=\frac{2}{3}$, is not equal to $2Q$ or $3Q$, hence the discriminant of the N\'eron-Severi lattice satisfies $\Delta=4 \times 4 \times 3^3 \times \frac{2}{3}=72 \times 4$.

Now we are going to compute its Gram matrix and deduce its transcendental lattice.

To compute the N\'eron-Severi lattice we order the generators as $(0)$ the zero section, $(f)$ the generic fiber, $\theta_i, 1 \leq i \leq 4$, $\eta_i, 1\leq i \leq 7$, $\gamma_i$, $\delta_i$, $\epsilon_i$, $1\leq i \leq 2$ the rational components of respectively $D_4$, $D_7$ and the three $A_2$ and finally the infinite section $(P)$.
\[NS=\left ( \begin{smallmatrix}
	-2 & 1 & 0 & 0 & 0 & 0 & 0 & 0 & 0 & 0 & 0 & 0 & 0 & 0 & 0 & 0 & 0 & 0 & 0 & 0\\
	1 & 0 & 0 & 0 & 0 & 0 & 0 & 0 & 0 & 0 & 0 & 0 & 0 & 0 & 0 & 0 & 0 & 0 & 0 & 1\\
	0 & 0 & -2 & 0 & 1 & 0 & 0 & 0 & 0 & 0 & 0 & 0 & 0 & 0 & 0 & 0 & 0 & 0 & 0 & 0\\
	0 & 0 & 0 & -2 & 1 & 0 & 0 & 0 & 0 & 0 & 0 & 0 & 0 & 0 & 0 & 0 & 0 & 0 & 0 & 0\\
	0 & 0 & 1 & 1 & -2 & 1 & 0 & 0 & 0 & 0 & 0 & 0 & 0 & 0 & 0 & 0 & 0 & 0 & 0 & 0\\
	0 & 0 & 0 & 0 & 1 & -2 & 0 & 0 & 0 & 0 & 0 & 0 & 0 & 0 & 0 & 0 & 0 & 0 & 0 & 1\\
	0 & 0 & 0 & 0 & 0 & 0 & -2 & 0 & 1 & 0 & 0 & 0 & 0 & 0 & 0 & 0 & 0 & 0 & 0 & 0\\
	0 & 0 & 0 & 0 & 0 & 0 & 0 & -2 & 1 & 0 & 0 & 0 & 0 & 0 & 0 & 0 & 0 & 0 & 0 & 0\\
	0 & 0 & 0 & 0 & 0 & 0 & 1 & 1 & -2 & 1 & 0 & 0 & 0 & 0 & 0 & 0 & 0 & 0 & 0 & 0\\
	0 & 0 & 0 & 0 & 0 & 0 & 0 & 0 & 1 & -2 & 1 & 0 & 0 & 0 & 0 & 0 & 0 & 0 & 0 & 0\\
	0 & 0 & 0 & 0 & 0 & 0 & 0 & 0 & 0 & 1 & -2 & 1 & 0 & 0 & 0 & 0 & 0 & 0 & 0 & 0\\
	0 & 0 & 0 & 0 & 0 & 0 & 0 & 0 & 0 & 0 & 1 & -2 & 1 & 0 & 0 & 0 & 0 & 0 & 0 & 0\\
	0 & 0 & 0 & 0 & 0 & 0 & 0 & 0 & 0 & 0 & 0 & 1 & -2 & 0 & 0 & 0 & 0 & 0 & 0 & 1\\
	0 & 0 & 0 & 0 & 0 & 0 & 0 & 0 & 0 & 0 & 0 & 0 & 0 & -2 & 1 & 0 & 0 & 0 & 0 & 0\\
	0 & 0 & 0 & 0 & 0 & 0 & 0 & 0 & 0 & 0 & 0 & 0 & 0 & 1 & -2 & 0 & 0 & 0 & 0 & 0\\
	0 & 0 & 0 & 0 & 0 & 0 & 0 & 0 & 0 & 0 & 0 & 0 & 0 & 0 & 0 & -2 & 1 & 0 & 0 & 1\\
	0 & 0 & 0 & 0 & 0 & 0 & 0 & 0 & 0 & 0 & 0 & 0 & 0 & 0 & 0 & 1 & -2 & 0 & 0 & 0\\
	0 & 0 & 0 & 0 & 0 & 0 & 0 & 0 & 0 & 0 & 0 & 0 & 0 & 0 & 0 & 0 & 0 & -2 & 1 & 1\\
	0 & 0 & 0 & 0 & 0 & 0 & 0 & 0 & 0 & 0 & 0 & 0 & 0 & 0 & 0 & 0 & 0 & 1 & -2 & 0\\
	0 & 1 & 0 & 0 & 0 & 1 & 0 & 0 & 0 & 0 & 0 & 0 & 1 & 0 & 0 & 1 & 0 & 1 & 0 & -2\\
\end{smallmatrix} \right ).
 \]

Fom lemma \ref{lem:gram} we deduce that $G_{NS}=\mathbb Z/12 \mathbb Z \oplus \mathbb Z/ 24 \mathbb Z$ and we get generators $f_1$ and $f_2$ with respective norms $q(f_1)=-\frac{41}{12}$, $q(f_2)=-\frac{59}{24}$ modulo $2$ and scalar product $f1.f2=\frac{7}{4}$ modulo $1$.

In order to prove that the transcendental lattice corresponds to the Gram matrix $\begin{pmatrix}
	12 & 0\\
	0 & 24\\
\end{pmatrix}$
we must find for the corresponding quadratic form generators $g_1$ and $g_2$ satisfying $q(g_1)=\frac{41}{12}$, $q(g_2)=\frac{59}{24}$ modulo $2$ and scalar product $g1.g2=-\frac{7}{4}$ modulo $1$.
This is obtained with $g_1=\begin{pmatrix}
	\frac{1}{4}\\
	\frac{1}{3}\\
\end{pmatrix}$ and $g_2=\begin{pmatrix}
	\frac{5}{12}\\
	\frac{1}{8}\\
\end{pmatrix}$.
Thus $F_p$, hence $F_t$, are elliptic fibrations of the Kummer surface $K_{10}$.

\end{proof}

In a previous paper, Bertin and Lecacheux \cite{B-L} explained that, in the Ap\'ery-Fermi family, only $Y_2$ and $Y_{10}$ may have "self-isogenies". A "self-isogeny" is a van Geemen-Sarti involution which either preserve an elliptic fibration(called "PF self-isogeny" for more precision) or exchanges two elliptic fibrations "EF self-isogeny".

Moreover, in the same paper, all the "self-isogenies" of $Y_2$ were listed. Since it is quite difficult to get all the elliptic fibrations of $Y_{10}$ with $2$-torsion sections, we shall give "self-isogenies" of $Y_{10}$ obtained either as specializations or from rank $0$ fibrations or from non specialized positive rank elliptic fibrations.

\subsection{Specialized "self-isogenies"}

In \cite{B-L} we characterized the surface $S_k$ obtained by $2$-isogeny deduced from van Geemen-Sarti involutions of $Y_k$ which are not Morrison-Nikulin. Let us recall the specialized Weierstrass equations of $S_{10}$. We denote $E{\#  n}$ (resp. $EE{\#  n}$) a Weierstrass equation of a fibration of $Y_{10}$ (resp. $S_{10}$).

The specialization of $S_k$ for $k=10$ has the following five elliptic  fibrations given on Table \ref{T:4}. The first Weiertrass equation concerns $Y_{10}$ and the second $S_{10}$ obtained as its $2$-isogenous curve.

\begin{table}[tp]
\[%
\begin{array}
[c]{|c|c|}

\hline
\text{No} & \text{Weierstrass Equation} \\ \hline
\#7 &
\begin{array}
[c]{c}%
 E7:{y}^{2}={x}^{3}+2{x}^{2}t\left (11t+1\right )-{t}^{2}\left (t-1                                              
\right )^{3}x \\ 

 III^*({\infty}), I_1^*(0), I_6(1), 2I_1(t^2+118t+25)\\

EE7:={Y}^{2}={X}^{3}-4{X}^{2}t\left (11t+1\right )+4{t}^{3}\left (118                                           
t+25+{t}^{2}\right )X \\
III^*                                                                               
                                                                                                     ({\infty}), I_2^*(0), I_3(1), 2I_2(t^2+118t+25)                                                                                                                                                                                                                                                                                                                                                                      
\end{array}
\\\hline\hline

\#9 &

\begin{array}
[c]{c}%
 E9:{y}^{2}={x}^{3}+28{t}^{2}{x}^{2}+{t}^{3}\left ({t}^{2}+98t+1                                                 
\right )x  \\         
2III^*(0,{\infty}), 2I_2(t^2+98t+1), 2I_1(t^2-98t+1)\\  
EE9:                                                                                                            
{Y}^{2}={X}^{3}-56{t}^{2}{X}^{2}-4{t}^{3}\left ({t}^{2}-98 t+1                                                  
\right )X    \\
2III^*(0,{\infty}), 2I_2(t^2-98t+1), 2I_1(t^2+98t+1)\\

\end{array}

\\\hline\hline
\#14 &
\begin{array}
[c]{c}%

E14:{y}^{2}={x}^{3}+t\left (98{t}^{2}+28t+1\right ){x}^{2}+{t}^{6}x

\\
 I_8^*(0), I_0^*(\infty), I_1(4t+1), I_1(24t+1), 2I_1(100t^2+28t+1)\\                                             
EE14:{Y}^{2}=X\left (X-96{t}^{3}-28{t}^{2}-t\right )\left (X-100{t}^{                                           
3}-28{t}^{2}-t\right)                                                                                           
\\

  I_4^*(0), I_0^*(\infty), I_2(4t+1), I_2(24t+1), 2I_2(100t^2+28t+1)\\

\end{array}
 \\\hline\hline
\#15 &
\begin{array}
[c]{c}%
E15:{y}^{2}={x}^{3}-t\left (2+{t}^{2}-22t\right ){x}^{2}+{t}^{2}\left (t                                        
+1\right )^{2}x                                                                                                 
\\

I_1^*(0), I_4^*(\infty), I_4(-1), I_1(24), 2I_1(t^2-20t+4)\\
                                                                                                             
EE15:{Y}^{2}=X\left (X+t^3-24t^2\right ) \left (X+t^3-20t^2+4t    \right )                                                                                          
\\

I_2^*(0), I_2^*(\infty), I_2(-1), I_2(24), 2I_2(t^2-20t+4)\\

\end{array}\\\hline\hline
\#20 &
\begin{array}
[c]
{c}
E20:{y}^{2}={x}^{3}+\left (\frac{1}{4}{t}^{4}-5{t}^{3}+{\frac {53}{2}}                                          
{t}^{2}-15t-\frac{3}{4}\right ){x}^{2}-t\left (t-10\right )x                                                    
\\

I_2(0), I_{12}(\infty), 2I_3(t^2-10t+1), I_2(10), I_1(1), I_1(9)\\

EE20:{Y}^{2}={X}^{3}+\left (-\frac{1}{2}{t}^{4}+10{t}^{3}-53{t}^{2}+30t+\frac{3}{2}                             
\right ){X}^{2}\\+\frac{1}{16}\left (t-1\right )\left (t-9\right )\left ({t}^{                                    
2}-10t+1\right )^{3}X \\

I_1(0), I_6(\infty), 2I_6(t^2-10t+1), I_2(1), I_2(9), I_1(10)\\

\end{array}

\\
\hline    
\end{array}
                                                                                                                                                                                             \]

 \caption{Fibrations $E{\#i}$ of $Y_{10}$ and $EE{\#i}$ of  $S_{10}$}\label{T:4}
 \end{table}

\begin{theorem}
\begin{enumerate}[leftmargin=*]
\item The previous $2$-isogenies are in fact "self-isogenies", 
the surface $S_{10}$ being equal to $Y_{10}$.

\item The possible corresponding primitive embeddings of $N\oplus A_1 \oplus A_2$ result for all elliptic fibrations but $\#7$\
 from embeddings of $N$ into a Niemeier lattice not in its root lattice.
\end{enumerate}

\end{theorem}

\begin{proof}
\begin{enumerate}[leftmargin=*]
\item
We observe that $E9$ and $EE9$ have the same singular fibers. In fact these two fibrations are isomorphic, the isomorphism being defined by $t=-T$, $x=-\frac{X}{2}$, $y=\frac{Y}{2\sqrt{-2}}$.
This property is sufficient to identify $S_{10}$ to $Y_{10}$.

Among these "self-isogenies" only the number $\#9$ is "PF".

\item Now we give various primitive embeddings for $EE7$, $EE14$, $EE15$ and $EE20$, a primitive embedding for $EE9$ being given in Table \ref{Ta:Fib2}.

\begin{itemize}[font=\bf,leftmargin=*]
\item [EE7] ($2A_1 A_2 D_6 E_7$)

The fibration follows from an embedding into $D_{10}E_7^2$.

We embed $N$ into $D_{10}$.

$(N)_{D_{10}}^{\perp}=A_2 \oplus D_5$. Then we embed $A_2$ into $D_5$ with orthogonal $2A_1$. Finally the embedding of $A_1$ into a copy of $E_7$ has for orthogonal $D_6$.

\item [EE14] ($4A_1 D_4 D_8$)

The fibration follows from an embedding into $Ni(D_{12}^2)$.
The glue code is $[1,2]$ and realizes a vector $v$ of norm $4$, with
\[[1]=(1/2,...,1/2)\qquad \qquad [2]=(0,...,0,1).\]
The embedding of $N$, $A_1$ and $A_2$ is given by
\[N=\begin{pmatrix}
	[1] & [2] \\
	0 & (0^{10},1,-1) \\
	(0^2,1,0^2,1,0^6) & 0 \\
\end{pmatrix}\]
\[A_1=\begin{pmatrix}
(0^6,1,-1,0^4),0)
\end{pmatrix}\]
\[A_2=\begin{pmatrix}
0,(1,-1,0^{10})\\
0,(0,1,-1,0^9)
\end{pmatrix}\]

\item [EE15] ($4A_1 2 D_6 $)

The fibration follows from an embedding into $Ni(D_8^3)$.

The glue code being $[1,2,2]$, we embed $N$, $A_1$ and $A_2$ as 
\[N=\begin{pmatrix}
	[1] & [2] & [2]\\
	0 & (0^{6},1,-1)& 0 \\
	0 & 0 & (0^6,1,-1)\\
\end{pmatrix}\]
\[A_1=\begin{pmatrix}
(0^4,1,-1,0^2) & 0 & 0)
\end{pmatrix}\]
\[A_2=\begin{pmatrix}
(0^2,1,-1,0^4) & 0 & 0\\
(0^3,1,-1,0^3) & 0 & 0
\end{pmatrix}.\]

\item [EE20] ($2A_1 3A_5$)

The fibration follows from an embedding into $Ni(A_{11}D_7E_6)$.
We take the glue vector $[6,2,0]$ with $[6]=(1/2,1/2,1/2,1/2,1/2,1/2,-1/2,...,-1/2)$.

 We embed $N$, $A_1$ and $A_2$ as
 \[N=\begin{pmatrix}
	[6] & [2] & 0\\
	0 & (0^5,1,-1)& 0 \\
	0 & (0^5,1,1) & 0\\
\end{pmatrix}\]
\[A_1=\begin{pmatrix}
0 & 0 & (-1,1,0^6))
\end{pmatrix}\]
\[A_2=\begin{pmatrix}
0 & (1,-1,0^5) & 0\\
0 & (0,1,-1,0^4) & 0
\end{pmatrix}.\]

\end{itemize}

\end{enumerate}

\end{proof}


\subsection{Rank $0$ "self-isogenies"}

\begin{theorem}

 There are four $2$-isogenies from $Y_{10}$ to $Y_{10}$ defined by extremal elliptic fibrations with $2$-torsion sections, namely $8$, $87$, $153$, $262$. They are all "PF self-isogenies".

\end{theorem}


\begin{proof}

We write below Weierstrass equation $E_{n}$, its $2$-isogenous $EE_{n}$ and the corresponding isomorphism.

\begin{tabular}{l}
    $E_{262} \qquad y^2=x^3+x^2(9(t+5)(t+3)+(t+9)^2)-t^3(t+5)^2x$\\
	$III^*(\infty)$, $I_6(0)$, $I_4(-5)$, $I_3(-9)$, $I_2(-4)$\\
	$EE_{262} \qquad Y^2=X^3-2(9(T+5)(T+3)+(T+9)^2)X^2+4(T+4)^2(T+9)^3X$\\
	$III^*(\infty)$, $I_6(-9)$, $I_4(-4)$, $I_3(0)$, $I_2(-5)$\\
	Isomorphism: $t=-T-9$, $x=-\frac{X}{2}$, $y=\frac{Y}{2\sqrt{-2}}$.
\end{tabular}

\bigskip
\begin{tabular}{l}
 $E_{153}	\qquad y^2=x^3+t(t^2+10t-2)x^2+(2t+1)^3t^2x$\\
  $I_3(4)$, $I_6(-1/2)$, $I_1^*(0)$, $I_2^*(\infty)$\\
	$EE_{153} \qquad Y^2=X^3-2T(T^2+10T-2)X^2+T^3(T-4)^3X$\\
	$I_6(4)$, $I_3(-1/2)$, $I_2^*(0)$, $I_1^*(\infty)$\\
	 Isomorphism: $t=-\frac{2}{T}$, $x=-\frac{2X}{T^4}$, $y=-\frac{2\sqrt{-2}Y}{T^6}$.
\end{tabular}
\bigskip

\begin{tabular}{l}
	$E_{8} \qquad y^2=x^3-(3t^4-60t^2-24)x^2-144(t^2-1)^3x$\\
	$I_2(0)$, $2I_3(t^2+8)$, $I_4(\infty)$, $2I_6(t^2-1)$\\
	$EE_{8} \qquad y^2=x^3+2(3t^4-60t^2-24)x^2+9t^2(t^2+8)^3x$\\
	$I_2(\infty)$, $2I_3(t^2-1)$, $I_4(0)$, $2I_6(t^2+8)$\\
Isomorphism: 	$t=\frac{2\sqrt{-2}}{T}$, $x=\frac{4X}{T^4}$, $y=\frac{8Y}{T^6}$.\\
\end{tabular}

\bigskip

\begin{tabular}{l}
	$E_{87} \qquad y^2=x^3-(9t^4+9t^3+6t^2-6t+4)x^2+(21t^2-12t+4)x$\\
	$I_{12}(\infty)$, $I_6(0)$, $2I_2(21t^2-12t+4)$, $2I_1(3t^2+6t+7)$\\
	$EE_{87} \qquad Y^2=X^3-(9T^4+9T^3+6T^2-6T+4)X^2+(21T^2-12T+4)X$\\
	$I_{12}(0)$, $I_6(\infty)$, $2I_2(3t^2+6t+7)$, $2I_1(21t^2-12t+4)$\\
	Isomorphism: $t=-\frac{2}{T}$, $x=-\frac{2X}{9T^4}$, $y=\frac{2\sqrt{-2}}{27T^6}$.
\\
\end{tabular}

\end{proof}

\subsection{Positive rank non specialized "EF" and "PF self-isogenies"}

Using the $2$-neighbor method we found many examples of $2$-torsion elliptic fibrations of $Y_{10}$.

Denote $E$, $E_1$, $E_2$, $E_3$, $E_4$ the following elliptic fibrations of $Y_{10}$ obtained in the following way. 
Starting from $E_{ 153}$ and new parameter $\frac{x}{t(2t+1)^2}$ we get $E$. Starting from $EE15$ we get successively $E_2$, $E_3$, $E_4$ with the successive parameters $\frac{x}{t^2(t-24)}$, $\frac{x}{t^2(t+1)}$, $\frac{x}{t(t-4)(t-24)}$. And from $EE_2=E_2/\langle (0,0 )\rangle$ we get $E_1$ with the new parameter $\frac{x}{t(t-1)(t-4)}$.

\begin{theorem}
\begin{enumerate}[leftmargin=*]

\item The $2$-isogenies, from $E_3$ to $EE_3$, from $E_4$ to $EE_4$, from $E_1$ to $EE_1$ are "PF self-isogenies".

\item The $2$-isogenies from $E$ to $EE$, from $EE14$ to $EE14/\langle (100t^2+28t+1,0)\rangle $ and from $E_2$ to $EE_2$ are "EF self-isogenies".

\end{enumerate}
\end{theorem}

\begin{proof}
\begin{enumerate}[leftmargin=*]
\item We only need to give the respective Weierstrass equations, singular fibers and isomorphisms concerning the $2$-isogenies from $E_i$ to $EE_i$.

\begin{tabular}{c}
	$E_3 \qquad y^2=x^3-2t(t^2-14t-2)x^2+t^4(t-4)(t-24)x$\\
$I_4^*(0)$, $2I_2(4,24)$, $2I_1(-1/2,-1/12)$, $I_2^*(\infty)$\\
$EE_3=E_3/\langle (0,0)\rangle \hfill$
\\$ Y^2=X^3+T(T^2-14T-2)X^2+T^2(2T+1)(12T+1)X$\\
 $I_2^*(0)$, $2I_1(4,24)$, $2I_2(-1/2,-1/12)$, $I_4^*(\infty)$
\\
Isomorphism: $t=-\frac{2}{T}, \qquad x=-\frac{8X}{T^4} \qquad y=\frac{16\sqrt{-2}Y}{T^6} $\\
\end{tabular}
\bigskip

\begin{tabular}{c}
$E_4 \qquad y^2=x^3-28t^2(t-1)x^2+4t^3(t-1)^2(24t+1)x$\\
$III^*(0)$, $I_0^*(1)$, $I_2(-1/24)$, $I_1(1/25)$, $I_0^*(\infty)$\\
$EE_4=E_4/\langle (0,0)\rangle \hfill $\\$Y^2=X^3+56T^2(T-1)X^2+16T^3(T-1)^2(25T-1)X$\\
$III^*(0)$, $I_0^*(1)$, $I_2(1/25)$, $I_1(-1/24)$, $I_0^*(\infty)$\\
Isomorphism: $t=\frac{T}{T-1}, \qquad x=-\frac{X}{2(T-1)^4} \qquad y=-\frac{\sqrt{-2}Y}{4(T-1)^6}$\\
\end{tabular}

\bigskip

\begin{tabular}{c}
$E_1 \qquad y^2=x^3-t(5t^2+56t+160)x^2+4t^2(t+6)^2(t+4)^2x$\\
$I_0^*(0)$, $2I_4(-4,-6)$, $2I_2(-8,-16/3)$, $I_0^*(\infty)$\\
$EE_1=E_1/\langle (0,0)\rangle \hfill$\\$ Y^2=X^3+2T(5T^2+56T+160)X^2+T^2(T+8)^2(3T+16)^2X$\\
$I_0^*(0)$, $2I_4(-8,-16/3)$, $2I_2(-6,-4)$, $I_0^*(\infty)$\\
Isomorphism: $t=\frac{32}{T}, \qquad x=-\frac{2^9X}{T^4} \qquad y=\frac{2^{13}\sqrt{-2}Y}{T^6} $\\

\end{tabular}

\item Let us give Weierstrass equations and singular fibers of $E$ and $EE$.

\begin{tabular}{c}
$E \qquad y^2=x^3+2t(2t^2+5t+1)x^2+t^3(4t+1)(t-1)^2x$	\\
	$I_2^*(0)$, $I_4(1)$, $I_3(-1/3)$, $I_2(-1/4)$, $I_1^*(\infty)$\\
	$EE=E/\langle (0,0)\rangle \hfill$\\$ Y^2=X^3-4T(2T^2+5T+1)X^2+4T^2(3T+1)^3X$\\
	$I_1^*(0)$, $I_6(-1/3)$, $I_2(1)$, $I_1(-1/4)$, $I_2^*(\infty) $\\
\end{tabular}

\noindent The fibration $EE$ is a fibration of $Y_{10}$, since with the new parameter $\frac{X}{(3T+1)^3}$, we get the rank $0$ elliptic fibration $E_{252}$.
\bigskip

\noindent We also obtain

\begin{tabular}{c}
$E_2 \qquad y^2=x^3-4t(t+1)(6t+5)x^2+4t^2(t+1)^3x$\\
$I_0^*(0)$, $I_2^*(-1)$, $2I_1(-8/9,-3/4)$\\
$EE_2=E_2/\langle (0,0)\rangle \hfill$\\$Y^2=X^3+8T(T+1)(6T+5)X^2$\\
$+16T^2(T+1)^2(9T+8)(4T+3)X$\\
$I_1^*(-1)$, $I_0^*(0)$, $2I_2(-3/4,-8/9)$, $I_1^*(\infty)$\\

\end{tabular}

\noindent To prove that $EE_2$ is an elliptic fibration of $Y_{10}$
first we change the parameter $T=1/u-1$ to get the new equation
$EE_2(1)$.
\[EE_2(1) \qquad y^2=x^3+8u(u-1)(u-6)x^2+16u^2(u-4)(u-9)(u-1)^2x\]
Now with the new parameter $\frac{x}{u(u-1)(u-4)(u-9)}$, we obtain
\[EE_2(2) \qquad y^2=x^3-t(59t^2-88t+32)x^2+32t^2(t-1))(3t-2)^3x\]
Again, from $EE_2(2)$, the parameter $\frac{x}{t^2(t-1)}$ leads to the rank $0$ fibration $E_{252}$ of $Y_{10}$.
\bigskip

\noindent Finally, the fibration $EE14/\langle(100t^2+28t+1,0)  \rangle $ with Weierstrass equation
\[ y^2=x^3-2t(104t^2+28t+1)x^2+t^2(4t+1)^2(24t+1)^2x\] 
is a fibration of $Y_{10}$, since with the new parameter $\frac{x}{t(4t+1)^2}$ we obtain $EE_2$.

\end{enumerate}

\end{proof}
In the previous theorem we gave "self-isogenies" of elliptic fibrations with rank less than $2$. However we found in section $8$ an interesting $2$-torsion rank $4$ fibration. We present it in the following theorem.

\begin{theorem} 
The rank $4$ elliptic fibration of $Y_{10}$ (\ref{rk42}) with singular fibers $3I_4$, $3I_2$, $2III$, Weierstrass equation 
\[F \qquad y^2=x^3+4t^2x^2+t(t^3+1)^2x\]
and its $2$-isogenous $F/\langle (0,0) \rangle$ are "PF self-isogenous".

\end{theorem}

\begin{proof}
We get
\[F/\langle (0,0) \rangle \qquad Y^2=X^3-8T^2X^2-4T(T^3-1)^2X\]
with the same type of singular fibers. The isomorphism is given by
\[T=-t,\qquad Y=-2\sqrt{-2}y,\qquad X=-2x.\]

The possible primitive embedding giving the fibration may be obtained in $Ni(A_3^8)$ using the norm $4$ glue vector $(3,2,0,0,1,0,1,1)=(a_1^*,a_2^*,0,0,a_3^*,0,a_3^*,a_3^*)$ as
\[N=\begin{pmatrix}
	a_1^* & a_2^* & 0 & 0 &a_3^* & 0 & a_3^* & a_3^* \\
	0 & 0 & 0 & 0 &0 & 0 &a_3 & 0 \\
	0 & 0 & 0 & 0 & 0 & 0 & 0 & a_3\\
\end{pmatrix}\]
\[A_1=\begin{pmatrix}
0 & 0 & 0 & 0 & a_1 & 0 & 0 &0
\end{pmatrix}\]
\[A_2=\begin{pmatrix}
a_2 &0 & 0 & 0 &0 & 0 & 0 & 0\\
a_3 &0 & 0 & 0 &0 & 0 & 0 & 0

\end{pmatrix}.\]

\end{proof}
\section{$3$-isogenies from $Y_2$ and from $Y_{10}$}

\subsection{$3$-isogenous curves}

\subsubsection{Method}

Let $E$ be an elliptic curve with a $3$-torsion point $\omega=\left(  0,0\right)                                                                                                                                      
$
\[E:Y^{2}+AYX+BY=X^{3}                             
\]
and $\phi$ the isogeny of \ kernel $<\omega>.$

To determine a Weierstrass \ equation for the elliptic curve $E/<\omega>$ we
need two functions $x_{1}$ of degree $2$ and $y_{1}$ of degree $3$ invariant
by $M\rightarrow M+\omega$ where $M=\left(  X_{M},Y_{M}\right)  $ \ is a
general point on $E.$ We compute $M+\omega$ and $M+2\omega\left(                                                                                                                                                 =M-\omega\right)  $ and can choose%

\begin{align*}
x_{1}  &  =X_{M}+X_{M+\omega}+X_{M+2\omega}=\frac{X^{3}+ABX+B^{2}}{X^{2}}\\                                                                                                                                       
y_{1}  &  =Y_{M}+Y_{M+\omega}+Y_{M+2\omega}
\\ &=\frac{Y\left(  X^{3}%
-AXB-2B^{2}\right)  -B\left(  X^{3}+A^{2}X^{2}+2AXB+B^{2}\right)  }{X^{3}}.%
\end{align*}
The relation between $x_{1}$ and $y_{1}$ gives a Weierstrass equation for
$E/<\omega>$
\[                                                                                                                                                                                                                 
y_{1}^{2}+\left(  Ax_{1}+3B\right)  y_{1}=x_{1}^{3}-6ABx_{1}-B\left(                                                                                                                                               
A^{3}+9B\right).                                     
\]

Notice that the points with $x_{1}=-\frac{A^{2}}{3}$ are $3$-torsion points.
Taking one of these points to origin and after some transformation we can
obtain a Weierstrass equation $y^{2}+ayx+by=x^{3}$ with the following transformations.

\subsubsection{Formulae} \label{formulae}

If $j^{3}=1$ then we define
\begin{align*}                                                                                                                                            
S_{1}  &  =2\left(  j^{2}-1\right)  y+6Ax-2\left(  j-1\right)  \left(                                                                                                                                              
A^{3}-27B\right) \\                                                                                                                                                                                                
S_{2}  &  =2\left(  j-1\right)  y+6Ax-2\left(  j^{2}-1\right)  \left(                                                                                                                                              
A^{3}-27B\right)                                                                                                                                                                                                   
\end{align*}
and
\[                                                                                                                                                                                                                 
X=\frac{-1}{324}\frac{S_{1}S_{2}}{x^{2}},\qquad Y=\frac{1}{5832}\frac                                                                                                                                              
{S_{1}^{3}}{x^{3}}%
\]
then we have
\[                                                                                                                                                                                                                 
E/<\omega>:y^{2}+\left(  -3A\right)  yx+\left(  27B-A^{3}\right)  y=x^{3}.%
\]
If $A_{1}=-3A$, $B_{1}=27B-A^{3},$ then we define%

\begin{align*}                                                                                                                                                                                                     
\sigma_{1}  &  =2\left(  j^{2}-1\right)  3^{6}Y+6A_{1}3^{4}X-2\left(                                                                                                                                               
j-1\right)  \left(  A_{1}^{3}-27B_{1}\right) \\                                                                                                                                                                    
\sigma_{2}  &  =2\left(  j-1\right)  3^{6}Y+6A_{1}3^{4}X-2\left(                                                                                                                                                   
j^{2}-1\right)  \left(  A_{1}^{3}-27B_{1}\right)                                                                                                                                                                   
\end{align*}
and then%

\[                                                                                                                                                                                                                 
x=\frac{-1}{324}\frac{\sigma_{1}\sigma_{2}}{3^{8}X^{2}}=-{\frac{3\,{X}^{3}%
+{A}^{2}{X}^{2}+3\,BAX+3\,{B}^{2}}{{X}^{2}}},\qquad y=\frac{1}{5832}%
\frac{\sigma_{1}^{3}}{3^{12}X^{3}}.%
\]
\subsubsection{Other properties of isogenies}

The divisor of the function $Y$ is equal to $-3\left(  0\right)  +3\omega$ so
$Y=W^{3}$ where $W$ is a function on the curve $E/<\omega>$. If $X=WZ$ the
function field of $E/<\omega>$ is generated by $W$ and $Z.$ So replacing in
the equation of $E$ we obtain the relation between $Z$ and $W$                                                                                                                                            
\[                                  
W^{3}+AZW+B-Z^{3}=0.                                                                                                                                             \]      

This cubic equation, with a rational point at infinity with $W=Z$ can be
transformed to obtain a Weierstrass equation in the coordinates $X_2$ and $Y_2$:
\begin{align*}
W  &  =\frac{1}{9}\frac{(-243B-3X_{2}A+9A^{3}-Y_{2})}{X_{2}},\quad Z=-\frac
{1}{9}\frac{Y_{2}}{X_{2}}\\
\text{of inverse} \quad X_{2}  &  =3\frac{A^{3}-27B}{3\left(  W-Z\right)
+A},\quad Y_{2}=-27Z\frac{A^{3}-27B}{3\left(  W-Z\right)  +A}
\end{align*}
\begin{align*}
Y_{2}^{2}+3AY_{2}X_{2}+\left(  -9A^{3}+243B\right)  Y_{2}  =\\  X_{2}%
^{3}-9X_{2}^{2}A^{2}+27A\left(  A^{3}-27B\right)  X_{2}-27\left(
A^{3}-27B\right)  ^{2}.%
\end{align*}

The points of $X_{2}$-coordinate equal to $0$ are $3-$torsion points and
easily we recover the previous formulae.

\subsection{Generic 3-isogenies}

In \cite{B-L}, Bertin and Lecacheux exhibited all the elliptic fibrations of the Ap\'ery-Fermi pencil, called generic elliptic fibrations and found
two $3$-torsion elliptic fibrations defined by a Weierstrass equation, namely $E_{\#19}$ with rank $1$ and
$E_{\#20}$ with rank $0$.
We are giving their $3$-isogenous $K3$ surface.

\begin{theorem}
The $3$-isogenous elliptic fibrations of fibration $\#19$ (resp. $\#20$) defined by Weierstrass equations $H_{\#19}(k)$ (resp. $H_{\#20}(k)$) are elliptic fibrations of the same $K3$ surface $N_k$ with transcendental lattice
 $\big( \begin{smallmatrix}                                                                                                                                                                          
        0 & 0 & 3\\                                                                                                                                                                                                
        0 & 4 & 0\\                                                                                                                                                                                                
        3 & 0 & 0\\                                                                                                                                                                                                
\end{smallmatrix} \big )$
and discriminant form of its N\'eron-Severi lattice $G_{NS}=\mathbb Z/3\mathbb Z(-\frac{2}{3})\oplus \mathbb Z/12\mathbb Z(\frac{5}{12})$. 
\end{theorem}

\begin{proof}
The $6$-torsion elliptic fibration $\#20$ has a Weierstrass equation
\[E_{\#20}(k) \qquad y^2-(t^2-tk+3)xy-(t^2-tk+1)y=x^3\]
with singular fibers $I_{12}(\infty)$, $2I_3(t^2-kt+1)$, $2I_2(0,k)$, $2I_1(t^2-kt+9)$ and $3$-torsion point $(0,0)$. Using \ref{formulae}, it follows the Weierstrass form of its $3$-isogenous fibration
\[H_{\#20}(k)=E_{\#20}(k)/\langle (0,0) \rangle  \qquad Y^2+3(t^2-tk+3)XY+t^2(t^2-tk+9)(t-k)^2Y=X^3\]
with singular fibers $2I_6(0,k)$, $I_4(\infty)$, $2I_3(t^2-kt+9)$, $2I_1(t^2-kt+1)$. Thus it is a rank $0$ and $6$-torsion elliptic fibration of a $K3$-surface with Picard number $19$ and discriminant $\frac{6\times 6 \times 3 \times 3 \times 4}{6\times 6}=12\times 3$.

Now we shall compute the Gram matrix $NS(20)$ of the N\'eron-Severi lattice of the $K3$ surface with elliptic fibration $H_{\#20}(k)$ in order to deduce its discriminant form.

Applying Shioda's result, we order the following elements as, $s_0$, $F$, $\theta_{0,i}$, $1\leq i \leq 4$, $s_3$, $\theta_{k,i}$, $1\leq i \leq 5$, $\theta_{\infty,i}$, $1\leq i \leq 3$, $\theta_{t_0,i}$, $1\leq i \leq 2$,  $\theta_{t_1,i}$, $1\leq i \leq 2$, where $s_0$ and $s_3$ denotes respectively the zero and $3$-torsion section, $F$ the generic section, $\theta_{k,i}$ the components of reducible singular fiber\
s, $t_0$ and $t_1$ being roots of $t^2-kt+9$. We obtain

\[ NS(20)=\left(
\begin{smallmatrix}
        -2&1&0&0&0&0&0&0&0&0&0&0&0&0
&0&0&0&0&0\\
        1&0&0&0&0&0&1&0&0&0&0&0&0&0&0&0&0&0&0  \\
        0&0&-2&1&0&0&1&0&0&0&0&0&0&0&0&0&0&0&0 \\
        0&0&1&-2&1&0&0&0&0&0&0&0&0&0&0&0&0&0&0  \\
        0&0&0&1&-2&1&0&0&0&0&0&0&0&0&0&0&0&0&0\\
        0&0&0&0&1&-2&0&0&0&0&0&0&0&0&0&0&0&0&0\\
        0&1&1&0&0&0&-2&1&0&0&0&0&0&1&0&0&1&0&1 \\
        0&0&0&0&0&0&1&-2&1&0&0&0&0&0&0&0&0&0&0 \\
        0&0&0&0&0&0&0&1&-2&1&0&0&0&0&0&0&0&0&0 \\                                                                                                                                                                  
        0&0&0&0&0&0&0&0&1&-2&1&0&0&0&0&0&0&0&0\\                                                                                                                                                                   
        0&0&0&0&0&0&0&0&0&1&-2&1&0&0&0&0&0&0&0 \\                                                                                                                                                                  
        0&0&0&0&0&0&0&0&0&0&1&-2&0&0&0&0&0&0&0     \\                                                                                                                                                              
        0&0&0&0&0&0&0&0&0&0&0&0&-2&1&0&0&0&0&0   \\                                                                                                                                                                
        0&0&0&0&0&0&1&0&0&0&0&0&1&-2&1&0&0&0&0 \\                                                                                                                                                                  
        0&0&0&0&0&0&0&0&0&0&0&0&0&1&-2&0&0&0&0    \\                                                                                                                                                               
        0&0&0&0&0&0&0&0&0&0&0&0&0&0&0&-2&1&0&0 \\                                                                                                                                                                  
        0&0&0&0&0&0&1&0&0&0&0&0&0&0&0&1&-2&0&0  \\                                                                                                                                                                 
        0&0&0&0&0&0&0&0&0&0&0&0&0&0&0&0&0&-2&1 \\                                                                                                                                                                  
        0&0&0&0&0&0&1&0&0&0&0&0&0&0&0&0&0&1&-2 \\                                                                                                                                                                  
\end{smallmatrix}                                                                                                                                                                                                  
\right ).\]

We get $\det(NS(20))=12\times 3$ and applying Shimada's lemma \ref{lem:gram}, the discriminant form $G_{NS(20)}\simeq \mathbb Z/3 \oplus \mathbb Z/12$ is generated by vectors $L_1$ and $L_2$ satisfying $q_{L_1}=0$, $q_{L_2}=-\frac{11}{12}$ and $b(L_1,L_2)=\frac{1}{3}$.
Denoting $M(20)$ the following Gram matrix of the lattice $U(3)\oplus \langle 4 \rangle$,
\[M(20)=\begin{pmatrix}                                                                                                                                                                                            
        0 & 0 & 3\\                                                                                                                                                                                                
        0 & 4 & 0\\                                                                                                                                                                                                
        3 & 0 & 0\\                                                                                                                                                                                                
\end{pmatrix},\]
we find for generators of its discriminant form the vectors
\[g_1=\begin{pmatrix}                                                                                                                                                                                              
        0\\                                                                                                                                                                                                        
        0\\                                                                                                                                                                                                        
        \frac{1}{3}\\                                                                                                                                                                                              
\end{pmatrix}              \qquad            g_2=\begin{pmatrix}                                                                                                                                                   
        \frac{1}{3}\\                                                                                                                                                                                              
        \frac{1}{4}\\                                                                                                                                                                                              
        \frac{1}{3}\\                                                                                                                                                                                              
\end{pmatrix}                                                                                                                                                                                                      
\]
satisfying $q_{g_1}=0$, $q_{g_2}=\frac{11}{12}$ and $b(g_1,g_2)=\frac{1}{3}$. We deduce that $M(20)$ is the transcendental lattice of the $K3$ surface with elliptic fibration $H_{\#20}(k)$.

A Weierstrass equation of the $3$-torsion, rank $1$, elliptic fibration $\#19$ can be written as
\[E_{\#19}(k) \qquad y^2+ktxy+t^2(t^2+kt+1)y=x^3\]
with singular fibers $2IV^*(0,\infty)$, $2I_3(t^2+kt+1)$, $2I_1(k^3t-27kt-27t^2-27)$ and infinite point $P=(-t^2,-t^2)$ of height $h(P)=\frac{4}{3}$.
Its $3$-isogenous elliptic fibration $E_{\#19}(k)/\langle (0,0) \rangle $ has a Weierstrass equation
\[H_{\#19}(k) \qquad Y^2-3ktXY-Yt^2(27t^2-k(k^2-27)t+27)=X^3.\]
It is a $3$-torsion, rank $1$, elliptic fibration of a $K3$-surface with Picard number $19$ and singular fibers $2IV^*(0,\infty)$, $2I_3(27t^2-k(k^2-27)t+27)$, $2I_1(t^2+kt+1)$ and infinite order point $Q$ with\
 $x$-coordinate $x_Q=-3-3kt-(k^2+3)t^2-3kt^3-3t^4$ and height $h(Q)=4$. This point $Q$ is the image of the point $P$ in the $3$-isogeny and non $3$-divisible, hence generator of the non torsion part of the Mordell-Weill lattice. We deduce the discriminant of this $K3$-surface $\frac{3\times 3 \times 3 \times 3\times 4}{3 \times 3}=12\times 3$.

Applying Shioda's result, we order the components of the singular fibers as, $s_0$, $F$, $\theta_{0,i}$, $1\leq i \leq 6$, $\theta_{\infty,i}$, $1\leq i \leq 6$, $\theta_{t_0,i}$, $1\leq i \leq 2$, $s_3$, $\theta_{t_1,2}$, $s_{\infty}$,where $s_0$, $s_3$, $s_{\infty}$ denotes respectively the zero, $3$-torsion and infinite section, $F$ the generic section, $t_0$ and $t_1$ being roots of $27t^2-k(k^2-27)t+27$. The numbering of components of $IV^*$ is done using Bourbaki's notations. It follows the Gram matrix $NS(19)$ of the corresponding $K3$ surface

\[ NS(19)=                                                                                                                                                                                                         
\left (                                                                                                                                                                                                            
\begin{smallmatrix}                                                                                                                                                                                                
-2&1&0&0&0&0&0&0&0&0&0&0&0&0                                                                                                                                                                                       
&0&0&0&0&0\\1&0&0&0&0&0&0&0&0&0&0&0&0&0&0&0&1&0&1                                                                                                                                                                  
\\0&0&-2&0&0&1&0&0&0&0&0&0&0&0&0&0&0&0&0                                                                                                                                                                           
\\0&0&0&-2&1&0&0&0&0&0&0&0&0&0&0&0&1&0&0                                                                                                                                                                           
\\0&0&0&1&-2&1&0&0&0&0&0&0&0&0&0&0&0&0&0                                                                                                                                                                           
\\0&0&1&0&1&-2&1&0&0&0&0&0&0&0&0&0&0&0&0                                                                                                                                                                           
\\0&0&0&0&0&1&-2&1&0&0&0&0&0&0&0&0&0&0&0                                                                                                                                                                           
\\0&0&0&0&0&0&1&-2&0&0&0&0&0&0&0&0&0&0&0                                                                                                                                                                           
\\0&0&0&0&0&0&0&0&-2&0&0&1&0&0&0&0&0&0&0                                                                                                                                                                           
\\0&0&0&0&0&0&0&0&0&-2&1&0&0&0&0&0&1&0&0                                                                                                                                                                           
\\0&0&0&0&0&0&0&0&0&1&-2&1&0&0&0&0&0&0&0                                                                                                                                                                           
\\0&0&0&0&0&0&0&0&1&0&1&-2&1&0&0&0&0&0&0                                                                                                                                                                           
\\0&0&0&0&0&0&0&0&0&0&0&1&-2&1&0&0&0&0&0                                                                                                                                                                           
\\0&0&0&0&0&0&0&0&0&0&0&0&1&-2&0&0&0&0&0                                                                                                                                                                           
\\0&0&0&0&0&0&0&0&0&0&0&0&0&0&-2&1&0&0&0                                                                                                                                                                           
\\0&0&0&0&0&0&0&0&0&0&0&0&0&0&1&-2&1&0&0                                                                                                                                                                           
\\0&1&0&1&0&0&0&0&0&1&0&0&0&0&0&1&-2&1&2                                                                                                                                                                           
\\0&0&0&0&0&0&0&0&0&0&0&0&0&0&0&0&1&-2&0                                                                                                                                                                           
\\0&1&0&0&0&0&0&0&0&0&0&0&0&0&0&0&2&0&-2\                                                                                                                                                                          
\end{smallmatrix}                                                                                                                                                                                                  
\right ) .                                                                                                                                                                                                          
\]
Its determinant satisfies
$\det(NS(19))=12\times 3$ and according to Shimada's lemma \ref{lem:gram}, $G_{NS(19)}\simeq \mathbb Z/3 \oplus \mathbb Z/12$ is generated by vectors $M_1$ and $M_2$ satisfying $q_{M_1}=-\frac{2}{3}$, $q_{M_2}=\frac{5}{12}$ an\
d $b(M_1,M_2)=0$. We find also generators for the transcendental discriminant form $M(20)$
\[h_1=\begin{pmatrix}                                                                                                                                                                                              
        \frac{1}{3}\\                                                                                                                                                                                              
        0\\                                                                                                                                                                                                        
        \frac{1}{3}\                                                                                                                                                                                               
\end{pmatrix}              \qquad            h_2=\begin{pmatrix}                                                                                                                                                   
        \frac{1}{3}\\                                                                                                                                                                                              
        \frac{1}{4}\\                                                                                                                                                                                              
        -\frac{1}{3}\\                                                                                                                                                                                             
\end{pmatrix}                                                                                                                                                                                                      
\]
satisfying $q_{h_1}=\frac{2}{3}$, $q_{h_2}=-\frac{5}{12}$ and $b(h_1,h_2)=0$. We deduce that $M(20)$ is also the transcendental lattice of the $K3$ surface with elliptic fibration $H_{\#19}(k)$.

It follows the discriminant form of its N\'eron-Severi lattice, 

$G_{NS(19)}=\mathbb Z/3\mathbb Z(-\frac{2}{3})\oplus \mathbb Z/12\mathbb Z(\frac{5}{12})$,
 which is also $G_{NS(20)}$ 
 since the generators $L'_1=15L_1+4L_2$, 
 $L'_2=L_1-L_2$ 
 satisfy $q_{L'_1}=-\frac{2}{3}$, 
 $q_{L'_2}=\frac{5}{12}$, $b(L'_1,L'_2)=0$.

\end{proof}

We can prove the following specializations for $k=2$ and $k=10$.
\begin{theorem}\label{theo8.2}
For $k=2$, the $K3$ surface $N_2$ is $Y_{10}$ with transcendental lattice $[6 \quad 0 \quad 12]=T(Y_2)[3]$.

For $k=10$ the $K3$ surface $N_{10}$ is the $K3$-surface with discriminant $72$ and transcendental lattice $[4 \quad 0 \quad 18\
]$.

\end{theorem}
\begin{proof}
To prove that $N_2=Y_{10}$ it is sufficient to prove that $H_{\#20}(2)$ is an elliptic fibration of $Y_{10}$ since, by the previous theorem, $H_{\#19}(2)$ is another fibration of the same $K3$-surface.

But we see easily that $H_{\#20}(2)$ is the $6$-torsion extremal fibration $8$ of $Y_{10}$.

Similarly, $N_{10}$ is the $K3$ surface with transcendental lattice $[4 \quad 0 \quad 18]$ since $H_{\#20}(10)$ is a fibration of that surface (see  section 8.10).

 \end{proof}

\begin{theorem}
Define  $Y_{k}^{\left(  3\right)  }$ the elliptic surface  obtained by the
base change $\tau$ of the elliptic fibration of $Y_{k}$ with two singular
fibers of type $II^{\ast},$ where $\tau$ is the morphism given by $u\mapsto
h=u^{3}.$ Then 
the $K3$ surface $Y_{k}^{\left(  3\right)  }$ has a genus one fibration
without section such that its Jacobian variety 
satisfies $J\left(  Y_{k}^{\left(  3\right)  }\right)  =N_{k}.$
\end{theorem}
\begin{proof}
Recall a  Weierstrass equation for fibration $\#19$
\[
Y^{2}+tkYX+t^{2}\left(  t+s\right)  \left(  t+1/s\right)  Y=X^{3}%
\]
where $k=s+\frac{1}{s}.$ The fibration of $Y_{k}$ with two singular fibers
$II^{\ast}$ can be  obtained  with the parameter $h_{k}=\frac{Y}{\left(
t+s\right)  ^{2}}$ (\cite{B-L} Table 3). The surface  $Y_{k}^{\left(  3\right)  }$ is
  defined by $h=u^{3}$ and has then the following equation
\[
u^{3}s\left(  t+s\right)  +tkuWs+t^{2}\left(  ts-1\right)  -W^{3}s=0,
\]
where $X=\left(  t+s\right)  uW$.

We consider the fibration
\begin{align*}
Y_{k}^{\left(  3\right)  } &  \rightarrow\mathbb{P}^{1}\\
\left(  u,t,W\right)   &  \mapsto t;
\end{align*}
this is a genus one fibration since we have a cubic equation in $u,W.$

However, this fibration seems to have no section. Nevertheless, taking its
Jacobian fibration produces an elliptic fibration with section and the same
fiber type.

If we make a base change of this fibration: $\left(  t+s\right)  =m^{3}$ then
we obtain the following elliptic fibration with $U=um.$%
\[
U^{3}sm-\left(  s-m^{3}\right)  ksWU-W^{3}sm-\left(  s-m^{3}\right)
^{2}\left(  s^{2}-sm^{3}-1\right)  m=0.
\]
The transformation
\begin{align*}
W &  =\frac{-24y+12(s^{2}+1)(s-m^{3})x+\left(  s-m^{3}\right)  ^{2}Q}{18sm(4x-m^{2}(s^{2}+1)^{2})}\\
U &  =\frac{-24y-12(s^{2}+1)(s-m^{3})x+\left(  s-m^{3}\right)  ^{2}Q}{18sm(4x-m^{2}(s^{2}+1)^{2})}%
\end{align*}
where $Q=\left(
108m^{6}s^{3}+(s^{6}+105s^{2}-111s^{4}-1)m^{3}+s(s^{2}+1)^{3}\right)$
gives a Weierstrass equation, the point $\pi_{3}$ of $x$ coordinate $\frac
{1}{4}m^{2}(s^{2}+1)^{2}$ is a $3$-torsion point. Taking again $m^{3}=\left(
t+s\right)  ,$  and $\pi_{3}=\left(  X=0,Y=0\right)  ,$ we recover a
Weierstrass equation for the $3$-isogenous fibration  $\#19$%

\[
Y^{2}-3tkYX-t^{2}\left(  27t^{2}-k\left(  k^{2}-27\right)  t+27\right)
Y=X^{3}%
\]
hence a fibration of $N_{k}.$
Recall that the transcendental lattice of $Y_k^{(3)}$ is $T(Y_k)[3]$ \cite{Sh}. 
\end{proof}

\subsection{3-isogenies of $Y_2$}

\begin{theorem}
The $K3$ surface $Y_{2}$ has $4$ elliptic fibrations with $3$-torsion, two of them are specializations. The $3$-isogenies induce elliptic fibrations of $Y_{10}.$ 
\end{theorem}

\begin{proof}
Recall the results, given in \cite{BL}, for the $4$ elliptic fibrations with $3$-torsion. %

\[%
\begin{array}
[c]{ccc}%
&
\begin{array}
[c]{c}%
\text{Weierstrass Equation}\\
\text{Singular Fibers}%
\end{array}
& \text{Rank}\\
\hline
\#20\left(  7-w\right)   &
\begin{array}
[c]{c}%
Y^{2}-\left(  w^{2}+2\right)  YX-w^{2}Y=X^{3}\\
I_{12}\left(  \infty\right)  ,\quad I_{6}\left(  0\right)  ,\quad2I_{2}\left(
\pm1\right)  ,\quad2I_{1}%
\end{array}
& 0\\
\hline
\#19\left(  8-b\right)   &
\begin{array}
[c]{c}%
Y^{2}+2bYX+b^{2}\left(  b+1\right)  ^{2}Y=X^{3}\\
2IV^{\ast}\left(  \infty,0\right)  ,\quad I_{6}\left(  -1\right)  ,\quad2I_{1}%
\end{array}
& 1\\
\hline
20-j &
\begin{array}
[c]{c}%
Y^{2}-4\left(  j^{2}-1\right)  YX+4\left(  j+1\right)  ^{2}Y=X^{3}\\
I_{12}\left(  \infty\right)  ,\quad IV^{\ast}\left(  -1\right)  ,\quad
I_{2}\left(  -\frac{1}{2}\right)  ,\quad2I_{1}%
\end{array}
& 0\\
\hline
21-c &
\begin{array}
[c]{c}%
Y^{2}+\left(  c^{2}+5\right)  YX+Y=X^{3}\\
I_{18}\left(  \infty\right)  ,\quad6I_{1}%
\end{array}
& 1
\\\hline
\end{array}
\]

From the previous paragraph we compute the $3$-isogenous elliptic fibrations
named $H_{w},H_{b}$ $,H_{j}$ and $H_{c}$ and given in the next table. To simplify we denote $H_w$ instead of $H_{\#20}(2)$ and $H_b$ instead of $H_{\#20}(2).$

We know from Theorem \ref{theo8.2} that $H_w$ and $H_b$ are elliptic fibrations of $Y_{10}$; we present here a proof for the two others 
and remarks using ideas from \cite{Kuw}, \cite{Ku-Ku}, \cite{Sh},\cite{Sh1}.  

\[%
\begin{array}
[c]{cc}
&
\begin{array}
[c]{c}%
\text{Weierstrass Equation}\\
\text{Singular Fibers}%
\end{array}
\\
\hline
H_{w} &
\begin{array}
[c]{c}%
Y^{2}+3\left(  w^{2}+2\right)  YX+\left(  w^{2}+8\right)  \left(
w^{2}-1\right)  ^{2}Y=X^{3}\\
I_{4}\left(  \infty\right)  ,\quad2I_{6}\left(  \pm1\right)  ,\quad
I_{2}\left(  0\right)  ,\quad2I_{3}%
\end{array}
\\
\hline
H_{b} &
\begin{array}
[c]{c}%
Y^{2}-6bYX+b^{2}\left(  27b^{2}+46b+27\right)  Y=X^{3}\\
2IV^{\ast}\left(  \infty,0\right)  ,\quad2I_{3},\quad I_{2}\left(  -1\right)
\end{array}
\\
\hline
H_{j} &
\begin{array}
[c]{c}%
Y^{2}+12\left(  j^{2}-1\right)  YX+4\left(  4j^{2}-12j+11\right)  \left(
j+1\right)  ^{2}\left(  2j+1\right)  ^{2}Y=X^{3}\\
I_{4}\left(  \infty\right)  ,\quad VI^{\ast}\left(  -1\right)  ,\quad
I_{6}\left(  \frac{-1}{2}\right)  ,\quad2I_{3}%
\end{array}
\\
\hline
H_{c} &
\begin{array}
[c]{c}%
Y^{2}-3\left(  c^{2}+5\right)  YX-\left(  c^{2}+2\right)  \left(
c^{2}+c+7\right)  \left(  c^{2}-c+7\right)  Y=X^{3}\\
I_{6}\left(  \infty\right)  ,\quad6I_{3}%
\end{array}
\\\hline
\end{array}
\]

Recall that the transcendental lattice of $Y_{10}$ is $T\left(  Y_{2}\right)[  3] =T\left(  Y_{10}\right)  .$ Notice that the surface $Y_{2}$
has an elliptic fibration with singular fibers $2II^{\ast}\left(
\infty,0\right)  ,I_{2},2I_{1}$ and Weierstrass equation $y^{2}=x^{3}%
-\frac{25}{3}x+h+\frac{1}{h}-\frac{196}{27}$ or $y^{2}=z^{3}-5z^{2}%
+\frac{\left(  h+1\right)  ^{2}}{h}$ with $x=z-\frac{5}{3}$. The base change of
degree $3,$ $h=u^{3}$ ramified at the two fibers $II^{\ast}$ induces an
elliptic fibration of the resulting $K3$ surface named $Y_{2}^{\left(
3\right)  }$ in \cite{Kuw}. As the transcendental lattice of the surface $Y_{2}^{\left(
3\right)  }$ is $T\left(  Y_{2}\right)[3]$  \cite{Sh}, this
surface $Y_{2}^{\left(  3\right)  }$ is $Y_{10.}$ Moreover we can precise the
fibration obtained: a Weierstrass equation is
\begin{align}
E_{u}  &  :Y^{2}=X^{3}-5u^{2}X^{2}+u^{3}(u^{3}+1)^{2}\label{M}\\
&  2I_{0}^{\ast}\left(  \infty,0\right)  ,3I_{2}\left(  u^{3}+1\right)
,6I_{1}\text{ rank }7.\nonumber
\end{align}

Now we are going to show that every elliptic fibration of $Y_2$ with $3$-torsion is
linked to the elliptic fibration of $Y_2$ with $2II^{\ast}\left(  \infty,0\right)  ,I_{2},2I_{1}.$

For the fibration $20-j$ we can obtain the elliptic fibration $2II^{\ast
}\left(  \infty,0\right)  ,I_{2},2I_{1}$ from the Weierstrass equation given
in the table and the  elliptic parameter $h=Y.$ So the fibration $20-j$
induces a fibration on $Y_{10}$ with parameter $j$ and an equation obtained after substitution of $Y$
by $u^{3}$. So, with the previous computations of 8.1.3 this is the
$3$-isogenous to $20-j.$

The same proof can be done for fibration $21-c.$ Moreover we can remark that
the $3$-isogenous to $21-c$ fibration has an equation
\[
W^{3}+\left(  c^{2}+5\right)  ZW+1-Z^{3}=0.
\]
Since the general elliptic surface with torsion $\left(  \mathbb{Z}%
/3\mathbb{Z}\right)  ^{2}$ is $x^{3}+y^{3}+t^{3}+3kxyt=0$, we deduce that the
torsion on the fibration induced on $Y_{10}$ is $\left(
\mathbb{Z}/3\mathbb{Z}\right)  ^{2}.$

For the fibration $\#19\left(  8-b\right)  $ we obtain the elliptic fibration
$2II^{\ast}\left(  \infty,0\right)  ,I_{2},2I_{1}$ from the Weierstrass
equation given in the table and the \ elliptic parameter $h=\frac{Y}{\left(
b+1\right)  ^{2}}.$ Substituing $h$ by $u^{3}$ \ and defining $W$ as
$X=\left(  b+1\right)  ^{2}uW$ we obtain a cubic equation in $u$ and $W$ with
a rational point $u=1,W=1,$ so an elliptic fibration of $Y_{10}$%
\begin{align*}
Y_{10}  &  \rightarrow\mathbb{P}_{1}\\
\left(  u,b,W\right)   &  \mapsto b.
\end{align*}
Computation gives the $3$-isogenous elliptic curve to $\#19\left(  8-b\right)  .$

For the last fibration $\#20\left(  7-w\right)$ we have previously shown the result in  Theorem 8.2. 
The relation with the fibration $2II^*,I_2,2I_1$ is less  direct. 
\end{proof}

Notice the two following results: 
with the previous method we can construct two elliptic fibrations of $Y_{10}$ of rank $4$.

First from Weierstrass equation $\#20\left(  7-w\right)$ and with the parameter $m=Y$ we
obtain the fibration $\#1\left(  11-f\right)  $ of $Y_{2}$ with singular fibers $II^{\ast
}\left(  \infty\right)  ,III^{\ast}\left(  0\right)  ,$ $I_{4}\left(  1\right)
,I_{1}\left(  \frac{32}{27}\right)  .$ A Weierstrass equation
\[
E_{m}:Y_{1}^{2}=X_{1}^{3}-m(2m-3)X_{1}^{2}+3m^{2}\left(  m-1\right)  ^{2}%
X_{1}+m^{3}\left(  m-1\right)  ^{4}%
\]
is obtained with the following transformations%
\begin{align*}
m  &  =Y,\quad X_{1}=-\frac{Y\left(  Y-1\right)  ^{2}}{X+1},\quad Y_{1}%
=w\frac{Y^{2}\left(  Y-1\right)  ^{2}}{X+1}\\
w  &  =\frac{-Y_{1}}{X_{1}m},\quad X=-\frac{X_{1}+m\left(  m-1\right)  ^{2}%
}{X_{1}},\quad Y=m.
\end{align*}

The base change $m=u^{\prime3}$  gives an elliptic
fibration of $Y_{10}$ with singular fibers $I_{0}^{\ast}\left(  \infty\right)
$,  $III\left(  0\right)  $,  $3I_{4}\left(  1,j,j^{2}\right)  $,  $3I_{1}$ , rank $4$,
a Weierstrass equation and sections
\begin{align*}
y^{\prime}  &  =x^{\prime3}+u^{\prime2}x^{\prime2}+2u^{\prime}\left(
u^{\prime3}-1\right)  x^{\prime}+u^{\prime3}\left(  u^{\prime3}-1\right)  ^{2}\\
P  &  =\left(  x_{P}\left(  u^{\prime}\right)
,y_{P}\left(  u^{\prime}\right)  \right)  =\left(  -\left(  u^{\prime
3}-1\right)  ,\left(  u^{\prime}-1\right)  ^{2}\left(  u^{\prime2}+u^{\prime
}+1\right)  \right) \\
Q  &  =\left(  x_{Q}\left(  u^{\prime}\right)  ,y_{Q}\left(  u^{\prime
}\right)  \right)  =\left(  -\left(  u^{\prime}+2\right)  \left(  u^{\prime
2}+u^{\prime}+1\right)  ,2i\sqrt{2}\left(  u^{\prime2}+u^{\prime}+1\right)
^{2}\right) . \label{rk41}
\end{align*}
Also we have the points $P^{\prime}$ with $x_{P^{\prime}}=jx_{P}\left(
ju^{\prime}\right)  $ and $Q^{\prime}$ with $x_{Q^{\prime}}=jx_{Q}\left(
ju^{\prime}\right)  $ (with $j^{3}=1$). As explained in the next paragraph we can
compute the height matrix and show that the Mordell-Weil  lattice is generated
by $P,P^{\prime},Q,Q^{\prime}$  and is equal to $A_2\left(  \frac{1}%
{4}\right)  \oplus A_2\left( \frac{1}{2}\right).$

The second exemple is obtained from $\#20\left(  7-w\right)$ with the parameter 
$n=\frac{Y}{t^{2}}$ we obtain  the fibration
$\#9\left(  12-g\right)  $%
\begin{align*}
E_{g}  & :y^{2}=x^{3}+4x^{2}n^{2}+n^{3}(n+1)^{2}x
\end{align*}
with the following transformation
\begin{align*}
X  & =\frac{x^{2}\left(  x+2n^{2}\right)  \left(  n-1\right)  }{y^{2}}%
,Y=\frac{1}{n}\frac{x^{2}\left(  x+2n^{2}\right)  ^{2}}{y^{2}},t=\frac
{x\left(  x+2n^{2}\right)  }{ny}\\
x  & =\frac{Y^{2}\left(  Y-2X-t^{2}\right)  }{Xt^{4}},y=\frac{Y^{3}\left(
Y-t^{2}\right)  \left(  Y-2X-t^{2}\right)  }{X^{2}t^{7}},n=\frac{Y}{t^{2}}.%
\end{align*}
Notice that if $n=\frac{Y}{t^{2}}=v^{3}$ in $E_{w}$ then we have the equation
of $H_{w}.$ More precisely if $X=tQv$, the equation  becomes%
\[
-tv^{3}+t^{2}Qv+2Qv+Q^{3}+t=0,
\]
a cubic equation in $Q$ and $v$ with a rational point $v=1,Q=0.$ Easely we
obtain $H_{w}.$
So in the Weierstrass equation $E_{g}$, if we replace the parameter $n$ by
$g^{3}$ we obtain the following fibration of $Y_{10}$%
\begin{align}
y^{2}  & =x^{3}+4g^{2}x^{2}+g\left(  g+1\right)  ^{2}\left(  g^{2}-g+1\right)
^{2}x  \label{rk42}
\end{align}
with singular fibers $2III\left(  0,\infty\right)  ,3I_{4}\left(
-1,g^{2}-g+1\right)  ,3I_{1}\left(  1,g^{2}+g+1\right)  $ and rank $4.$
Notice the two infinite sections with $x$ coordinates $(t+1)^2(t^2-t+1)$ and $-\frac{1}{3}(t-1)^2(t^2-t+1).$

\subsection{ Mordell-Weil group of $E_{u}$ \label{P}}

The aim of this paragraph is to construct generators of the Mordell-Weil
lattice of the previous fibration of rank $7$ with Weierstrass equation%

\begin{align*}
E_{u} &  :Y^{2}=X^{3}-5u^{2}X^{2}+u^{3}(u^{3}+1)^{2}\\
&  2I_{0}^{\ast}\left(  \infty,0\right)  ,3I_{2}\left(  u^{3}+1\right)
,6I_{1}.
\end{align*}

Notice that the $j$-invariant of $E_{u}$ is invariant by the two
transformations $u\mapsto\frac{1}{u}$ and $u\mapsto ju.$ These automorphisms
of the base $\mathbb{P}^{1}$ of the fibration $E_{u}$ can be extended to the
sections as explained below.

Let $S_{3}=<\gamma,\tau;\gamma^{3}=1,\tau^{2}=1>$ be the non abelian group of
order $6$ and define an action of $S_{3}$ on the sections of $E_{u}$ by%
\begin{align*}
&  \left(  X(u),Y(u)\right)  \overset{\tau}{\mapsto}\left(  u^{4}X\left(
\frac{1}{u}\right)  ,u^{6}Y\left(  \frac{1}{u}\right)  \right)  \\
&  \left(  X\left(  u\right)  ,Y\left(  u\right)  \right)  \overset{\gamma
}{\mapsto}\left(  jX\left(  ju\right)  ,Y\left(  ju\right)  \right).
\end{align*}

To obtain generators of $E_{u}$ following Shioda \cite{Sh} we use the rational
elliptic surface $X^{\left(  3\right)  +}$ with $\sigma=u+\frac{1}{u}$ and a
Weierstrass equation%
\begin{align*}
E_{\sigma}  &  :y^{2}=x^{3}-5x^{2}+\left(  \sigma-1\right)  ^{2}\left(
\sigma+2\right) \\
&  I_{0}^{\ast}\left(  \infty\right)  ,I_{2}\left(  -1\right)  ,4I_{1}%
\end{align*}
of rank $3.$

The Mordell-Weil lattice of a rational elliptic surface is generated by
sections of the form $(a+b\sigma+c\sigma^{2},d+e\sigma+f\sigma^{2}+g\sigma
^{3}).$ Moreover since we have a singular fiber of type $I_{0}^{\ast}$ at
$\infty$ the coefficients $c$ and $f,g$ are $0$ \cite{Ku-Ku}. So after an easy
computation we find the $3$ sections \ (with $j^{3}=1,i^{2}=-1).$%
\begin{align*}
q_{1}  & =\left(  -\left(  \sigma-1\right)  ,i\sqrt{2}\left(  \sigma-1\right)
\right)  \quad\\
q_{2}  & =\left(  -j\left(  \sigma-1\right)  ,\left(  3+j\right)  \left(
\sigma-1\right)  \right)  \quad q_{3}=\left(  -j^{2}\left(  \sigma-1\right)
,\left(  3+j^{2}\right)  \left(  \sigma-1\right)  \right).
\end{align*}
These sections give the sections $\pi_{1},\pi_{2},\pi_{3}$ on $E_{u}$ which
are fixed by $\tau$.
\[%
\begin{array}
[c]{c}%
\pi_{1}=\left(  -u\left(  u^{2}-u+1\right)  ,i\sqrt{2}u^{2}\left(
u^{2}-u+1\right)  \right)  \\
\pi_{2}=\left(  -ju\left(  u^{2}-u+1\right)  ,\left(  3+j\right)  u^{2}\left(
u^{2}-u+1\right)  \right)  \\
\pi_{3}=\left(  -j^{2}u\left(  u^{2}-u+1\right)  ,\left(  3+j^{2}\right)
u^{2}\left(  u^{2}-u+1\right)  \right).
\end{array}
\]

We notice $\rho_{i}=\gamma\left(  \pi_{i}\right)  $ and $\mu_{i}=\gamma
^{2}\left(  \pi_{i}\right)  $ for $1\leq i\leq3$ which give $9$ rational
sections with some relations.

Moreover we have another section from the fibration $E_{h}$ of rank
$1.$

The point of $x$ coordinate $\frac{1}{16}\left(  h^{2}+\frac{1}{h^{2}}\right)
-h-\frac{1}{h}+\frac{29}{24}$ is defined on $\mathbb{Q}\left(  h\right)  .$
Passing to $E_{u}$ we obtain $\omega=$
\[
\left(  \frac{1}{16}\frac{(  1-16u^{3}+46u^{6}-16u^{9}%
+u^{12})  }{u^{4}},-\frac{1}{64}\frac{(  u^{6}-1)  (
1-24u^{3}+126u^{6}-24u^{9}+1)  }{u^{6}}\right ).
\]

We hope to get a generator system with $\pi_{i},\rho_{i}$ and $\omega$ so we
have to compute the height-matrix. The absolute value of its determinant  is
$\frac{81}{16}.$ Since the discriminant of the surface is $72,$ we obtain a
subgroup of index $a$ with $\frac{81}{16}\times\frac{1}{a^{2}}\times2^{3}%
4^{2}=72$ so $a=3$.

After some specialization of $u$ $\in\mathbb{Z}$ (for example if $u=11$, $E_{u}$
has rank $3$ on $\mathbb{Q})$ we find other sections with $x$ coordinate
of the shape \ $\ \left(  au+b\right)  \left(  u^{2}-u+1\right)  $
\begin{align*}
\mu &  =\left(  -\left(  u-1\right)  \left(  u^{2}-u+1\right)  ,-\left(
u^{2}-u+1\right)  \left(  u^{2}+2u-1\right)  \right)  \\
\mu_{1} &  =\left(  -\left(  u-9\right)  \left(  u^{2}-u+1\right)  ,\left(
u^{2}-u+1\right)  \left(  5u^{2}-18u+27\right)  \right)  \\
\mu_{2} &  =\left(  -\left(  u+\frac{1}{3}\right)  \left(  u^{2}-u+1\right)
,\frac{i\sqrt{3}}{9}\left(  u^{2}-u+1\right)  \left(  9u^{2}+4u+1\right)
\right).
\end{align*}

We deduce the relations
\begin{align*}
3\mu &  =\omega+\pi_{2}-\gamma\left(  \pi_{3}\right)  +\pi_{3}-\gamma
^{2}\left(  \pi_{2}\right)  \\
&  =\omega+2\pi_{2}+\gamma\left(  \pi_{2}\right)  +\pi_{3}-\gamma\left(
\pi_{3}\right)
\end{align*}
so, the Mordell-Weil lattice is generated by $\pi_{j},\rho_{j}=\gamma\left(
\pi_{j}\right)  $ for $1\leq j\leq3$ and $\mu$ with Gram matrix%
\[
\left(
\begin{array}
[c]{ccccccc}%
1 & -\frac{1}{2} & 0 & 0 & 0 & 0 & 0\\
-\frac{1}{2} & 1 & 0 & 0 & 0 & 0 & 0\\
0 & 0 & 1 & -\frac{1}{2} & 0 & 0 & \frac{1}{2}\\
0 & 0 & -\frac{1}{2} & 1 & 0 & 0 & 0\\
0 & 0 & 0 & 0 & 1 & -\frac{1}{2} & \frac{1}{2}\\
0 & 0 & 0 & 0 & -\frac{1}{2} & 1 & \frac{-1}{2}\\
0 & 0 & \frac{1}{2} & 0 & \frac{1}{2} & \frac{-1}{2} & 2
\end{array}
\right)  .
\]

\subsection{A fibration for Theorem 7.1}
From the previous fibration $E_u$ we construct by a $2$-neigbour method a fibration with a $2$-torsion section used in Theorem 7.1.

We start from the Weierstrass equation (\ref{M})
\[
Y^{2}=X^{3}-5u^{2}X^{2}+u^{3}\left(  u^{3}+1\right)  ^{2}%
\]
and obtain another elliptic fibration with the parameter $m=\frac{X}{u\left(
u^{2}-u+1\right)  }$, which gives the Weierstrass equation
\begin{align}
E_{m}  &  :y^{2}=x^{3}-\left(  m^{3}+5m^{2}-2\right)  x^{2}+\left(
m^{3}+1\right)  ^{2}x \label{rk4}
\end{align}
with singular fibers 
$  I_{0}^{\ast}\left(  \infty\right) $, $3I_{4}\left(  m^{3}+1\right)
$,  $ I_{2}\left(  0\right)  $,  $4I_{1}\left(  1,-\frac{5}{3}%
,m^{2}-4m-4\right)$ 
 and rank $4.$ 
\begin{remark}
From this fibration with the parameter $q=\frac{y}{xm}$ we recover the
fibration $H_{c.}$
\end{remark}

\subsection{$3$-isogenies from $Y_{10}$}
\begin{theorem}
Consider the two $K3$ surfaces of discriminant $72$ and of transcendental lattice  $[4 \quad 0 \quad 18]$  or $[2 \quad 0 \quad 36]$. There exist elliptic fibrations of $Y_{10}$ with a $3$-torsion section which induce by $3$-isogeny an elliptic fibration of one of the surface or of the other.     
\end{theorem}

\begin{proof}
For the following elliptic fibrations of $Y_{10}$ with a $3$-torsion section, not specializations, given in the last paragraphs, we compute the N\'eron-Severi and
 transcendental lattices of the surface with the elliptic fibration induced by the $3$-isogeny. 
\end{proof}
\subsubsection{Fibration (11)}
In section $6$ we found the rank $4$ elliptic fibration $(11)$.

After a translation to put the $3$-torsion section in $\left(  0,0\right)  $
we obtain the following Weierstrass equation and also
$p_{1}$ and $p_{2}$ generators of the Mordell-Weil lattice
\begin{equation}%
\begin{array}
[c]{c}%
E_{11}:Y^{2}+\left(  t^{2}-4\right)  YX+t^{2}\left(  2t^{2}-3\right)
Y=X^{3}\\
p_{1}=\left(  6t^{2},27t^{2}\right)  ,p_{2}=\left(  6i\sqrt{3}t-3t^{2}%
,27t^{2}\right)  \\
2I_{6}\left(  \infty,0\right)  ,\quad2I_{3}\left(  2t^{2}-3\right)
,\quad2I_{2}\left(  \pm1\right)  ,\quad2I_{1}\left(  \pm8\right).
\end{array}
\end{equation}
We see that the $3$-isogenous elliptic fibration has a Weierstrass equation, generators of
Mordell-Weil lattice and singular fibers
\begin{align*}
H_{11}: &  Y^{2}-3\left(  t^{2}-4\right)  YX+\left(  t^{2}-1\right)
^{2}\left(  t^{2}-64\right)  Y=X^{3}\\
&  2I_{6}\left(  \pm1\right)  ,\quad2I_{3}\left(  \pm8\right)  ,\quad
2I_{2}\left(  \infty,0\right)  ,\quad2I_{1}\left(  2t^{2}-3\right)  \text{ of
rank }2.
\end{align*}
Notice the two sections
\begin{align*}
\pi_{1}  & =\left(  -\frac{1}{4}\left(  t^{2}-1\right)  \left(  t^{2}%
-64\right)  ,\frac{1}{8}\left(  t-8\right)  \left(  t-1\right)  \left(
t+1\right)  ^{2}\left(  t+8\right)  ^{2}\right)  \quad\\
\omega & =\left(  -7\left(  t^{2}-1\right)  ^{2},49\alpha\left(
t^{2}-1\right)  ^{3}\right)  \quad\text{where }49\alpha^{2}+20\alpha+7=0.
\end{align*}
So, computing the height matrix we see  the discriminant is $72.$ 

For each reducible fiber at $t=i$ we denote $\left(  X_{i},Y_{i}\right)  $ the
singular point of $H_{11}$
\[%
\begin{array}
[c]{ccc}%
t=\pm1 \quad t=\pm8 & t=0 & t=\infty\\
(  X_{\pm1}=0,Y_{\pm1}=0) & 
(  X_{0}=-16,Y_{0}=64)  & (  x_{\infty}=-1,y_{\infty
}=-1)\\
(  X_{\pm8}=0,Y\pm_{8}=0)& &
\end{array}
\]
where if $t=\infty$ we substitute $t =\frac{1}{T}$,
$x=T^{4}X,y=T^{6}Y.$ We notice also $\theta_{i,j}$ the $j$-th component
of the reducible fiber at $t=i.$ A section $M=\left(  X_{M},Y_{M}\right)  $
intersects the component $\theta_{i,0}$ if and only if $\left(  X_{M},Y_{M}\right)
\not \equiv\left(  X_{i},Y_{i}\right)  \operatorname{mod}\left(  t-i\right)
.$ Using the additivity on the component, we deduce that $\omega$ does not
intersect $\theta_{i,0},2\omega$ intersects $\theta_{i,0}$ and so $\omega$
intersects $\theta_{i,3}$ for $i=\pm1.$ Also $\omega$ intersects $\theta_{i,0}$
for $i=\pm8$ $,i=0$ and $\infty.$

For $\pi_{1}$ we compute $k\pi_{1}$ with $2\leq k\leq6.$ For $i=\pm1$, only
$6\pi_{1}$ intersects $\theta_{i,0}$ so $\pi_{1}$ intersects $\theta_{i,1}.$
(this choice 1, not 5, fixes the numbering of components). For $i=\pm8$, only
$3\pi_{1}$ intersects $\theta_{i,0},$ so $\pi_{i}$ intersects $\theta_{i,1}.$
Modulo $t$,\ we get $\pi_{1}=\left(  -16,64\right)  ,$ so $\pi_{1}$ intersects
$\theta_{0,1},$ and $\pi_{1}$ intersects $\theta_{\infty,0}$.

As for the $3$-torsion section $s_{3}=\left(  0,0\right)  $,  $s_{3}$
intersects $\theta_{i,2}$ or $\theta_{i,4}$ if $i=\pm1$. Computing $2\pi_{1}-s_{3}$, we see that $s_{3}$ intersects $\theta_{1,2}$ and $\theta_{-1,4}$.

For $i=\pm8$, we compute $\pi_{1}-s_{3},$ for $t=8$ and show that $s_{3}$
intersects $\theta_{8,2} $ and
$\theta_{-8,1}.$ For $t=0$ and $t=\infty\ s_{3}$ intersects \ the $0$ component.

So we can compute the relation between the section $s_{3}$ and the
$\theta_{i,j}$ and find  that $3s_{3}\approx -2\theta_{1,1}-4%
\theta_{1,2}-3\theta_{1,3}-2\theta_{1,4}-\theta
_{1,5}$. Thus, we can choose the following base of the N\'{e}ron-Severi lattice ordered as 
$s_{0},F,\theta_{1,j},$ with $1\leq j\leq4,$ $s_{3},\theta_{-1,k}$ with $1\leq
k\leq5,\theta_{8,k},k=1,2,$ $\theta_{-8,k},k=1,2$ and $\theta_{0,1},$
$\theta_{\infty,1},\omega,\pi_{1}.$

The last remark is that only the two sections $\omega$ and $\pi_{1}$
intersect. So we can write the Gram matrix $Ne$ of the N\'{e}ron-Severi
lattice,
\[
\left(
\begin{smallmatrix}
-2 & 1 & 0 & 0 & 0 & 0 & 0 & 0 & 0 & 0 & 0 & 0 & 0 & 0 & 0 & 0 & 0 & 0 & 0 &
0\\
1 & 0 & 0 & 0 & 0 & 0 & 1 & 0 & 0 & 0 & 0 & 0 & 0 & 0 & 0 & 0 & 0 & 0 & 1 &
1\\
0 & 0 & -2 & 1 & 0 & 0 & 0 & 0 & 0 & 0 & 0 & 0 & 0 & 0 & 0 & 0 & 0 & 0 & 0 &
1\\
0 & 0 & 1 & -2 & 1 & 0 & 1 & 0 & 0 & 0 & 0 & 0 & 0 & 0 & 0 & 0 & 0 & 0 & 0 &
0\\
0 & 0 & 0 & 1 & -2 & 1 & 0 & 0 & 0 & 0 & 0 & 0 & 0 & 0 & 0 & 0 & 0 & 0 & 1 &
0\\
0 & 0 & 0 & 0 & 1 & -2 & 0 & 0 & 0 & 0 & 0 & 0 & 0 & 0 & 0 & 0 & 0 & 0 & 0 &
0\\
0 & 1 & 0 & 1 & 0 & 0 & -2 & 0 & 0 & 0 & 1 & 0 & 0 & 1 & 1 & 0 & 0 & 0 & 0 &
0\\
0 & 0 & 0 & 0 & 0 & 0 & 0 & -2 & 1 & 0 & 0 & 0 & 0 & 0 & 0 & 0 & 0 & 0 & 0 &
1\\
0 & 0 & 0 & 0 & 0 & 0 & 0 & 1 & -2 & 1 & 0 & 0 & 0 & 0 & 0 & 0 & 0 & 0 & 0 &
0\\
0 & 0 & 0 & 0 & 0 & 0 & 0 & 0 & 1 & -2 & 1 & 0 & 0 & 0 & 0 & 0 & 0 & 0 & 1 &
0\\
0 & 0 & 0 & 0 & 0 & 0 & 1 & 0 & 0 & 1 & -2 & 1 & 0 & 0 & 0 & 0 & 0 & 0 & 0 &
0\\
0 & 0 & 0 & 0 & 0 & 0 & 0 & 0 & 0 & 0 & 1 & -2 & 0 & 0 & 0 & 0 & 0 & 0 & 0 &
0\\
0 & 0 & 0 & 0 & 0 & 0 & 0 & 0 & 0 & 0 & 0 & 0 & -2 & 1 & 0 & 0 & 0 & 0 & 0 &
1\\
0 & 0 & 0 & 0 & 0 & 0 & 1 & 0 & 0 & 0 & 0 & 0 & 1 & -2 & 0 & 0 & 0 & 0 & 0 &
0\\
0 & 0 & 0 & 0 & 0 & 0 & 1 & 0 & 0 & 0 & 0 & 0 & 0 & 0 & -2 & 1 & 0 & 0 & 0 &
1\\
0 & 0 & 0 & 0 & 0 & 0 & 0 & 0 & 0 & 0 & 0 & 0 & 0 & 0 & 1 & -2 & 0 & 0 & 0 &
0\\
0 & 0 & 0 & 0 & 0 & 0 & 0 & 0 & 0 & 0 & 0 & 0 & 0 & 0 & 0 & 0 & -2 & 0 & 0 &
1\\
0 & 0 & 0 & 0 & 0 & 0 & 0 & 0 & 0 & 0 & 0 & 0 & 0 & 0 & 0 & 0 & 0 & -2 & 0 &
0\\
0 & 1 & 0 & 0 & 1 & 0 & 0 & 0 & 0 & 1 & 0 & 0 & 0 & 0 & 0 & 0 & 0 & 0 & -2 &
1\\
0 & 1 & 1 & 0 & 0 & 0 & 0 & 1 & 0 & 0 & 0 & 0 & 1 & 0 & 1 & 0 & 1 & 0 & 1 & -2
\end{smallmatrix}
\right ).
\]

According to Shimada's lemma \ref{lem:gram},  $G_{Ne}\equiv \mathbb{Z/}2\mathbb{Z} \oplus\mathbb{Z/}36\mathbb{Z}$ is generated by the vectors $L_1$ and $L_2$ satisfying $q_{L_{1}}=-\frac{1}{2},$ $q_{L_{2}}=\frac{37}{36},$ and $b(L_1,L_2)=\frac{1}{2}.$

Moreover the following generators of the discriminant group of the lattice with Gram matrix 
$M_{18}=\left(
\begin{smallmatrix}
4 & 0\\
0 & 18
\end{smallmatrix}
\right)$ namely $f_1=(0,\frac{1}{2})$, $f_2=(\frac{1}{4},\frac{7}{18})$ verify $q_{f_1}=\frac{1}{2}$, $q_{f_2}=-\frac{37}{36}$ and $b(f_1,f_2)=-\frac{1}{2}.$

So the Gram matrix of the transcendental lattice of the surface is $M_{18}$.
\subsubsection{Fibration $H_c$}
We have shown along the proof of Theorem 8.3 that $H_c$ has a $(Z/3Z)^2$- torsion group 
and exhibited a 3-isogeny between some elliptic fibrations of $Y_{10}$ and $Y_2$. 
Notice that with the previous Weierstrass equation $H_{c}$ the point
$\sigma_{3}$ of $X$ coordinate $-\left(  c^{2}+c+7\right)  \left(
c^{2}-c+7\right)  $ defines a $3$-torsion section. After a translation to put this point in $(0,0)$ 
and scaling, we obtain a Weierstrass equation 
\[
Y'^2-(t^2+11)Y'X'-(t^2+t+7)(t^2-t+7)Y'=X'^3
.\]

The $3$-isogenous curve of
kernel $<\sigma_{3}>$ has a Weierstrass equation  %
\begin{align*}
y^{2}+3\left(  t^{2}+11\right)  xy+\left(  t^{2}+2\right)  ^{3}y &  =x^{3},
\end{align*}
with singular fibers $I_{2}\left(  \infty\right)  $, $2I_{9}\left(
t^{2}+2\right)  $, $4I_{1}\left(  t^{4}+13t^{2}+49\right)  .$ 

The section
$P_{c}=\left(  -\frac{1}{4}\left(  t^{4}+t^{2}+1\right)  ,-\frac{1}{8}\left(
t-j\right)  ^{3}\left(  t+j^{2}\right)  ^{3}\right)  $ where $j=\frac
{-1+i\sqrt{3}}{2},$ of infinite order, generates the Mordell-Weil lattice.

We consider the components of the reducible fibers in the following order
$\theta_{i\sqrt{2},j}$ $j\leq1\leq8,$ $\theta_{-i\sqrt{2},k}$ $1\leq k\leq8$,
$\theta_{\infty,1}$.

The $3$-torsion section $s_{3}=\left(  0,0\right)  $ and the previous
components are linked by the relation
\[
3s_{3}\approx -\theta_{i\sqrt{2},8}+\sum a_{i,j}\theta_{i,j}.%
\]
So we can replace, in the previous ordered sequence of components, the element
$\theta_{i\sqrt{2},8}$ by $s_{3}.$ We notice that $\left(  s_{3}.P_{c}\right)
=2,\left(  P_{c}.\theta_{\pm i\sqrt{2},0}\right)  =1$ and $\left(
P_{c}.\theta_{\infty,0}\right)  =1,$ so  the Gram matrix of the N\'{e}ron-Severi lattice is
\[
\left (\begin {smallmatrix} -2&1&0&0&0&0&0&0&0&0&0&0&0
&0&0&0&0&0&0&0\\1&0&0&0&0&0&0&0&0&1&0&0&0&0&0&0&0&0&0
&1\\0&0&-2&1&0&0&0&0&0&0&0&0&0&0&0&0&0&0&0&0
\\0&0&1&-2&1&0&0&0&0&0&0&0&0&0&0&0&0&0&0&0
\\0&0&0&1&-2&1&0&0&0&1&0&0&0&0&0&0&0&0&0&0
\\0&0&0&0&1&-2&1&0&0&0&0&0&0&0&0&0&0&0&0&0
\\0&0&0&0&0&1&-2&1&0&0&0&0&0&0&0&0&0&0&0&0
\\0&0&0&0&0&0&1&-2&1&0&0&0&0&0&0&0&0&0&0&0
\\0&0&0&0&0&0&0&1&-2&0&0&0&0&0&0&0&0&0&0&0
\\0&1&0&0&1&0&0&0&0&-2&0&0&1&0&0&0&0&0&0&2
\\0&0&0&0&0&0&0&0&0&0&-2&1&0&0&0&0&0&0&0&0
\\0&0&0&0&0&0&0&0&0&0&1&-2&1&0&0&0&0&0&0&0
\\0&0&0&0&0&0&0&0&0&1&0&1&-2&1&0&0&0&0&0&0
\\0&0&0&0&0&0&0&0&0&0&0&0&1&-2&1&0&0&0&0&0
\\0&0&0&0&0&0&0&0&0&0&0&0&0&1&-2&1&0&0&0&0
\\0&0&0&0&0&0&0&0&0&0&0&0&0&0&1&-2&1&0&0&0
\\0&0&0&0&0&0&0&0&0&0&0&0&0&0&0&1&-2&1&0&0
\\0&0&0&0&0&0&0&0&0&0&0&0&0&0&0&0&1&-2&0&0
\\0&0&0&0&0&0&0&0&0&0&0&0&0&0&0&0&0&0&-2&0
\\0&1&0&0&0&0&0&0&0&2&0&0&0&0&0&0&0&0&0&-2
\end {smallmatrix}\right).
\]
Its determinant is $-72.$ 
According to Shimada's lemma \ref{lem:gram}, $G_{NS}=\mathbb{Z}/2\mathbb{Z}%
\oplus\mathbb{Z}/36$ is generated by the vectors $L_{1}$ and $L_2$ satisfying 
$q_{L_1}=\frac{-1}{2},q_{L_2}=\frac{5}{36}$ and $b(L_1,L_2)=\frac{1}{2}.$

Moreover the following generators of the discriminant group of the lattice with Gram matrix
$M_{36}=\left(\begin{smallmatrix}
2&0\\
0&36
\end{smallmatrix}\right)$, namely $f_1=(\frac{1}{2},0)$, $f_2=(\frac{-1}{2},\frac{-7}{36})$ verify $q_{f_1}=\frac{1}{2}$, $q_{f_2}=-\frac{5}{36}$ and 
$b(f_1,f_2)=\frac{-1}{2}.$

So the transcendental lattice is $M_{36}=\left(
\begin{array}
[c]{cc}%
2 & 0\\
0 & 36
\end{array}
\right).  $

\subsubsection{Fibration $89$ of rank $0$}

The fibration $89$ of rank $0$ has a $3$-torsion section. A Weierstrass equation for the $3$-isogenous fibration is 
\begin{align*}
 Y^{2}+\left(  -27t^{2}-18t+27\right)  YX+27\left(  4t+3\right)  \left(
5t-3\right)  ^{2}Y=X^{3}
\end{align*}
with singular fibers $ I_{9}\left(  \infty\right) ,I_{6}\left(  \frac{3}{5}\right),I_{4},\left(  0\right) ,I_{3}\left(  -\frac{3}{4}\right),2I_1.$

From singular fibers, torsion and rank we see in Shimada-Zhang table \cite{Shim} that it is the  $n^\circ 48$ case. So the transcendental lattice of the surface is 
$\left(
\begin{smallmatrix}
4 & 0\\
0 & 18
\end{smallmatrix}
\right) . $ 

\subsubsection{Specialization of case \#19 and \#20 for $k=10$}
In the following paragraph, we prove that $H_{\#20}(10)$ is an elliptic fibration of a K3 surface 
with trancendental lattice [4 0  18], thus achieving the proof of Theorem 8.2.
 We give also a direct proof of the fact that $H_{\#19}(10)$ is another elliptic fibration of the same surface.
 
\subsubsection{case \#19}
We have the Weierstrass equation
\[
H_{\#19}\left(  10\right)  :Y^{2}-30tYX+t^{2}\left(  t-27\right)  \left(
27t-1\right)  Y=X^{3}%
\]
with the two sections%
\[
P^{\prime}=(-100t^{2}-3-3t(t+10)(1+t^{2}),\frac{1}{72}i\sqrt{3}(6t^{2}%
+10i\sqrt{3}t+30t+3i\sqrt{3}+3)^{3})%
\]
image of the point $\left(  -t^{2},-t^{2}\right)  $ on $E_{\#19}\left(  10\right)  ,$
and
\[
P^{\prime\prime}=\left(  -\left(  t+1\right)  \left(  27t-1\right)  ,-\left(
27t-1\right)  \left(  t+1\right)  ^{3}\right).
\]

The N\'eron-Severi lattice, with the following basis $(s_{0},F,\theta_{\infty
,i}$ $1\leq i\leq6$ , $\theta_{0},i,1\leq i\leq6$, $\theta_{t_{0},1}%
,\theta_{t_{0},2},$ $s_{3},\theta_{t_{1},2},P^{\prime},P^{\prime\prime})$ has
for Gram matrix 

\[
\left(
\begin{smallmatrix}
-2 & 1 & 0 & 0 & 0 & 0 & 0 & 0 & 0 & 0 & 0 & 0 & 0 & 0 & 0 & 0 & 0 & 0 & 0 &
0\\
1 & 0 & 0 & 0 & 0 & 0 & 0 & 0 & 0 & 0 & 0 & 0 & 0 & 0 & 0 & 0 & 1 & 0 & 1 &
1\\
0 & 0 & -2 & 0 & 0 & 1 & 0 & 0 & 0 & 0 & 0 & 0 & 0 & 0 & 0 & 0 & 0 & 0 & 0 &
0\\
0 & 0 & 0 & -2 & 1 & 0 & 0 & 0 & 0 & 0 & 0 & 0 & 0 & 0 & 0 & 0 & 1 & 0 & 0 &
0\\
0 & 0 & 0 & 1 & -2 & 1 & 0 & 0 & 0 & 0 & 0 & 0 & 0 & 0 & 0 & 0 & 0 & 0 & 0 &
0\\
0 & 0 & 1 & 0 & 1 & -2 & 1 & 0 & 0 & 0 & 0 & 0 & 0 & 0 & 0 & 0 & 0 & 0 & 0 &
0\\
0 & 0 & 0 & 0 & 0 & 1 & -2 & 1 & 0 & 0 & 0 & 0 & 0 & 0 & 0 & 0 & 0 & 0 & 0 &
0\\
0 & 0 & 0 & 0 & 0 & 0 & 1 & -2 & 0 & 0 & 0 & 0 & 0 & 0 & 0 & 0 & 0 & 0 & 0 &
1\\
0 & 0 & 0 & 0 & 0 & 0 & 0 & 0 & -2 & 0 & 0 & 1 & 0 & 0 & 0 & 0 & 0 & 0 & 0 &
0\\
0 & 0 & 0 & 0 & 0 & 0 & 0 & 0 & 0 & -2 & 1 & 0 & 0 & 0 & 0 & 0 & 1 & 0 & 0 &
0\\
0 & 0 & 0 & 0 & 0 & 0 & 0 & 0 & 0 & 1 & -2 & 1 & 0 & 0 & 0 & 0 & 0 & 0 & 0 &
0\\
0 & 0 & 0 & 0 & 0 & 0 & 0 & 0 & 1 & 0 & 1 & -2 & 1 & 0 & 0 & 0 & 0 & 0 & 0 &
0\\
0 & 0 & 0 & 0 & 0 & 0 & 0 & 0 & 0 & 0 & 0 & 1 & -2 & 1 & 0 & 0 & 0 & 0 & 0 &
0\\
0 & 0 & 0 & 0 & 0 & 0 & 0 & 0 & 0 & 0 & 0 & 0 & 1 & -2 & 0 & 0 & 0 & 0 & 0 &
0\\
0 & 0 & 0 & 0 & 0 & 0 & 0 & 0 & 0 & 0 & 0 & 0 & 0 & 0 & -2 & 1 & 0 & 0 & 0 &
1\\
0 & 0 & 0 & 0 & 0 & 0 & 0 & 0 & 0 & 0 & 0 & 0 & 0 & 0 & 1 & -2 & 1 & 0 & 0 &
0\\
0 & 1 & 0 & 1 & 0 & 0 & 0 & 0 & 0 & 1 & 0 & 0 & 0 & 0 & 0 & 1 & -2 & 1 & 2 &
1\\
0 & 0 & 0 & 0 & 0 & 0 & 0 & 0 & 0 & 0 & 0 & 0 & 0 & 0 & 0 & 0 & 1 & -2 & 0 &
0\\
0 & 1 & 0 & 0 & 0 & 0 & 0 & 0 & 0 & 0 & 0 & 0 & 0 & 0 & 0 & 0 & 2 & 0 & -2 &
2\\
0 & 1 & 0 & 0 & 0 & 0 & 0 & 1 & 0 & 0 & 0 & 0 & 0 & 0 & 1 & 0 & 1 & 0 & 2 & -2
\end{smallmatrix}
\right).
\]%
According to Shimada's lemma \ref{lem:gram}, $G_{NS}=\mathbb{Z}/2\mathbb{Z}%
\oplus\mathbb{Z}/36$ is generated by the vectors $L_1$ and $L_2$ satisfying 
$q_{L_1}=\frac{-1}{2},q_{L_2}=\frac{13}{36}$ and $b(L_1,L_2)=\frac{-1}{2}.$

Moreover the following generators of the discriminant group of the lattice with Gram matrix
$M_{18}=\left(\begin{smallmatrix}
4&0\\
0&18
\end{smallmatrix} \right)$ namely $f_1=(0,\frac{1}{2})$,$f_2=(\frac{1}{4},\frac{5}{18})$ verify $q_{f_1}=\frac{1}{2}$, $q_{f_2}=-\frac{13}{36}$ and 
$b(f_1,f_2)=\frac{1}{2}.$

So the transcendental lattice is $M_{18}=\left(
\begin{array}
[c]{cc}%
4 & 0\\
0 & 18
\end{array}
\right).  $

\begin{remark}
An alternative proof: From the equation $H_{\#19}\left(  10\right)  $ and with
the parameter $m=\frac{Y}{\left(  t-27\right)  ^{2}}$ we obtain another
elliptic fibration defined by the following cubic equation in $W$ and $t$ with
$X=W\left(  t-27\right)  $%
\[
W^{3}+30Wtm-m\left(  t^{2}\left(  27t-1\right)  +m\left(  t-27\right)
\right)  =0
\]
and the rational point%
\[
W=676\frac{m\left(  m-1\right)  }{m^{2}-648m+27},t=\frac{27m^{2}-648+1}%
{m^{2}-648m+27}.%
\]
This fibration has a Weierstrass equation of the form
\[
y^{2}=x^{3}-3ax+\left(  m+\frac{1}{m}-2b\right)
\]
with $a=38425/9,b=-7521598/27$. $\ $\ So the Kummer surface associated is the
product of two elliptic curves with $J,J^{\prime}=2950584125/27\pm
1204567000/27\sqrt{6}$ that is $j,j^{\prime}=188837384000\pm77092288000\sqrt
{6}.$ So the fibration $H_{\#19}\left(  10\right)  $ corresponds to a surface
with transcendental lattice of Gram matrix $\left(
\begin{array}
[c]{cc}%
4 & 0\\
0 & 18
\end{array}
\right)  .$
\end{remark}

\subsubsection{case \#20}
A Weierstrass equation for $H_{\#20}\left(  10\right)  $ has a Weierstrass
equation
\begin{align*}
&  Y^{2}+3\left(  t^{2}-22\right)  YX+\left(  t^{2}-25\right)  ^{2}\left(
t^{2}-16\right)  Y=X^{3}
\end{align*}
with singular fibers $I_{4}\left(  \infty\right)  ,2I_{6}\left(  \pm5\right)
,2I_{3}\left(  \pm4\right)  ,2I_{1}\left(  t^{2}-24\right)  .$ We have a
$2$-torsion section $s_2=\left(  -\left(  t^{2}-25\right)  ^{2},\left(
t^{2}-25\right)  ^{3}\right)  $ and a $6$-torsion section 

\noindent
$s_{6}=\left(  -\left(
t^{2}-25\right)  \left(  t^{2}-16\right)  ,-\left(  t^{2}-25\right)  \left(
t^{2}-16\right)  ^{2}\right)  $. 

\noindent
The section
$P_{w}=\left(  4\left(  t+5\right)  ^{2}\left(  t-4\right)  ,-\left(
t+5\right)  ^{4}\left(  t-4\right)  ^{2}\right)  $ is of infinite order and is
a generator of the Mordell Weil lattice. \ Moreover we have $\theta_{\pm
5,1}.s_{6}=1,\theta_{\infty,2}.s_{6}=1,\theta_{\pm3,1}.s_{6}=1$. So the
following divisor is $6$ divisible%
\[
\sum_{i=1}^{5}\left(  6-i\right)  \theta_{\pm5,i}+3\theta_{\infty,1}%
+3\theta_{\infty,3}+4\theta_{\pm,4,1}+2\theta_{\pm4,2}\approx 6s_{6}%
\]
So we can replace $\theta_{5,5}$ by $s_{6}.$

Moreover we can compute $s_{6}.P_{w}=1$ (for $t=-3).$ We have also
$\theta_{5,0}.P_{w}=1,\theta_{-5,4}.P_{w}=1$ and $\theta_{\infty,0}%
.P_{w}=1,\theta_{4,1}.P_{w}=1,\theta_{-4,0}.P_{w}=1.$ All these computations
give the Gram matrix of the N\'{e}ron-Severi lattice of discriminant $-72$

\[
\left (\begin {smallmatrix}
-2&1&0&0&0&0&0&0&0&0&0&0&0&0&0&0&0&0&0&0\\
1&0&0&0&0&0&0&0&0&0&0&1&0&0&0&0&0&0&0
&1\\0&0&-2&1&0&0&0&0&0&0&0&1&0&0&0&0&0&0&0&0
\\0&0&1&-2&1&0&0&0&0&0&0&0&0&0&0&0&0&0&0&0
\\0&0&0&1&-2&1&0&0&0&0&0&0&0&0&0&0&0&0&0&0
\\0&0&0&0&1&-2&1&0&0&0&0&0&0&0&0&0&0&0&0&0
\\0&0&0&0&0&1&-2&0&0&0&0&0&0&0&0&0&0&0&0&0
\\0&0&0&0&0&0&0&-2&1&0&0&1&0&0&0&0&0&0&0&0
\\0&0&0&0&0&0&0&1&-2&1&0&0&0&0&0&0&0&0&0&0
\\0&0&0&0&0&0&0&0&1&-2&1&0&0&0&0&0&0&0&0&0
\\0&0&0&0&0&0&0&0&0&1&-2&0&0&0&0&0&0&0&0&1
\\0&1&1&0&0&0&0&1&0&0&0&-2&0&1&0&1&0&1&0&1
\\0&0&0&0&0&0&0&0&0&0&0&0&-2&1&0&0&0&0&0&0
\\0&0&0&0&0&0&0&0&0&0&0&1&1&-2&1&0&0&0&0&0
\\0&0&0&0&0&0&0&0&0&0&0&0&0&1&-2&0&0&0&0&0
\\0&0&0&0&0&0&0&0&0&0&0&1&0&0&0&-2&1&0&0&1
\\0&0&0&0&0&0&0&0&0&0&0&0&0&0&0&1&-2&0&0&0
\\0&0&0&0&0&0&0&0&0&0&0&1&0&0&0&0&0&-2&1&0
\\0&0&0&0&0&0&0&0&0&0&0&0&0&0&0&0&0&1&-2&0
\\0&1&0&0&0&0&0&0&0&0&1&1&0&0&0&1&0&0&0&-2
\end {smallmatrix}\right ).
\]
According to Shimada's lemma \ref{lem:gram}, $G_{NS}=\mathbb{Z}/2\mathbb{Z}%
\oplus\mathbb{Z}/36$ is generated by the vectors $L_1$ and $L_2$ satisfying 
$q_{L_1}=\frac{-1}{2},q_{L_2}=\frac{-35}{36}$ and $b(L_1,L_2)=\frac{-1}{2}.$

Moreover the following generators of the discriminant group of the lattice with Gram matrix
$M_{18}=\left(\begin{smallmatrix}
4&0\\
0&18
\end{smallmatrix} \right)$, namely $f_1=(0,\frac{1}{2})$,$f_2=(\frac{1}{4},\frac{7}{18})$ verify $q_{f_1}=\frac{1}{2}$, $q_{f_2}=\frac{35}{36}$ and 
$b(f_1,f_2)=\frac{1}{2}.$

So the Gram matrix of the transcendental lattice is $\left(
\begin{array}
[c]{cc}%
4 & 0\\
0 & 18
\end{array}
\right)  .$

\end{document}